\documentclass[a4paper,reqno]{amsart}

\usepackage{amsmath, amssymb, amsthm, eucal}
\usepackage{graphicx}
\usepackage[shortlabels]{enumitem}

\usepackage[foot]{amsaddr}
\usepackage{hyperref}
\usepackage{verbatim}
\usepackage{pdfsync}

\usepackage{comment}

\usepackage{tikz}
\usepackage{scalefnt}

\usepackage[textheight=630pt,
  textwidth=468pt,
  centering]{geometry}

\usepackage{color}

\newtheorem{theorem}{Theorem}
\newtheorem{corollary}[theorem]{Corollary}
\newtheorem{lemma}[theorem]{Lemma}
\newtheorem{proposition}[theorem]{Proposition}
\newtheorem{conjecture}[theorem]{Conjecture}
\newtheorem{remark}[theorem]{Remark}

\theoremstyle{definition}
\newtheorem{definition}[theorem]{Definition}

\newcommand{\norm}[1]{\Vert #1 \Vert}

\newcommand{\kk}{\ell}

\newcommand{\NN}{\ensuremath{\mathbb{N}}}

\newcommand{\RR}{\mathbb{R}}

\newcommand{\CC}{\mathbb{C}}

\DeclareMathOperator{\tr}{tr}

\newcommand{\id}{I}

\newcommand{\Id}{I}

\DeclareMathOperator{\diag}{diag}

\DeclareMathOperator{\rank}{rank}
\DeclareMathOperator{\Hess}{Hess}

\newcommand{\CMAN}{\mathcal{M}}

\begin{document}

\title[Learning deep linear networks]
{Learning deep linear neural networks: Riemannian gradient flows and convergence
to global minimizers} 

\date{October 15, 2020}
\author[]{Bubacarr Bah$^1$}
\address{$^1$African Institute for Mathematical Sciences (AIMS) South Africa, \& Stellenbosch University, 6 Melrose Road, Muizenberg, Cape Town 7945, South Africa}
\curraddr{}
\email{bubacarr@aims.ac.za}
\thanks{}

\author{Holger Rauhut$^2$, Ulrich Terstiege$^2$}
\address{$^2$Chair for Mathematics of Information Processing, RWTH Aachen University, Pontdriesch 10, 52062 Aachen, Germany}
\email{rauhut@mathc.rwth-aachen.de}
\thanks{}

\email{terstiege@mathc.rwth-aachen.de}
\thanks{}

\author{Michael Westdickenberg$^3$}
\address{$^3$Institute for Mathematics, RWTH Aachen University, Templergraben 55, 52062 Aachen, Germany}
\curraddr{}
\email{mwest@instmath.rwth-aachen.de}
\thanks{}

\begin{abstract}
We study the convergence of gradient flows related to learning deep linear
neural networks (where the activation function is the identity map) from data.
In this case, the composition of the network layers amounts to simply
multiplying the weight matrices of all layers together, resulting in an
overparameterized problem. The gradient flow with respect to these
factors can be re-interpreted as a Riemannian gradient flow on the manifold of
rank-$r$ matrices endowed with a suitable Riemannian metric. We show that the
flow always converges to a critical point of the underlying functional.
Moreover, we establish that, for almost all initializations, the flow converges
to a global minimum on the manifold of rank $k$ matrices for some $k\leq r$.
\end{abstract}

\maketitle


\medskip

%

\section{Introduction}

Deep learning \cite{Goodfellow-et-al-2016} forms the basis of remarkable
breakthroughs in many areas of machine learning. Nevertheless, its inner
workings are not yet well-understood and mathematical theory of deep learning is
still in its infancy. Training a neural networks amounts to solving a suitable
optimization problem, where one tries to minimize the discrepancy between the
predictions of the model and the data. One important open question concerns the
convergence of commonly used gradient descent and stochastic gradient descent
algorithms to the (global) minimizers of the corresponding objective
functionals. Understanding this problem for general nonlinear deep neural
networks seems to be very involved. In this paper, we study the convergence
properties of gradient flows for learning deep \textit{linear} neural networks
from data. While the class of linear neural networks may be not
be rich enough for many machine learning tasks, it is nevertheless instructive
and still a non-trivial task to understand the convergence properties of
gradient descent algorithms. Linearity here means that the activation functions
in each layer are just the identity map, so that the weight matrices of all
layers are multiplied together. This results in an overparameterized problem.

Our analysis builds on previous works on optimization aspects for learning linear networks 
\cite{yahemo94, kawag16, AroraCohenHazan2018,
arora2018convergence, chlico18, TragerKohnBruna19}. In \cite{arora2018convergence} the gradient
flow for weight matrices of all network layers is analyzed and an equation for
the flow of their product is derived. The article \cite{arora2018convergence}
then establishes local convergence for initial points close enough to the
(global) minimum. In \cite{chlico18} it is shown that under suitable conditions
the flow converges to a critical point for any initial point. We contribute to
this line of work in the following ways:
\begin{itemize}
\item 
We show that in the balanced case (see definition \ref{D:BALA}) the evolution of the
product of all network layer matrices can be re-interpreted as a Riemannian
gradient flow on the manifold of matrices of some fixed rank (see Corollary~\ref{cor:Riemann}).
 This is remarkable because it is
shown in \cite{arora2018convergence} that the flow of this product cannot be
interpreted as a standard gradient flow with respect to some functional. Our
result is possible because we use a non-trivial Riemannian metric. 
\item We show in Theorem~\ref{globconv} that the flow always converges to a critical point of the loss
functional $L^N$, see \eqref{E:LN}. This results applies under significantly more general assumptions
than the mentioned result of \cite{chlico18}.
\item 
We show that the flow converges to the global optimum of $L^1$, see \eqref{E:LONE}, restricted to the manifold of rank $k$ matrices for almost all initializations
(Theorem~\ref{thm:main-general}), where the rank may be anything between $0$ and $r$ (the smallest of the involved matrix dimensions). 
In the case of two layers, 
we  show in the same theorem that for almost all initial conditions, the flow converges to a global optimum of $L^2$, see \eqref{E:LN}. 
Our result in the case of two layers again applies under significantly more general conditions than a similar result in \cite{chlico18}.
 For the proof, we extend an abstract result in
\cite{Lee19} that shows that strict saddle points of the functional are avoided
almost surely. Moreover, we give an analysis of the critical points and saddle points 
of $L^1$ and $L^N$, which generalizes and refines results of \cite{kawag16,TragerKohnBruna19}.
\end{itemize} 

We believe that our results shed new light on global convergence of gradient
flows (and thereby on gradient descent algorithms) for learning neural network.
We expect that the insights will be useful for extending them to learning
\textit{nonlinear} neural networks.

\subsection*{Structure} This article is structured as follows.
Section~\ref{sec-lin-gradflows} describes the setup of gradient flows for
learning linear neural networks and collects some basic results. 
Section~\ref{sec:convergence} shows convergence of the flow to a critical point
of the functional. 
Section~\ref{sec: Riemannian_gradient_flows} provides the interpretation as
Riemannian gradient flow on the manifold of rank-$r$ matrices.
For the special case of a linear autoencoder with two coupled
layers and balanced initial points, Section~\ref{sec:autoencoders} shows
convergence of the flow to a global optimum for almost all starting points by
building on \cite{yahemo94}. Section~\ref{sec:saddle-points} extends this result
to general linear networks with an arbitrary number of (non-coupled) layers by
first extending an abstract result in \cite{Lee19} that first order methods
avoid strict saddle points almost surely to gradient flows and then analyzing
the strict saddle point property for our functional under consideration.
Section~\ref{sec:numerics} illustrates our findings with numerical experiments.
Appendices~\ref{appendix-metric-explicit} and \ref{appendix-Prop-metric-C1}
contain detailed proofs of Propositions~\ref{metric-explicit} and
\ref{C1metric}; while Appendices~\ref{appendix:flows} and
\ref{appendix:nonsymmetric} collect additional results on flows on manifolds and
on the autoencoder case with two (non-coupled) layers, respectively.

\subsection*{Acknowledgement}

B.B., H.R. and U.T. acknowledge funding through the DAAD project 
\textit{Understanding stochastic gradient descent in deep learning} (project 
number 57417829). B.B. acknowledges funding by BMBF through the 
Alexander-von-Humboldt Foundation.

\section{Gradient flows for learning linear networks}
\label{sec-lin-gradflows}

Suppose we are given data points $x_1,\hdots, x_m\in \RR^{d_x}$  and label
points $y_1,\hdots, y_m\in \RR^{d_y}$. The learning task consists in finding a map $f$ such that $f(x_j) \approx y_j$.
In deep learning, candidate maps are given by deep neural networks of the form 
\[
f(x) = f_{W_1,\hdots,W_N,b_1,\hdots,b_N}(x) = g_{N} \circ g_{N-1} \circ \cdots \circ g_1(x),
\]
where each layer is of the form $g_j(z) = \sigma(W_j z + b_j)$ with matrices $W_j$ and vectors $b_j$
and an activation function $\sigma : \RR \to \RR$ that acts componentwise.
The parameters $W_1,\hdots,W_N, b_1,\hdots,b_N$ are commonly learned from the data via empirical risk minimization.
Given a suitable loss function $\ell : \RR^{d_y} \times \RR^{d_y} \to \RR$, one considers the optimization
problem
\[
\min_{W_1,\hdots,W_N,b_1,\hdots,b_N} \sum_{j=1}^m \ell( f_{W_1,\hdots,W_N,b_1,\hdots,b_N}(x_j), y_j).
\]
In this article, we are interested in understanding the convergence behavior of the gradient flow (as simplification of gradient descent)
for the minimization of this functional.
Since providing such understanding for the general case seems to be hard, we concentrate
on the special case of linear networks (with $b_j = 0$ for all $j$) and the $\ell_2$-loss $\ell(z,y) = \|y-z\|_2^2/2$ 
in this article, i.e., the network takes the form
\[
f (x) = W_N \cdot W_{N-1} \cdots W_1 x, \quad \mbox{ for } N \geq 2,
\]
where $W_j\in \RR^{d_j\times d_{j-1}}$ for $d_0 = d_x$, $d_N = d_y$ and $d_1,\hdots, d_{N-1}\in \NN$. Clearly, 
$f(x) = Wx$ with the factorization
\begin{equation}\label{E:PROD}
	W = W_N\cdots W_1,
\end{equation}
which can be viewed as an overparameterization of the matrix $W$. Note that the factorization imposes a rank constraint as the rank of 
$W$ is at most $r = \min\{d_0,d_1,\hdots,d_N\}$.
The $\ell_2$-loss leads to the functional
\begin{equation}\label{E:LN}
L^N(W_1,\hdots,W_N) = \frac 1 2 \sum_{j=1}^m \| y_j - W_N \cdots W_1 x_j\|_2^2 = \frac 1 2 \norm{Y-W_N\cdots W_1 X}_F^2
\end{equation}
where $X\in \RR^{d_x\times m}$ is the
matrix with columns $x_1,\hdots, x_m $ and $Y\in \RR^{d_y\times m}$ the
matrix with columns $y_1,\hdots, y_m$.
Here $\|\cdot\|_F$ denotes the Frobenius norm induced by the inner
product $\langle A, B\rangle_F := \tr(AB^T)$.

Empirical risk minimization is the optimization problem
\begin{equation}\label{eqNNMinimization}
\underset{W_1,\hdots, W_N}\min L^N(W_1,\hdots,W_N),\quad\mbox{where}\;W_j\in \RR^{d_j\times d_{j-1}},\ j=1,\hdots, N.
\end{equation}
For $W\in \RR^{d_y\times d_x }$, we further introduce the functional 
\begin{equation}\label{E:LONE}
	L^1(W) := \frac 1 2\norm{Y-WX}_F^2.
\end{equation} 
Since the rank of $W = W_N \cdots W_1$ is at most $r = \min\{d_0,d_1,\hdots,d_N\}$, minimization of $L^N$ is closely related
to the minimization of $L^1$ restricted to the set of matrices of rank at most $r$, but the optimization of $L^N$ does
not require to formulate this constraint explicitly. However, $L^N$ is not jointly convex in $W_1,\hdots,W_N$
so that understanding the behavior of corresponding optimization algorithms is not trivial.

The case of an autoencoder \cite[Chapter 14]{Goodfellow-et-al-2016}, studied in
detail below, refers to the situation where $Y = X$. Here one tries to find for
$W$ a projection onto a subspace of dimension $r$ that best approximates the data,
i.e., $Wx_\ell \approx x_\ell$ for $\ell=1,\hdots,m$. This task is relevant for
unsupervised learning and only the rank deficient case, where $r :=
\min_{i = 0,\hdots,N} d_i < m$ is of interest then, as otherwise one could simply set
$W = I_{d_x}$ and there would be nothing to learn.

The gradient of $L^1$ is given as
$$
\nabla_W L^1(W)=WXX^T-YX^T.
$$
For given initial values $W_j(0), \ j\in \{1,\hdots, N\}$,  we consider the system of  gradient flows 
\begin{equation}\label{gradflow}
\dot{W_j}=-\nabla_{W_j} L^N(W_1,\hdots, W_N).
\end{equation}
Our aim is to investigate when this system converges to an optimal solution,
i.e.,  one that is minimizing our optimization problem (\ref{eqNNMinimization}).
For $W=W_N\cdots W_1$ we also want to understand the 
behavior of $W(t)$ as $t$ tends to infinity. Clearly, the gradient flow is a
continuous version of gradient descent algorithms used in practice and has the
advantage that its analysis does not require discussing step sizes etc. We
postpone the extension of our results to gradient descent algorithms to later
contributions.

\begin{definition}\label{D:BALA}
Borrowing notation from  \cite{arora2018convergence}, for $W_j\in
\RR^{d_j\times d_{j-1}},\ j=1,\hdots, N$, we say that $W_1,\hdots, W_N$ are
{\it{$0$-balanced}} or simply {\it{balanced} } if
$$ W_{j+1}^TW_{j+1}=W_{j}W_{j}^T \text{ for } j=1,\hdots,N-1.$$
We say that the flow (\ref{gradflow}) has balanced initial conditions if $W_1(0),\hdots, W_N(0)$ are balanced.
\end{definition}

The following lemma summarizes basic properties of the flow which are well 
known; see \cite{AroraCohenHazan2018, arora2018convergence, chlico18}.

\begin{lemma}\label{L:PREL}
With the notation above, the following holds:
\begin{enumerate}
\item For  $j\in \{1,\hdots, N\} $,
$$
	\nabla_{W_j}L^N(W_1,\hdots,W_N)=W_{j+1}^T\cdots W_N^T\nabla_W L^1(W_N\cdots 
	W_1)W_1^T\cdots W_{j-1}^T.
$$

\item Assume the $W_j(t)$ satisfy (\ref{gradflow}). Then $W=W_N\cdots W_1$ 
satisfies
\begin{equation}\label{Wdot}
	\frac{dW(t)}{dt}= -\sum_{j=1}^{N} W_N\cdots W_{j+1}W_{j+1}^T\cdots W_N^T  
	\nabla_W L^1(W)W_1^T\cdots W_{j-1}^T W_{j-1}\cdots W_1.
\end{equation}

\item For all $j=1,\hdots,N-1$ and all $t\geq 0$ we have that
$$
	\frac{d}{dt} \bigg( W_{j+1}^T(t)W_{j+1}(t) \bigg) 
		= \frac{d}{dt} \bigg( W_{j}(t)W_{j}^T(t) \bigg).
$$
In particular, the differences 
$$
	W_{j+1}^T(t)W_{j+1}(t)-W_{j}(t)W_{j}^T(t), 
	\quad j=1,\hdots,N-1, 
$$ 
and the  differences 
$$
	\norm{W_j(t)}^2_F-\norm{W_i(t)}^2_F,
	\quad i,j=1,\hdots,N,
$$ 
are all constant in time.
 
\item If $W_1(0),\hdots, W_N(0)$ are \emph{balanced}, then
$$
	W_{j+1}^T(t)W_{j+1}(t) = W_{j}(t)W_{j}^T(t)
$$
for all $j\in\{1,\hdots,N-1\}$ and $t\geq 0$, 
and
\begin{equation}\label{Wdotbalanced}
	R(t) := \frac{dW(t)}{dt}
		+ \sum_{j=1}^{N} (W(t)W(t)^T)^{\frac{N-j}{N}}\nabla_W 
		L^1(W)(W(t)^TW(t))^{\frac{j-1}{N}} = 0.
\end{equation}
\end{enumerate}
\end{lemma}

Here and the sequel, by the $p$-th root of a symmetric and  positive semidefinite matrix we mean the principal $p$-th root, i.e. the $p$-th root is symmetric and  positive semidefinite again.
A concrete reference for the statements of the lemma is  \cite[Theorem 1]{AroraCohenHazan2018} together with  its proof.
For point (3), see also  \cite[Lemma 1]{chlico18}.

\begin{definition}
For $W,Z\in  \RR^{d_y\times d_x}$ and $N\geq 2$ let
\begin{equation}\label{A_W}
	\mathcal{A}_W(Z)=\sum_{j=1}^{N} (WW^T)^{\frac{N-j}{N}}\cdot Z\cdot 
		(W^TW)^{\frac{j-1}{N}}.
\end{equation}
\end{definition}

Thus, if the $W_j(0)$ are balanced (see Definition~\ref{D:BALA}), then
\begin{equation}\label{gradflowprod}
	\frac{dW(t)}{dt}
		= -\mathcal{A}_{W(t)}\Big( \nabla_W L^1\big( W(t) \big) \Big).
\end{equation}
We will write this as a gradient flow with respect to a 
suitable Riemannian metric in Section~\ref{sec: Riemannian_gradient_flows}.

\section{Convergence of the gradient flow}
\label{sec:convergence}

In this section we will show that the gradient flow always converges to a
critical point of $L^N$, also called an equilibrium point in the following,
provided that $XX^T$ has full rank. We do not assume balancedness of the initial
data. A similar statement was shown in \cite[Proposition 1]{chlico18} and 
similarly as in loc.\ cit., our proof is based on Lojasiewicz's Theorem and uses the fact that for  $j=1,\hdots,N-1$ the terms $	W_{j+1}^T(t)W_{j+1}(t)-W_{j}(t)W_{j}^T(t) $ are constant (cf. Lemma \ref{L:PREL}), but the
technical exposition differs and we do not need the assumptions $d_y\leq d_x$ 
and $d_y \leq r = \min\{d_1,\hdots,d_{N-1}\}$ made in \cite{chlico18}, which, for instance, exclude the autoencoder case and imply that the set $\mathcal{M}_r$ of all admissible matrices appearing as a product $W = W_N \cdots W_1$, i.e., the variety of matrices of rank at most $r$, coincides
with the vector space $\RR^{d_x \times d_y}$. 
Let us first recall the following corollary of Lojasiewicz's Inequality; see
\cite{Absil05,lojasiewicz1982trajectoires,chlico18,Kurdyka02,Simon96}.

\begin{theorem}\label{T:LOJA}
If $f\colon \RR^n\rightarrow \RR$ is analytic and the curve $t \mapsto x(t) \in
\RR^n$, $t\in [0,\infty)$, is bounded and a solution of the gradient flow
equation $\dot x(t)=-\nabla f(x(t))$, then $x(t)$ converges to a critical point
of $f$ as $t\to\infty$.
\end{theorem}

This result, sometimes called Lojasiewicz's Theorem, follows from Theorem~2.2 in
\cite{Absil05}, for example (see also Theorem~1 in \cite{chlico18}). Indeed it
is shown in \cite{Absil05} that under our assumptions $x(t)$ converges to a
limit point $x^*$. By continuity, it follows that also the time derivative $\dot
x(t)=-\nabla f(x(t))$ converges to a limit point $z := -\nabla f(x^*)$. Then
$z=0$, i.e., $x^*$ is a critical point of $f$. Indeed, if $z$ had a component
$z_k\neq 0$ then for $t_0$ large enough we would have $|\dot x_k(t)-z_k|\leq
\frac{|z_k|}{2}$ for all $t\geq t_0$ and hence for $t_2\geq t_1\geq t_0$ we
would have $|x_k(t_2)-x_k(t_1)|=|\int_{t_1}^{t_2}\dot x_k(t)dt|\geq
(t_2-t_1)\frac{|z_k|}{2}$, contradicting the convergence of $x_k$.

\begin{theorem}\label{globconv}
Assume $XX^T$ has full rank. Then the flows $W_i(t)$ defined by \eqref{gradflow}
and $W(t)$ given by \eqref{Wdot} are defined and bounded for all $t\geq 0$ and 
$(W_1,\hdots, W_N)$ converges  to a critical  point of $L^N$ as $t\to\infty$.
\end{theorem}

\begin{proof} 
Note that the right-hand sides of \eqref{gradflow} and \eqref{Wdot} are
continuous functions so existence of solutions locally in time follows from the
Cauchy-Peano theorem. In order to show that the solutions exist for all times
and to be able to apply Lojasiewicz's Theorem, we want to show that the
$\norm{W_i(t)}_F$ are bounded. We will first show that the flow  $W(t)$ given by
\eqref{Wdot} remains bounded for all $t$. We  observe that for all $t\geq 0$ for
which $W(t)$ is defined we have $L^1(W(t))\leq L^1(W(0))$. To see this, note
that
 \begin{align*}
\frac{d}{dt}L^1(W(t))&=\frac{d}{dt}L^N(W_1(t),\hdots,W_N(t))=\sum_{i=1}^N D_{W_i}L^N((W_1(t),\hdots,W_N(t))\dot{W_i}(t)\\&=-\sum _{i=1}^N \norm{\nabla_{W_i}L^N((W_1(t),\hdots,W_N(t))}_F^2\leq 0.
\end{align*}
Here the notation $D_{W_i}$ denotes the directional derivative  w.r.t. $W_i$.
Hence, for any $t\geq 0$ we have
\begin{align*}
\norm{W(t)}_F&= \norm{W(t)XX^T(XX^T)^{-1}}_F\leq \norm{W(t)X}_F\norm{X^T(XX^T)^{-1}}_F=\norm{W(t)X-Y+Y}_F\norm{X^T(XX^T)^{-1}}_F \displaybreak[2]\\
&\leq \left(\norm{W(t)X-Y}_F+\norm{Y}_F\right)\norm{X^T(XX^T)^{-1}}_F=
\left(\sqrt{2L^1(W(t))}+\norm{Y}_F\right)\norm{X^T(XX^T)^{-1}}_F\displaybreak[2]\\
&\leq \left(\sqrt{2L^1(W(0))}+\norm{Y}_F\right)\norm{X^T(XX^T)^{-1}}_F.
\end{align*}
In particular, $\norm{W(t)}_F$ is bounded. Recall that the Frobenius norm is submultiplicative.

Next, 
in order to show the boundedness of the $\norm{W_i(t)}_F$, we show the following claim:
 For any $i\in\{1,\hdots, N\}$, we have
\begin{equation} \label{normestimate}
\norm{W_i(t)}_F\leq C_i\norm{W(t)}^{1/N}_F+\widetilde C_i,
\end{equation} 
for all $t\geq 0$ (for which the $W_i(t)$ and hence also $W(t)$ are defined). 
Here  $C_i$ and $\widetilde C_i$ are suitable positive constants 
depending only on the initial conditions.

Before we prove the claim, we introduce the following notation.
\begin{definition}
Suppose we are given a set of (real valued) matrices $\{X_i, i\in I\}$, where 
$I$ is a finite set. A polynomial $P$ in the matrices $X_i$, $i\in I$, with matrix 
coefficients is a (finite) sum of terms of the form 
\begin{equation}\label{summand}
A_1X_{i_1}A_2X_{i_2}\cdots A_n X_{i_n}A_{n+1}.
\end{equation}
The $A_j$ are the matrix coefficients of the monomial (\ref{summand}) (where 
the dimensions of the $A_j$ have to be such that the product (\ref{summand}) as 
well as the sum of all the terms of the form (\ref{summand}) in the polynomial 
$P$ are well defined). The degree of the polynomial $P$ is the maximal value of 
$n$ in the summands of the above form (\ref{summand}) defining $P$ (where $n=0$ 
is also allowed).
\end{definition}

 In the 
following, the constants are allowed to depend on the dimensions $d_i$ and the 
initial matrices $W_i(0)$. We will suppress the argument $t$.

To prove the claim, we observe that 
$$
WW^T=W_N\cdots W_1W_1^T\cdots W_N^T.
$$
Replacing $W_1W_1^T$ by $W_2^TW_2+A_{12}$, where $A_{12}$ is a constant matrix 
(see Lemma~\ref{L:PREL} (3)), we obtain 
$$
WW^T=W_N\cdots W_3W_2W_2^T W_2W_2^TW_3^T\cdots W_N^T+W_N\cdots 
W_2A_{12}W_2^T\cdots W_N^T.
$$
We now replace $W_2W_2^T$ by $W_3^TW_3+A_{23}$ and, proceeding in this manner,
we finally arrive at
\begin{equation}\label{WNWpP}
WW^T=(W_NW_N^T)^N+P(W_2,\hdots,W_{N},W_2^T,\hdots,W_{N}^T),
\end{equation}
where $P(W_2,\hdots,W_{N},W_2^T,\hdots,W_{N}^T)$ is a polynomial in
$W_2,\hdots,W_{N},W_2^T,\hdots,W_{N}^T$  (with  matrix coefficients) whose
degree is at most $2N-2$.

In the following, we denote by $\sigma_N$ the maximal singular value of $W_N$. 
Thus  
\begin{equation}\label{WNvsW}
\sigma_N^{2N}\leq \norm{(W_NW_N^T)^N}_F\leq 
\norm{WW^T}_F+\norm{P(W_2,\hdots,W_{N},W_2^T,\hdots,W_{N}^T)}_F.
\end{equation}

Since $\norm{W_N}_F^2$ and  $\norm{W_i}_F^2$ differ only by a 
constant (depending on $i$), there are  suitable constants $a_i$ and $b_i$ such 
that $\norm{W_i}_F\leq a_i \sigma_N +b_i$ for all $i\in\{1,\hdots, N\}$. It 
follows that 
$$
\norm{P(W_2,\hdots,W_{N},W_2^T,\hdots,W_{N}^T)}_F\leq  P_N(\sigma_N),
$$ 
where $P_N$ is a polynomial in one variable of degree at most $2N-2$. Since the
degree of $P_N$ is strictly smaller than $2N$, there exists a constant $C$,
which depends on the coefficients of $P_N$, such that $|P_N(x)|\leq \frac 12
x^{2N}+C$ for all $x\geq 0$. Hence we obtain from (\ref{WNvsW}) 
\begin{equation}\label{sigmaNest}
\sigma_N^{2N}\leq B_N\norm{WW^T}_F+\widetilde B_N,
\end{equation} 
and therefore also
\begin{equation}\label{sigmaNest2}
\sigma_N\leq B^{\prime}_N\norm{W}^{1/N}_F+\widetilde B^{\prime}_N,
\end{equation} 
for suitable positive constants $B_N, \widetilde B_N, B^{\prime}_N, \widetilde 
B^{\prime}_N$ (we can choose $B_N=2$ by the discussion above). Since $\norm{W_i}_F\leq a_i \sigma_N +b_i$, estimate  
(\ref{normestimate}) for $\norm{W_i}_F$	follows.

The fact that all the $\norm{W_i}_F$ are bounded now follows from the fact that $\norm{W}_F$ is bounded as shown above together with
estimate (\ref{normestimate}). This ensures the existence of  solutions $W_i(t)$
(and hence $W(t)$) for all $t\geq 0$.
The convergence of $(W_1,\hdots, W_N)$ to an 
equilibrium point (i.e., a critical point of $L^N$) now follows from
Lojasiewicz's Theorem~\ref{T:LOJA}.
\end{proof}

\section{Riemannian gradient flows}\label{sec: Riemannian_gradient_flows}

Recall that in order to define a gradient flow, it is necessary to also specify
the local geometry of the space. More precisely, suppose that a $C^2$ 
manifold $\CMAN$ is given, on which a $C^2$-function $x \mapsto E(x) \in
\mathbb{R}$ is defined for all $x \in \CMAN$. Then the differential $dE(x)$ of
$E$ at the point $x$ is a \emph{co-tangent} vector, i.e., a linear map from the
tangent space $T_x\CMAN$ to $\mathbb{R}$. On the other hand, the derivative
along any curve $t \mapsto \gamma(t) \in \CMAN$ is a \emph{tangent} vector. If
now $g_x$ denotes a Riemannian metric on $\CMAN$ at $x$, then it is possible to
associate to the differential $dE(x)$ a unique tangent vector $\nabla E(x)$,
called the \emph{gradient} of $E$ at $x$, that satisfies
\[
	dE(x)v =: g_x(\nabla E(x), v)
	\quad\text{for all tangent vectors $v\in T_x\CMAN$.}
\]
It is the tangent vector $\nabla E(x)$ that enters in the definition of 
gradient flow $\dot{\gamma}(t) = -\nabla E(\gamma(t))$.

In this section, we are interested in minimizing the functional $L^N$ introduced
in \eqref{E:LN} over the family of all matrices $W_1,\ldots, W_N$. This can be
accomplished by considering the long-time limit of the gradient flow of $L^N$.
Alternatively, we observe that we can equivalently lump all matrices together in the product $W := W_N \cdots
W_1$ and minimize the functional $L^1$ defined in \eqref{E:LONE} over the set of all matrices $W$ having this product form.

We consider the manifold $\mathcal{M}_k$ of  real $d_y\times d_x$ matrices of
rank  $k \leq \min\{d_x,d_y\}$. We regard $\mathcal{M}_k$ as a submanifold of
the manifold of all real $d_y\times d_x$ matrices, from which we inherit the
structure of a differentiable manifold for $\mathcal{M}_k$. We denote by
$T_W(\mathcal{M}_k)$ the tangential space of $\mathcal{M}_k$ at the point $W\in
\mathcal{M}_k$. We have 
\begin{equation}\label{E:TANG}
	T_W(\mathcal{M}_k) = \{WA+BW \colon A\in \RR^{d_x\times d_x},
		B\in \RR^{d_y\times d_y} \};
\end{equation}
see \cite[Proposition 4.1]{helmkeshayman95}. We will need the following result on the orthogonal projection onto the tangent space, which is  well-known, cf. \cite[Equation 9]{Absil15}.
Below the notions \emph{self-adjoint}, \emph{positive definite}, and
\emph{orthogonal complement} are understood with respect to the Frobenius scalar
product, which we denote by $\langle\ ,\ \rangle_F$. Recall that $\langle A,
B\rangle_F=\tr(AB^T)$.

\begin{lemma}\label{lem:proj-tangent} Let $W \in \mathcal{M}_k$ with full singular value decomposition $W = U S V^T$ and reduced singular decomposition $W = \bar{U} \Sigma \bar{V}^T$, where $U \in \RR^{d_y \times d_y}$ and $V \in \RR^{d_x \times d_x}$ are orthogonal and $\bar{U} \in \RR^{d_y \times k}$ and $\bar{V} \in \RR^{d_x \times k}$ are submatrices consisting of the first $k$ columns of $U$ and $V$, respectively. 
Let $Q_U = \bar{U} \bar{U}^T = U P_k U^T$ denote the orthogonal projection onto the range of $\bar{U}$, where 
$P_k = \diag(1,\hdots,1,0,\hdots,0)$ is the diagonal matrix with $k$ ones on the diagonal, and likewise define $Q_V = \bar{V} \bar{V}^T = V P_k V^T$.
Then the orthogonal projection $P_W : \RR^{d_y \times d_x} \to T_W(\mathcal{M}_k)$ onto the tangent space at $W$ is given by
\[
P_W(Z) = Q_U  Z + Z Q_V - Q_U Z Q_V \quad  \mbox{ for } Z \in \RR^{d_y \times d_x}.
\]
\end{lemma}
\begin{proof} For convenience, we give a proof. 
For a matrix $Z = WA + BW \in T_W(\mathcal{M}_k)$, a simple computation using $\bar{U}^T \bar{U} = I_k = \bar{V}^T \bar{V}$ gives
\begin{align*}
P_W(Z) & = \bar{U} \bar{U}^T (\bar{U} \Sigma \bar{V}^T A + B \bar{U} \Sigma \bar{V}^T) +
(\bar{U} \Sigma \bar{V}^T A + B \bar{U} \Sigma \bar{V}^T) \bar{V} \bar{V}^T
- \bar{U} \bar{U}^T (\bar{U} \Sigma \bar{V}^T A + B \bar{U} \Sigma \bar{V}^T) \bar{V} \bar{V}^T\\
& = \bar{U} \Sigma \bar{V}^T A + B \bar{U} \Sigma \bar{V}^T = Z.
\end{align*}
Moreover, for an arbitrary $Z \in \RR^{d_y \times d_x}$ we have
\begin{align*}
P_W(Z) & =  \bar{U} \bar{U}^T Z (I_{d_x} - Q_V) + Z \bar{V} \bar{V}^T =
\bar{U} \Sigma \bar{V}^T \bar{V} \Sigma^{-1} \bar{U}^T Z (I_{d_x} - Q_V) + Z \bar{V} \Sigma^{-1} \bar{U}^T  \bar{U} \Sigma \bar{V}^T\\
&= W \bar{V}\Sigma^{-1} \bar{U}^T Z (I_{d_x} - Q_V) + Z \bar{V} \Sigma^{-1} \bar{U}^T W  
\end{align*}
so that $P_W(Z) \in T_W(\mathcal{M}_k)$. We conclude that $P_W^2 = P_W$. Moreover, it is easy to verify that $\langle P_W(Z), Y \rangle_F = \langle Z, P_W(Y)\rangle_F$
for all $Z,Y \in \RR^{d_y \times d_x}$
so that $P_W$ is self-adjoint. Altogether, this proves the claim.
\end{proof}

Inspired by \cite{carlenmaas12}, we
use the operator $\mathcal{A}_{W}$ to define a Riemannian metric on
$\mathcal{M}_k$. The following lemma can be seen as an extension of (a part of) claim 1 in the proof of \cite[Theorem 1]{AroraCohenHazan2018}, where  positive semidefiniteness of ${\mathcal{A}}_W$ is established.

\begin{lemma}\label{PosDefLemma}
For any given $W\in \RR^{d_y\times d_x}$ let $k$ be the rank of $W$, so that
$W\in \mathcal{M}_k$. Let $N\geq 2$. Then the map ${\mathcal{A}}_W \colon
\RR^{d_y\times d_x} \rightarrow \RR^{d_y\times d_x}$ defined in \eqref{A_W} is a
self-adjoint endomorphism. 
Its image is $T_W(\mathcal{M}_k)$ and its kernel is (consequently) the
orthogonal complement $T_W(\mathcal{M}_k)^{\perp}$ of $T_W(\mathcal{M}_k)$. The
restriction of ${\mathcal{A}}_W$ to arguments $Z\in T_W(\mathcal{M}_k)$ defines
a self-adjoint and positive definite endomorphism
$$
	\bar{\mathcal{A}}_W\colon T_W(\mathcal{M}_k) \rightarrow T_W(\mathcal{M}_k).
$$
In particular, $\bar{\mathcal{A}}_W$ is invertible and the inverse 
$\bar{\mathcal{A}}^{-1}_W$ is self-adjoint and positive definite as well.
\end{lemma}

\begin{proof}
We split the proof into four steps.


\textbf{Step~1.} It is clear that $\mathcal{A}_W$ defines an endomorphism of  
$\RR^{d_y\times d_x}$. To see that it is self-adjoint, we calculate, for $Z_1,Z_2\in  \RR^{d_y\times
d_x}$,
\begin{align*}
	\langle \mathcal{A}_W(Z_1), Z_2\rangle_F
		&= \tr\left(\sum_{j=1}^{N} (WW^T)^{\frac{N-j}{N}} Z_1
			(W^TW)^{\frac{j-1}{N}} Z_2^T \right)
		=\tr\left(\sum_{j=1}^{N} Z_1
			(W^TW)^{\frac{j-1}{N}}Z_2^T(WW^T)^{\frac{N-j}{N}}\right)
\\
		&= \tr\left(Z_1\mathcal{A}_W(Z_2)^T\right)
		=\langle Z_1, \mathcal{A}_W(Z_2)\rangle_F.
\end{align*}
We conclude that ${\mathcal{A}}_W$ is indeed self-adjoint.


\textbf{Step~2.} Next we show that the image of ${\mathcal{A}}_W$ lies in
$T_W(\mathcal{M}_k)$; see \eqref{E:TANG}. Let $W = \bar{U} \Sigma \bar{V}^T$ be a (reduced) singular value
decomposition of $W$ in the following form: $\Sigma = \diag(\sigma_1,\hdots,\sigma_k) \in \RR^{k \times k}$ is the diagonal matrix containing
the non-zero singular values of $W$ and the columns of $\bar{U} \in \RR^{d_y \times k}$ and $\bar{V} \in \RR^{d_x\times k}$ are orthonormal,
so that $\bar{U}^T \bar{U} = I_k = \bar{V}^T \bar{V}$.
For any index $1\leq j<N$, we observe the identity
\begin{equation} 
	(WW^T)^{\frac{N-j}{N}} 
		= \bar{U} \Sigma^{2\frac{N-j}{N}} \bar{U}^T
		= \bar U \Sigma \bar V^T \bar V \Sigma^{2\frac{N-j}{N}-1} \bar U^T
		= W (\bar V \Sigma^{2\frac{N-j}{N}-1} \bar U^T).
\label{E:POL}
\end{equation}
%
Note that the second factor on the right-hand side of \eqref{E:POL} is
well-defined even though the exponent $2\frac{N-j}{N}-1$ may be negative because
the diagonal entries of $\Sigma$ are all positive. Similarly, for $1 < j \leq N$
we find
\[
	(W^TW)^{\frac{j-1}{N}} 
		= \bar V \Sigma^{2\frac{j-1}{N}} \bar V^T
		= \bar V \Sigma^{2\frac{j-1}{N}-1} (\bar U^T \bar U) \Sigma \bar V^T
		= (\bar V \Sigma^{2\frac{j-1}{N}-1} \bar U^T) W.
\]
We observe that every term in the sum \eqref{A_W} is of the form $WA$ or of the
form $BW$ for suitable $A\in  \RR^{d_x\times d_x}$ or $B\in  \RR^{d_y\times
d_y}$. Hence $\mathcal{A}_W(Z)\in T_W(\mathcal{M}_k)$ for any $Z\in
\RR^{d_y\times d_x}$. It follows that the restriction of ${\mathcal{A}}_W$ to
$T_W(\mathcal{M}_k)$, denoted by $\bar{\mathcal{A}}_W$, is a self-adjoint
endomorphism. 
To prove that
$\bar{\mathcal{A}}_W$ is injective, it therefore suffices to show that all
eigenvalues are non-zero. Since $T_W(\mathcal{M}_k)$ is a finite-dimensional
vector space, injectivity of $\bar{\mathcal{A}}_W$ then implies bijectivity.


\textbf{Step~3.} We show that $\bar{\mathcal{A}}_W$ is positive definite. For 
$Z \in T_W(\mathcal{M}_k)$, we need to establish that $\langle
{\mathcal{A}}_W(Z), Z\rangle_F>0$ if $Z\neq 0$. We will first show that for all
$ j\in\{1,\hdots,N\}$
\[
	\tr\left( (WW^T)^{\frac{N-j}{N}} Z (W^TW)^{\frac{j-1}{N}}Z^T\right)\geq 0.
\]
Let again $W = \bar U \Sigma \bar V^T$ be a (reduced) singular value decomposition of $W$ as in Step~2. If
$j=1$, then
\[
	\tr\left( (WW^T)^{\frac{N-1}{N}} ZZ^T\right)
		= \tr\left( \big( \bar U \Sigma^{2 \frac{N-1}{N}} \bar U^T \big) Z Z^T \right)
		= \tr(R_1 R_1^T) \geq 0,
\]
where $R_1 := \Sigma^{\frac{N-1}{N}} \bar U^T Z$. 
Similarly, for $j=N$ we get
\[
	\tr\left( Z (W^TW)^{\frac{N-1}{N}} Z^T\right)
		= \tr\left( Z \big( \bar V \Sigma^{2\frac{N-1}{N}} \bar V^T \big) Z^T \right)
		= \tr(R_N R_N^T) \geq 0,
\]
where $R_N := Z \bar V \Sigma^{\frac{N-1}{N}}$. 
Finally, if $1<j<N$, then
\[
	\tr\left( (WW^T)^{\frac{N-j}{N}} Z (W^TW)^{\frac{j-1}{N}}Z^T\right)
		= \tr\left( \big( \bar U \Sigma^{2\frac{N-j}{N}} \bar U^T \big) Z 
				\big( \bar V \Sigma^{2\frac{j-1}{N}} \bar V^T \big) Z^T \right)
		= \tr(R_j R_j^T) \geq 0,
\]
where $R_j := \Sigma^{\frac{N-j}{N}} \bar U^T Z \bar V \Sigma^{\frac{j-1}{N}}$. It follows that
$\langle {\mathcal{A}}_W(Z),Z\rangle_F\geq 0$ for all $Z \in
T_W(\mathcal{M}_k)$.

Suppose now that $\langle {\mathcal{A}}_W(Z),Z\rangle_F = 0$. Then $\|R_j\|_F^2
= \tr(R_j R_j^T) = 0$, thus $R_j = 0$ for every $j \in \{1, \ldots, N\}$. 
Since $\Sigma \in \RR^{k \times k}$ is invertible this implies for $j=1$ that $\bar U^T Z = 0$ and for $j=N$ that $Z \bar{V} = 0$.
By Lemma~\ref{lem:proj-tangent} we have 
\[
Z = P_W(Z) = \bar{U} \bar{U}^T Z + Z \bar{V} \bar{V}^T - \bar{U} \bar{U}^T Z \bar{V} \bar{V}^T = 0.
\]
This proves that $\bar{\mathcal{A}}_W$ is strictly positive definite, therefore
injective (bijective) as a map from $T_W(\mathcal{M}_k)$ to itself.

\textbf{Step~4.} It remains to prove that the  kernel of ${\mathcal{A}}_W$ is
the orthogonal complement of $T_W(\mathcal{M}_k)$. This follows from the general
fact that for any self-adjoint endomorphism $f$ of an inner product space, the
kernel of $f$ is the orthogonal complement of the image of the adjoint  of $f$.
\end{proof}

\begin{definition}
We introduce a Riemannian metric\footnote{Our use of the term {\em Riemannian metric} does not imply any smoothness properties. Nevertheless, we will show in Proposition \ref{C1metric}
	that our metric $g$ is of class $C^1$.} $g$ on the manifold $\mathcal{M}_k$ (for $k \leq \min\{d_x,d_y\}$) by
\begin{equation}\label{skp} 
	g_W(Z_1,Z_2) := \langle \bar{\mathcal{A}}_W^{-1}(Z_1),Z_2\rangle_F
\end{equation}
for any $W\in \mathcal{M}_k $ and for all tangent vectors $Z_1,Z_2 \in  
T_W(\mathcal{M}_k)$.
\end{definition}
By Lemma \ref{PosDefLemma}, the map $g_W$ is well defined and defines indeed a
scalar product on $ T_W(\mathcal{M}_k)$. We provide explicit expressions for this scalar product in the next result.


\begin{proposition}\label{metric-explicit}
For $N \geq 2$, the metric $g$ on $\mathcal{M}_k$ defined in \eqref{skp} satisfies
\begin{align}
g_W(Z_1,Z_2) & = \frac{\sin(\pi/N)}{\pi} \int_0^\infty \tr\left( (t \id_{d_y} + WW^T)^{-1} Z_1 (t \id_{d_x} + W^T W)^{-1} Z_2^T\right) t^{1/N} dt  \label{metric_F1}\\
& =  \frac{1}{N\Gamma(1-1/N)} \int_0^\infty \int_0^t \tr\left( e^{-s WW^T} Z_1 e^{-(t-s)W^TW} Z_2^T\right) ds \, t^{-(1/N+1)} dt \label{metric_F2}
\end{align}
for all $W \in \mathcal{M}_k$ and $Z_1, Z_2 \in T_W(\mathcal{M}_k)$,
where $\Gamma$ denotes the Gamma function.

In the case $N=2$, we additionally have 
\begin{equation}\label{metric_F3N2}
g_W(Z_1,Z_2)
		= \int_0^{\infty} \tr\Big( e^{-t(WW^T)^{\frac 1 2}} Z_1 
			e^{-t(W^TW)^{\frac 1 2}} Z_2^T \Big) \,dt.
\end{equation}
\end{proposition}
\begin{proof}
The proof is postponed to Appendix~\ref{appendix-metric-explicit}.
\end{proof}

The next result states that the Riemannian metric is continuously differentiable as a function of $W \in \mathcal{M}_k$.
\begin{proposition}\label{C1metric}
	The metric $g$ on $\mathcal{M}_k$ given by \eqref{skp} is of class $C^1$. 
\end{proposition}
\begin{proof}
The proof uses the representation \eqref{metric_F1}. The main step consists in showing that the directional derivates with respect to $W$ of the corresponding integrand remain integrable so that
Lebesgue's dominated convergence theorem can be applied to interchange integration and differentiation. The lengthy details are postponed to Appendix~\ref{appendix-Prop-metric-C1}.
\end{proof}

For any differentiable function $f\colon \RR^{d_y\times d_x} \rightarrow \RR$,
any $W\in \mathcal{M}_k \subset \RR^{d_y\times d_x}$, and any $Z\in
T_W(\mathcal{M}_k)$, we have
\[
	g_W\big( \mathcal{A}_{W}(\nabla f(W)),Z \big)
		= \Big\langle \bar{\mathcal{A}}_W^{-1}\Big( \mathcal{A}_{W}\big( \nabla 
		f(W) \big) \Big), Z \Big\rangle_F
		= \langle \nabla f(W), Z\rangle_F=D f(W)[Z],
\]
where $Df$ denotes the differential of $f$ (which can be computed from the
derivative with respect to $W$). Note here that by Lemma \ref{PosDefLemma}, the
two quantities $\bar{\mathcal{A}}_W^{-1}(\mathcal{A}_{W}(\nabla f(W)))$ and
$\nabla f(W)$ differ only by an element in $T_W(\mathcal{M}_k)^{\perp}$, which
is perpendicular to $Z$ with respect to the Frobenius norm, as noticed above.
This allows us to identify $\mathcal{A}_{W}(\nabla f(W))$ with the gradient of
$f$ with respect to the new metric $g$. We write
\begin{align}\label{ggrad}
	\mathcal{A}_{W}(\nabla f(W)) =: \nabla^g f(W).
\end{align}
In particular, we have for all $Z\in T_W(\mathcal{M}_k)$ that $g_W(\nabla^g
f(W),Z) = D f(W)[Z]$. 
Let now $k\leq \min\{d_0,\hdots,d_N\}$ and recall that, in the balanced case,
the evolution of the product $W=W_N\cdots W_1$ is given by \eqref{gradflowprod}.

We note that the solutions $W_1(t), \ldots, W_N(t)$ of the gradient flow \eqref{gradflow} of $L^N$ are unique  (given initial values), since  \eqref{gradflow} obviously satisfies a local Lipschitz condition. Therefore the tuple $W_1(t), \ldots, W_N(t)$ gives rise to a well defined product $W(t)=W_N(t)\cdots W_1(t)$ which in the balanced case solves equation \eqref{gradflowprod}. 
However, due to the appearance of  $N$-th roots in  \eqref{gradflowprod}, it is unclear at the moment whether there are also other solutions of  \eqref{gradflowprod}. 
The next proposition shows that in the balanced case (and for $XX^T$ of full rank) the solution $W(t)=W_N(t)\cdots W_1(t)$ of   \eqref{gradflowprod} stays in $\mathcal{M}_k$ for all finite times $t$ provided that $W(0) \in \mathcal{M}_k$.

\begin{proposition}\label{rank_const}
	Assume that $XX^T$ has full rank and
suppose  that $W_1(t), \ldots, W_N(t)$ are 
solutions of the gradient flow \eqref{gradflow} of $L^N$  with balanced initial values
$W_j(0)$  and 
define the product $W(t) := W_N(t) \cdots W_1(t)$. If $W(0)$ is
contained in $\mathcal{M}_k$ 
for some $k \leq \min\{d_0,\hdots,d_N\}$
 then $W(t)$  is
contained in $\mathcal{M}_k$ for all $t\geq 0$.
\end{proposition}

\begin{proof}
It follows from Theorem~\ref{globconv} that for any given $t_0\in \RR$ and
initial values $W_1(t_0),\hdots, W_N(t_0)$, a (unique) solution $W_1(t),\hdots,
W_N(t)$ of \eqref{gradflow} is defined for all $t\geq t_0$. By the Picard–Lindelöf
theorem, there also exists $\varepsilon>0$ such that the solution
$W_1(t),\hdots, W_N(t)$ is defined and unique on $(t_0-\varepsilon, 0]$, hence
on $(t_0-\varepsilon, \infty)$.

Since the initial values $W_j(0)$  are balanced, for any $t\geq 0$ the matrices
$W_1(t),\hdots, W_N(t)$ are balanced as well, cf. Lemma \ref{L:PREL}. It follows
that for any $t\geq 0$, we have $$W(t)W(t)^T=(W_N(t)W_N(t)^T)^N \text{ and }
W(t)^TW(t)=(W_1(t)^TW_1(t))^N,$$ compare also the proof of Theorem 1 in  \cite{AroraCohenHazan2018}.
 Both equations are directly verified for $N=2$
and easily follow by induction for any $N\geq 2$.
	
Let now   $P(t)=W_1(t)^TW_1(t)$ and $Q(t)=W_N(t)W_N(t)^T$. It follows that
$P(t)=(W(t)^TW(t))^{1/N}$ and $Q(t)=(W(t)W(t)^T)^{1/N}$ for all $t\in
[0,\infty)$.  Using $ \nabla_W L^1(W)=WXX^T-YX^T$ together with the explicit
form of the gradient flow \eqref{gradflow} for $W_1$ and $W_N$ given by Lemma
\ref{L:PREL}, point (1),   and substituting $P=(W^TW)^{1/N}$ and
$Q=(WW^T)^{1/N}$ in the flow equation \eqref{Wdotbalanced} for $W$, we
obtain the following system of differential equations for $P,Q,W$.
\begin{align}\label{PQW-ODE}
\begin{split}
\dot P &=-W^T(WXX^T-YX^T)-(WXX^T-YX^T)^TW,
\\
\dot Q &=-(WXX^T-YX^T)W^T-W(WXX^T-YX^T)^T,
\\
\dot W &=-  \sum_{j=1}^{N} Q^{N-j}(WXX^T-YX^T)P^{j-1}.
\end{split}
\end{align}
Since the right hand side of the system \eqref{PQW-ODE} is locally Lipschitz
continuous in $P,Q,W$, it follows in particular that $W(t)$ (and also $P(t)$ and
$Q(t)$) is uniquely determined by any initial values $P(t_0),Q(t_0),W(t_0)$.
	
Assume now that the claim of the proposition does not hold. Then there are
$t_0,t_1\in [0,\infty)$ with $\rank(W(t_1))> \rank(W(t_0))$. Since
$W(t)=W_N(t)\cdots W_1(t)$, it follows that
\[
  \min(d_0,\hdots, d_N)\geq \rank(W(t_1))> \rank(W(t_0)).
\]
We define  $\kk=\rank(W(t_0))$ and distinguish the cases $\kk=0$ and $\kk>0$.

\textbf{Case~1.} $\kk=0$. Then $W(t_0)=0$ and hence also $W_1(t_0)=0$. Due to balancedness it follows that $W_i(t_0)=0$ for all $i\in \{1,\hdots, N\}$. But then it follows that $W_i(t)=0$  for all $t\in[0,\infty)$  and for all $i\in \{1,\hdots, N\}$, hence also $W(t)=0$ for all $t\in[0,\infty)$, so the rank of $W$ is constant.

\textbf{Case~2.} $\kk>0$.
We  assume first that $t_1>t_0$ and will discuss the case $t_0>t_1$ below.

We  replace the first hidden layer (which has size $d_1$) by a new hidden layer
of size $\kk$ (all other layer sizes remain as before) and define new initial
values $\tilde W_1,\hdots, \tilde W_N$ (at $t_0$) for our new layer sizes in
such a way  that $\tilde W_1,\hdots, \tilde W_N$ are balanced and $\tilde
W:=\tilde W_N\cdots \tilde W_1=W(t_0)$ and $\tilde P:=\tilde W_1^T\tilde W_1=
W_1(t_0)^T W_1(t_0)$  and  $\tilde Q:=\tilde W_N\tilde W_N^T= W_N(t_0)
W_N(t_0)^T$. For $t\in [t_0,\infty)$, let $\tilde W_1(t),\hdots, \tilde W_N(t)$
be the corresponding solutions of  the gradient flow \eqref{gradflow} for the
new layer sizes with initial values at $t_0$ given by $\tilde W_1(t_0)=\tilde
W_1,\hdots, \tilde W_N(t_0)=\tilde W_N$. Similarly, let $\tilde W(t)=\tilde
W_N(t)\cdots \tilde W_1(t)$. Assuming that we can construct $\tilde W_1,\hdots,
\tilde W_N$ as above, it follows in particular that $\tilde W(t)=W(t)$ for all
$t\in [t_0,\infty)$, since, as discussed before,  $W(t)$ is uniquely determined
by $P(t_0)=\tilde P,Q(t_0)=\tilde Q, W(t_0)=\tilde W$ for all $t\in [0,\infty)$.
But  our new minimal layer size is $\kk$, so it follows that the product
$\tilde W(t)=\tilde W_N(t)\cdots \tilde W_1(t)$ has rank at most $\kk$ for any
$t\in [t_0,\infty)$. In particular, $\rank(W(t_1))=\rank(\tilde W(t_1))\leq
\kk$. This contradicts our assumption $\rank(W(t_1))> \rank(W(t_0))=\kk$.

Assume now that $t_0>t_1$. Here we cannot directly argue as above since
\emph{backward in time} we only have \emph{local} existence of solutions of
\eqref{gradflow}. However, since the set $\{W\in \RR^{d_y\times d_x}:
\rank(W)<\rank(W(t_1))\}$ is closed in $\RR^{d_y\times d_x}$, it follows that
the set $\{t\geq t_1: \rank(W(t))< \rank(W(t_1) \}$ has a minimum $\tau_0$,
which is larger than $t_1$. Then for any $\varepsilon>0$, there is a $\tau_1\in
(\tau_0-\varepsilon, \tau_0)$ with $\rank(W(\tau_1))>\rank (W(\tau_0))$. 

Now  replace $t_0$ by $\tau_0$, define as before
$\kk=\rank(W(t_0))=\rank(W(\tau_0))$ and assume that we can construct   $\tilde
W_1,\hdots, \tilde W_N$ as above. Then for some $\varepsilon>0$, the flows
$\tilde W_1(t),\hdots, \tilde W_N(t)$ solving the  gradient flow
\eqref{gradflow} for the new layer sizes and with initial values at $t_0$ given
by $\tilde W_1(t_0)=\tilde W_1,\hdots, \tilde W_N(t_0)=\tilde W_N$ are defined
(and balanced) on the interval $(t_0-\varepsilon,\infty)$. Next we replace $t_1$
by a suitable $\tau_1\in (\tau_0-\varepsilon, \tau_0)$ with
$\rank(W(\tau_1))>\rank(W(\tau_0))=\kk$. Now we can argue  as above: On the one
hand, the rank of $\tilde W(t_1)=\tilde W_N(t_1)\cdots \tilde W_1(t_1)$ is at most
$\kk$, on the other hand, we have $\tilde W(t_1)=W(t_1)$, so the rank of
$W(t_1)$ is also at most $\kk$, giving the desired contradiction.

 It remains to construct  $\tilde W_1,\hdots, \tilde W_N$ as announced. First, we introduce some notation.
Let ${\tilde d}_1=\kk$ and for $j\in \{0,\hdots, N\}\setminus \{1\}$ let $\tilde d_j=d_j$. (Thus the $\tilde d_j$ are our new layer sizes.)
 		Given  integers $a,b\geq \kk$ and $c_1,\hdots, c_{\kk}\in \RR$, we denote by $S_{a,b}(c_1,\hdots, c_\kk)\in \RR^{a\times b}$ the $a\times b$ diagonal matrix whose first $\kk$ diagonal entries are $c_1,\hdots, c_\kk$ and whose remaining entries are all equal to $0$.
 		
 		Now write $W:=W(t_0)=USV^T,$ where $U\in O(d_y)=O(d_N)$ and $V\in  O(d_x)=O(d_0) $ and $S=S_{d_N,d_0}(\sigma_1,\hdots,\sigma_\kk)$, where $\sigma_1\geq \hdots\geq \sigma_\kk>0$.
 		Let $W_N=W_N(t_0)$ and $W_1=W_1(t_0)$. Then
 		since 
 		$W^TW=(W_1^TW_1)^N$, we can write $W_1=U_1S_{d_1,d_{0}}(\sigma_1^{1/N},\hdots,\sigma_\kk^{1/N})V^T$ for some $U_1\in O(d_1)$. Similarly,
 		 since $WW^T=(W_NW_N^T)^N$, we have $W_N=US_{d_N,d_{N-1}}(\sigma_1^{1/N},\hdots,\sigma_\kk^{1/N})V_N^T$ for some $V_N\in O(d_{N-1})$.  
 			Define now $
 		\tilde W_1=S_{\tilde d_1, \tilde d_{0}}(\sigma_1^{1/N},\hdots,\sigma_\kk^{1/N})V^T$ and 
 		$\tilde W_N=US_{\tilde d_N,\tilde d_{N-1}}(\sigma_1^{1/N},\hdots,\sigma_\kk^{1/N})
 		$
 		and, for $j\in \{2,\hdots, N-1\}$, 
 		$\tilde W_j= S_{\tilde d_j,\tilde d_{j-1}}(\sigma_1^{1/N},\hdots,\sigma_\kk^{1/N}).
 		$
 		Note that this construction is possible since $\min(\tilde d_0,\hdots, \tilde d_N)=\kk$. (Compare \cite[Section 3.3]{arora2018convergence} for a similar construction of balanced initial conditions.)
 		Then  obviously, the $\tilde W_i$ are indeed balanced, and we have $\tilde W_1^T\tilde W_1= W_1^T W_1$  and  $\tilde W_N\tilde W_N^T= W_N W_N^T$ and $\tilde W_N\cdots \tilde W_1=W.$
 		This ends the proof.
\end{proof}

\begin{remark} Assume again the situation of Proposition \ref{rank_const}. Then 
	in the limit $t\to \infty$, the rank of $W$ still cannot increase, i.e., if $W(0)$ has rank $k$ then the rank of $\underset{t\to \infty}{\lim}W(t)$ is at most $k$. This follows from Proposition~\ref{rank_const} together with the fact that the set of matrices of rank at most $k$ is closed in $\RR^{d_y\times d_x}$. 
	However, it can happen that the rank of $\underset{t\to \infty}{\lim}W(t)$ is strictly smaller than $k$, see Remark~\ref{rem:conjecturecondition}  for an explicit example. 
\end{remark}

\begin{remark} Proposition~\ref{rank_const} may fail if the initial values $W_j(0)$, $j=1,\hdots,N$, are not balanced, i.e., the rank of $W(t)$ may then drop or increase in finite time. An example for such behaviour can be easily given in the case $N = 2$, $d_0 = d_1 = d_2 = 1$, $X = Y = 1$. Choosing $W_1(0) = 0$ and $W_2(0) = 1$ gives $W(0) = 0$. Moreover $\frac{d}{dt} W_1 (0) = W_2(0) = 1$ and $\frac{d}{dt} W_2 (0) = W_1(0) = 0$. This means that for $t\neq 0$ and $|t|$ sufficiently small we have $W_1(t)\neq 0$ and $W_2(t)\neq 0$, hence rank $W(t) = 1$. In other words, for small enough $\varepsilon > 0$, when moving with $t$ from $-\epsilon$ to $0$ the rank of $W(t)$ drops from $1$ to $0$ and when continuing from $t=0$ to $t=\epsilon$ the rank increases again to $1$.
\end{remark}

The statements of Lemma \ref{C1vectorfield} and Corollary \ref{uniquesolution} below are probably well known, but we include them here for completeness.
\begin{lemma}\label{C1vectorfield}
	Let $\mathcal M$ be a $C^2$-manifold which carries a Riemannian metric $g$ of class $C^1$ and let $L:\mathcal M\to \RR$ be a $C^2$-map. Then  $-\nabla^g(L)$ is a $C^1$-vector field. 
\end{lemma}
\begin{proof} In local coordinates, we have 
$-\nabla^g (L) = -\sum_{i,j} g^{i,j}\frac{\partial L}{\partial x_i}\frac{\partial}{\partial x_j}$, compare  \cite[Lemma 4.3]{bahu04}.
Since by assumption the matrix  with entries $g_{i,j}$ is $C^1$, also the inverse matrix $(g^{i,j})_{i,j}$ is $C^1$. Since also by assumption $L$ is a $C^2$-map, the partial derivatives $\frac{\partial L}{\partial x_i}$ are  $C^1$. It follows that $ -\nabla^g (L)$ is indeed a $C^1$-vector field.
\end{proof}

\begin{corollary}\label{uniquesolution}
	In the situation of Lemma \ref{C1vectorfield}, for any $x_0\in \mathcal M$, there is a unique maximal integral curve $\phi: J\to M$ with $\phi(0)=x_0$ and $$
	\dot\phi(t)=-\nabla^g(L(\phi(t))) \ \forall t\in J.
	$$
	Here maximal means that the interval $J$ is the maximal open interval containing $0$ with this property. 
\end{corollary}
\begin{proof}
	This follows from Lemma \ref{C1vectorfield} together with Theorem \ref{uniqueness} 
	in appendix \ref{appendix:flows}. For the existence of  $J$, see also appendix  \ref{appendix:flows} or directly \cite[Section IV, §2]{Lang99}.
\end{proof}

\begin{corollary}\label{cor:Riemann} Suppose that $XX^T$ has full rank and that $W_1(t), \ldots, W_N(t)$ are 
	solutions of the gradient flow \eqref{gradflow} of $L^N$, with initial values
	$W_j(0)$ that are balanced; recall Definition~\ref{D:BALA}. 
	Define the product $W(t) := W_N(t) \cdots W_1(t)$. If $W(0)$ is
	contained in $\mathcal{M}_k$ (i.e., has rank $k$), then $W(t)$ 
	solves for all $t\in[0,\infty)$ the gradient flow equation
	\begin{align}\label{riemflow}
	\dot{W}=-\nabla^g L^1(W)
	\end{align}
	on $\mathcal{M}_k$, where $\nabla^g$ denotes the Riemannian gradient of $L^1$ with respect to the
	metric $g$ on $\mathcal{M}_k$ defined in \eqref{skp}.
	Further this is the only solution of \eqref{riemflow} in $\mathcal M_k$.
\end{corollary}
\begin{proof}  Proposition \ref{rank_const}  shows that 
	 $W(t) \in \mathcal{M}_k$ for all  $t\in[0,\infty)$.
	Lemma~\ref{PosDefLemma} and the discussion below Proposition \ref{C1metric}  show that $W(t)$ solves indeed  equation \eqref{riemflow} with the particular choice of $g$ as in \eqref{skp} as the
	metric. (Note that then   \eqref{riemflow} 
	is a reformulation of
	\eqref{gradflowprod}.) By Corollary \ref{uniquesolution}  there are no other solutions in $\mathcal{M}_k$.
\end{proof}

\begin{theorem}\label{thm:balanced_flow_on_M_k} Assume that $XX^T$ has full rank and let $N\geq 2$.
	Then for any initialization $W(0)\in \RR^{d_y\times d_x}$, denoting by $k$  the rank of $W(0)$,   there
	is a uniquely defined flow  $W(t)$ on  $\mathcal{M}_k$  for  $t\in [0,\infty)$
	which satisfies \eqref{riemflow}.
\end{theorem}

\begin{proof}
	Any $W(0)\in \mathcal{M}_k$ can be written as a product $W(0)=W_N(0)\cdots
	W_1(0)$ for suitable balanced $W_i(0)\in \RR^{d_i\times d_{i-1}}, \ i=1,\hdots
	N$, where  $d_0=d_x$ and $d_N=d_y$ and the remaining $d_i$ (i.e. for $i\in
	\{1,\hdots, N-1\}$) are arbitrary integers greater than or equal to $k$; compare
	the proof of Proposition \ref{rank_const} or \cite[Section
	3.3]{arora2018convergence}. If $W_1(t),\hdots, W_N(t)$ satisfy equation
	\eqref{gradflow} with initial values  $W_1(0),\hdots, W_N(0)$, then $W(t)$ solves \eqref{riemflow} on  $\mathcal{M}_k$; see
	Corollary \ref{cor:Riemann}. Since the $W_i(t)$ are defined for all $t\geq 0$ by
	Theorem \ref{globconv}, the  claim follows (see again Corollary
	\ref{cor:Riemann} for the fact that $W$ is well defined, i.e., there are no other
	solutions on $\mathcal{M}_k$).
\end{proof}

\begin{remark}
Our Riemannian metric is (in the limit
$N\rightarrow \infty$) similar to the Bogoliubov inner product of quantum
statistical mechanics (when replacing $\mathcal{A}_W^{-1}$ with $\mathcal{A}_W$), which is defined on the manifold of positive definite
matrices; see \cite{carlenmaas12}.
\end{remark}

\section{Linear Autoencoders with one hidden layer}
\label{sec:autoencoders}

In this section we consider linear autoencoders with one hidden layer in the symmetric case, i.e., we assume $Y=X$ and $N=2$ and we impose that $W_2=W_1^T$. The nonsymmetric case with one hidden layer will be discussed in Appendix \ref{appendix:nonsymmetric}.
 
For $V:=W_2=W_1^T\in \RR^{d\times r}$ (where we write $d$ for $d_x=d_y$ and  $r$ for $d_1$), let
$$E(V)=L^2(V^T,V)=\frac 1 2\norm{X-VV^T X}_F^2.$$ 

We consider the gradient  flow:
\begin{equation}\label{gradflowSym}
\dot{V}=-\nabla E(V), \ \ V(0)=V_0,
\end{equation}
where we assume that $V_0^TV_0=I_r.$
Computing the gradient of $E$ gives  $$
\nabla E(V)= -(I_d-VV^T)XX^TV-XX^T(I_d-VV^T)V.
$$ 
Thus the gradient flow for $V$ is given by
\begin{equation}\label{gradflowSym2}
\hspace{-0.5cm} \dot{V}=(I_d-VV^T)XX^TV+XX^T(I_d-VV^T)V, \qquad V(0)=V_0, \ V_0^TV_0=I_r.
\end{equation}
This can be analyzed using results by Helmke, Moore, and Yan on Oja's flow  \cite{yahemo94}.

\begin{theorem}
	\begin{enumerate}
		\item The flow (\ref{gradflowSym2}) has a unique solution on the interval $[0,\infty)$. 
		\item $V(t)^TV(t)=I_r$ for all $t\geq 0$. 
		\item The limit $\overline V=\lim_{t\rightarrow \infty} V(t)$ exists and it is an equilibrium. 
		\item The convergence is exponential: There are positive constants $c_1,c_2$ such that $$\|V(t)-\overline V\|_F\leq c_1e^{-c_2t}$$ for all $t\geq 0$.

\item 	The equilibrium points of the  flow (\ref{gradflowSym2})  are precisely the matrices of the form $$\overline	V=\left ( v_{1}|\hdots| v_{r} \right ) Q,
	$$
	where
		  $v_{1}, \hdots, v_{r}$ are orthonormal eigenvectors of $XX^T$
		and $Q$ is an orthogonal $r\times r$-matrix.
		\end{enumerate}
\end{theorem}
\begin{proof}
	In \cite{yahemo94} it is shown that Oja's flow given by $$\dot{V}=(I_d-VV^T)XX^TV $$
	satisfies all the claims in the proposition provided that  $V(0)^TV(0)=I_r$. In particular, 
	by \cite[Corollary 2.2]{yahemo94}, all $V(t)$ in any solution of Oja's flow with  $V(0)^TV(0)=I_r$
	fulfill  $V(t)^TV(t)=I_r$. It follows that under the initial condition  $V(0)^TV(0)=I_r$ the
	 flow (\ref{gradflowSym2}) is identical to Oja's flow because the term $XX^T(I_d-VV^T)V$ then vanishes for all $t$ if $V$ is a solution to Oja's flow.

	Hence,
	(2) follows from \cite[Corollary 2.2]{yahemo94}. In \cite[Theorem 2.1]{yahemo94} an existence and uniqueness result on $[0,\infty)$ is shown for Oja's flow and thus implies  (1). Statements  (3) and (4) follow from \cite[Theorem 3.1]{yahemo94} (which states that the solution to Oja's flow exponentially converges to an equilibrium point). Point (5) follows from \cite[Corollary 4.1]{yahemo94} (which shows that the equilibrium points $V$ of Oja's flow satisfying $V^TV=I_r$ are of the claimed form).
\end{proof}
\begin{remark}
	Choosing $v_{1}, \hdots, v_{r}$ orthonormal  eigenvectors corresponding to the largest $r$ eigenvalues of $XX^T$, we obtain 
	(for varying $Q$)
	precisely the possible solutions for the matrix  $V$ in  the $PCA$-problem.
\end{remark}
In order to make this more precise and to see this claim, we recall the PCA-Theorem, cf. \cite{murphy2013machine}.
Given: $x_1,\hdots,x_m\in \RR^d$
and $1\leq r\leq d$, we consider the following 
problem: Find  $v_1,\hdots, v_r \in\RR^d$ orthonormal and  $h_1,\hdots, h_m \in\RR^r$ such that \begin{equation}\label{pca}
\mathcal{L}(V; h_1 \hdots, h_m):=\frac{1}{m}\sum_i\|x_i-Vh_i\|_2^2
\end{equation}
is minimal. (Here $V=\left (v_1|\hdots | v_r \right )\in \RR^{d\times r} $.) 

\begin{theorem}[PCA-Theorem \cite{murphy2013machine}]	A minimizer of  (\ref{pca}) is obtained by choosing   $v_1, \hdots, v_r$ as  orthonormal eigenvectors corresponding to the $r$ largest eigenvalues of $\sum_i x_ix_i^T=XX^T$ and $h_i=V^T x_i$.

	The other possible solutions for $V$ are of the form 
	$V=\left (v_1|\hdots | v_r \right )Q$, where  $v_1, \hdots, v_r$ are chosen as above and  $Q$ is an orthogonal  $r\times r
	$-matrix. Again  $h_i=V^T x_i$.
\end{theorem}

Let $\lambda_1\geq\hdots\geq\lambda_d$ be the eigenvalues of $XX^T$ and let $v_1,\hdots, v_d$ be corresponding orthonormal eigenvectors. 
\begin{theorem}\label{optconv1}
	Assume that  $XX^T$ has full rank and that $\lambda_{r}>\lambda_{r+1}$.
	Then $\lim_{t\rightarrow \infty} V(t)=\left ( v_{1}|\hdots| v_{r} \right ) Q$ for some orthogonal $Q$ 
	if and only if $  V_0^T\left( v_{1}|\hdots| v_{r} \right)$ has rank $r$.
\end{theorem}
\begin{proof}
	This follows from \cite[Theorem 5.1]{yahemo94} (where an analogous statement for Oja's flow is made) together with \cite[Corollary 2.1]{yahemo94}.
\end{proof}

\begin{corollary}\label{cor-auto-global-convergence}
	Under the assumptions of Theorem \ref{optconv1},
	 for almost all initial conditions (w.r.t. the Lebesgue measure), the flow converges to an optimal equilibrium, i.e., one of the form $V=\left ( v_{1}|\hdots| v_{r} \right ) Q$ in the notation of Theorem \ref{optconv1}. 
\end{corollary}
\begin{proof}
	This follows from Theorem \ref{optconv1}, cf. also the analogous \cite[Corollary 5.1]{yahemo94}.
\end{proof}
In Section~\ref{sec:saddle-points} we extend this result to 
autoencoders with $N > 2$ layers using a more abstract approach.

	The following theorem shows that the optimal equilibria are the only stable equilibria: 
\begin{theorem}\label{nonstable}
	Assume $V=\left ( v_{i_1}|\hdots| v_{i_r} \right ) Q$, where the orthonormal eigenvectors $ v_{i_1},\hdots, v_{i_r}$ are not eigenvectors   corresponding to the largest $r$ eigenvalues of $XX^T$.
	Then in any neighborhood of $V$ there is a matrix $\widetilde V$ with $E(\widetilde V)<E(V)$ (and $\widetilde V^T \widetilde V=I_r$).
\end{theorem}
\begin{proof}
	Let $v_{i_j}$ be one of the eigenvectors $ v_{i_1},\hdots, v_{i_r}$ whose eigenvalue does not belong to the $r$ largest eigenvalues of $XX^T$. Let $v$ be an eigenvector of $XX^T$ of unit length which is orthogonal to the eigenvectors  $ v_{i_1},\hdots, v_{i_r}$ and whose eigenvalue  $\lambda$ belongs to the $r$ largest eigenvalues of $XX^T$. Now for any $\varepsilon\in [0,1]$ consider $v_{i_j}(\varepsilon):=\varepsilon v +\sqrt{1-\varepsilon^2}v_{i_j}$. Then $V(\varepsilon):=\left ( v_{i_1}|\hdots|v_{i_j}(\varepsilon)|\hdots| v_{i_r} \right ) Q$ satisfies $E(V(\varepsilon))<E(V)$ for $\varepsilon\in (0,1]$ and $V(\varepsilon)^TV(\varepsilon)=I_r$.
	To see that indeed $E(V(\varepsilon))<E(V)$, we compute $E(V)=\frac 12 \norm{X-VV^TX}_F^2=\frac 12 \tr(XX^T-XX^TVV^T)$ and  $E(V(\varepsilon))=\frac 12 \tr(XX^T-XX^TV(\varepsilon)V(\varepsilon)^T).$ Writing $XX^Tv_{i_k}=\lambda_{i_k}v_{i_k}$, we note that $\tr\left(XX^TVV^T\right)=\sum_{k=1}^r \lambda_{i_k}$ and 
	 $\tr\left(XX^TV(\varepsilon)V(\varepsilon)^T\right)=\varepsilon^2 \lambda +(1-\varepsilon^2)\lambda_{i_j}+\sum_{k=1, k\neq j}^r \lambda_{i_k}$. Since $\lambda> \lambda_{i_j}$, the claim follows.
\end{proof}

\section{Avoiding saddle points}
\label{sec:saddle-points}

In Section~\ref{sec:convergence} we have proven convergence of the gradient flow \eqref{gradflow} 
a to critical point of $L^N$. (Together with Proposition \ref{prop:critical-LN-general}  below, this also implies that the product $W$ converges to a critical point of 
 $L^1$ 
restricted to $\mathcal{M}_k$ for some $k\leq r$.)
Since we will remain in a saddle point forever if the initial point is a saddle point, 
the best we can hope for is convergence to global optima for almost all initial points
(as in Corollary~\ref{cor-auto-global-convergence} for the particular autoencoder case with $N=2$).

We will indeed establish such a result for both $L^N$ and $L^1$ restricted to $\mathcal{M}_r$
in the autoencoder case. We note, however, that we can only ensure that the limit corresponds
to an optimal point for $L^1$ restricted to $\mathcal{M}_k$ for some $k \leq r$ for almost all initialization.
We conjecture $k=r$ (for almost all initializations), but this remains open for now. 

We proceed by showing a general result on the avoidance of saddle points
by extending the main result of \cite{Lee19} from gradient descent to gradient flows. 
A crucial ingredient is the notion of a strict saddle point. The application of the general abstract result
to our scenario then requires to analyze the saddle points.


\subsection{Strict saddle points}

We start with the definition of a strict saddle point of a function on the Euclidean space $\mathbb{R}^d$.

\begin{definition}
Let $f : \Omega \to \mathbb{R}$ be a twice continuously differentiable function on an open domain $\Omega \subset \mathbb{R}^d$.
A critical point $x_0 \in \Omega$ is called a strict saddle point if the Hessian $Hf(x_0)$ has a negative eigenvalue.
\end{definition} 

Intuitively, the function $f$ possesses a direction of descent at a strict saddle point. 
Note that our definition also includes local maxima, which does not pose problems for our purposes.

Let us extend the notion of strict saddle points to functions on Riemannian manifolds $(\mathcal{M},g)$.
To this end, we first introduce the Riemannian Hessian of a $C^2$-function $f$ on $\mathcal{M}$.
Denoting by $\nabla$ be the Riemannian connection (Levi-Civita connection) 
on $(\mathcal{M},g)$ the Riemannian Hessian of 
$f$ at $x \in \mathcal{M}$ is the linear mapping $\operatorname{Hess} f(x) : T_x \mathcal{M} \to T_x\mathcal{M}$ defined by
\[
\operatorname{Hess}^g f(x)[\xi] := \nabla_\xi \nabla^{g} f.
\]
Of course, if $(\mathcal{M},g)$ is Euclidean, then this definition can be identified with the standard definition of
the Hessian. Moreover, if $x \in \mathcal{M}$ is a critical point of $f$, i.e., $\nabla^g f(x)=0$, then
the Hessian $\operatorname{Hess}^g f(x)$ is independent of the choice of the connection.
Below, we will need the following chain type rule for curves $\gamma$ on $\mathcal{M}$, 
see e.g.~\cite[Eq.~(3.1)]{ottowestdickenberg06},
\begin{equation}\label{Riemannian-chain:rule}
 \frac{d^2}{dt^2} f (\gamma(t))=g\left( \dot\gamma(t), \Hess^g f(\gamma(t))[\dot\gamma(t)] \right) 
 +g\left( \frac{D}{dt}\dot\gamma(t), \nabla^g f(\gamma(t))\right),
\end{equation}
where  $\frac{D}{dt}\dot\gamma(t)$ is related to the Riemannian connection that is used to define the Hessian,
see \cite[Section 5.4]{Absil08}.  We refer to \cite{Absil08} for more details on the Riemannian Hessian.

\begin{definition}
Let $(\mathcal{M},g)$ be a Riemannian manifold with Levi-Civita connection $\nabla$ and let 
$f : \mathcal{M} \to \mathbb{R}$ be a twice continuously differentiable function. A critical point $x_0 \in \mathcal{M}$, i.e.,
$\nabla^g f(x_0) = 0$ is called a strict saddle point if $\operatorname{Hess} f(x)$ has a negative eigenvalue.
We denote the set of all strict saddles of $f$ by $\mathcal{X} = \mathcal{X}(f)$.
We say that $f$ has the strict saddle point property, if all critical points of $f$ 
that are not local minima are strict saddle points.
\end{definition}

Note that our definition of strict saddle points includes local maxima, which is fine for our purposes.

\subsection{Flows avoid strict saddle points almost surely}

We now prove a general result that gradient flows on a Riemannian manifold $(\mathcal{M},g)$ 
for functions
with the strict saddle point property
avoid saddle point for almost all initial values. This result extends the main result of \cite{Lee19}
from time discrete systems to continuous flows and should be of independent interest.

For a twice continuously differentiable function $L : \mathcal{M} \to \mathbb{R}$, we consider the Riemannian gradient flow
\begin{equation}\label{abstract:flow}
\frac{d}{dt} \phi(t) = - \nabla^g L( \phi(t)), \quad \phi(0) = x_0  \in \mathcal{M},
\end{equation}
where $\nabla^g$ denotes the Riemannian gradient. When emphasizing the dependence on $x_0$, we write 
\begin{equation}\label{def:psi}
\psi_t(x_0) = \phi(t),
\end{equation}
where $\phi(t)$ is the solution to \eqref{abstract:flow} with initial condition $x_0$.

Sets of measure zero on $\mathcal{M}$ (as used in the next theorem) 
can be defined using push forwards of the Lebesgue measure on charts of the manifold $\mathcal{M}$.

\begin{theorem}\label{thm:avoid-saddles} Let  $L : \mathcal{M} \to \mathbb{R}$ be a 
$C^2$-function
on a second countable finite dimensional Riemannian manifold $(\mathcal{M},g)$, where we assume that $\mathcal{M}$ is of class $C^2$ as a manifold and the metric $g$ is of class $C^1$.  Assume that  $\psi_t(x_0)$ exists
for all $x_0 \in \mathcal{M}$ and  all $t \in [0,\infty)$. 
Then the set
\[
\mathcal{S}_L := \{ x_0 \in \mathcal{M}  : \lim_{t \to \infty} \psi_t(x_0) \in \mathcal{X} = \mathcal{X}(L) \}
\]
of initial points such that the corresponding flow converges to a strict saddle point of $L$ has measure zero. 
\end{theorem}

The proof of this relies on the following result for iteration maps (e.g., gradient descent iterations)
shown in \cite{Lee19}. 

\begin{theorem}\label{thm:saddles-discrete} Let $h : \mathcal{M} \to \mathcal{M}$ be a continuously differentiable function on a second countable differentiable finite-dimensional 
manifold such that $\det(Dh(x)) \neq 0$ for all $x \in \mathcal{M}$ (in particular,
$h$ is a local $C^1$ diffeomorphism). Let 
\[
\mathcal{A}^*_h = \{x \in \mathcal{M} : h(x) = x, \max_j |\lambda_j(Dh(x))| > 1 \},
\]
where $\lambda_j(Dh(x))$ denote the eigenvalues of $Dh(x)$, 
and consider sequences with initial point $x_0 \in \mathcal{M}$, $x_k = h(x_{k-1})$, $k \in \NN$.
Then the set $\{x_0 \in \mathcal{M} : \lim_{k \to \infty} x_k \in \mathcal{A}^*_h\}$ has measure zero.
\end{theorem}

\begin{proof}[Proof of Theorem \ref{thm:avoid-saddles}]
By Lemma \ref{C1vectorfield} and Theorem \ref{flowprop} in appendix \ref{appendix:flows}, the map $$h\colon \mathcal{M} \to \mathcal{M}, \ x_0 \mapsto \psi_T(x_0)$$ defines  a diffeomorphism of $\mathcal M$ onto an open subset of $\mathcal M$. 
In particular, $D h = D \psi_T$ is non-singular, i.e. $\det (Dh(x)) \neq 0$ for all $x \in \mathcal{M}$.

Because of the semigroup property $\psi_{t+s}(x_0) = \psi_t ( \psi_s (x_0))$ the
sequence $x_k = \psi_{kT}(x_0)$, $k \in \NN$, satisfies $x_{k} = h(x_{k-1})$ and
$\lim_{t \to \infty} \psi_t(x_0) \in \mathcal{X}$ implies $\lim_{k \to \infty}
x_k \in \mathcal{X}$.

By Theorem~\ref{thm:saddles-discrete} the set
\[
\{x_0 \in \mathcal{M} : \lim_{k \to \infty} \psi_{kT}(x_0) \in \mathcal{A}^*_{\psi_T} \}
\]
has measure zero. We need to show that if $\bar{x}$ is a strict saddle point of
$L$, then $\bar{x} \in \mathcal{A}^*_{\psi_T}$ for suitable (i.e., sufficiently
small) $T>0$. We will work with a sequence of parameters $T = \frac{1}{n}$ with
$n \in \NN$. 

Let $\bar{x} \in \mathcal{X}(L)$ be a strict saddle point of $L$. If we choose
local coordinates giving rise to an orthonormal basis with respect to the
Riemannian metric at $\bar x$, then it follows from (\ref{abstract:flow}) that,
for all $n \in \NN$,
\[
  D \psi_{1/n}(\bar{x}) = \Id - \frac{1}{n} \Hess^g L (\bar{x}) + o(1/n),
\]
where $\lim_{t \to 0} o(t)/t = 0$. Compare also \cite[Lemma 4.4]{bahu04} for the
fact that we can identify here the differential of $\nabla^g L(\bar{x})$ with
$\Hess^g L (\bar{x})$. (More precisely, it is shown in loc.\ cit.\ that the the
differential of $\nabla^g L(\bar{x})$ coincides with the matrix
$(\frac{\partial^2 L(\bar x)}{\partial x_i \partial x_j})_{i,j}$ at the critical
point $\bar x$, if we assume that the local coordinates give rise to an
orthonormal basis at this point. Using again that  $\bar x$ is  a critical
point, we see that this matrix is the Riemannian Hessian at $\bar x$ in our
local coordinates.) Since $\bar{x}$ is a strict saddle point of $L$, the matrix
$\Hess^g L(\bar{x})$ has at least one strictly negative eigenvalue. It follows
that there exists $N \in \NN$ such that for all $n \geq N$ the differential $D
\psi_{1/n}(\bar{x})$ has an eigenvalue larger than $1$. Hence $\bar{x} \in
\mathcal{A}^*_{\psi_{1/n}}$ and
\begin{align*}
& \{x_0 \in \mathcal{M} : \lim_{t \to \infty} \psi_t(x_0) = \bar{x} \} \subset \{x_0 \in \mathcal{M} : \lim_{k \to \infty} \psi_{k/n}(x_0) = \bar{x} \} 
 \subset \{x_0 \in \mathcal{M} : \lim_{t \to \infty} \psi_{k/n} \in \mathcal{A}^*_{\psi_{1/n}}\}  
\end{align*}
for all $n \geq N$.
It follows that 
\[
\{ x_0 \in \mathcal{M} : \lim_{t \to \infty} \psi_t(x_0) \in \mathcal{X}(L)\} \subset
\bigcup_{n \in \NN} \{x_0 \in \mathcal{M} : \lim_{k \to \infty} \psi_{k/n}(x_0) \in \mathcal{A}^*_{\psi_{1/n}} \}.
\]
The set on the right hand side is a countable union of null sets and therefore has measure zero.
This implies the claim of the theorem and the proof is completed.
\end{proof}

\begin{remark} The proof of 
	 Theorem~\ref{thm:saddles-discrete} uses the center and stable manifold theorem, see, e.g., \cite[Chapter 5, Theorem III.7]{Shub86}. If the absolute 
eigenvalues of $Dh(x)$ are all different from $1$, i.e., if all eigenvalues of the Hessian 
$\operatorname{Hess}^gf(x)$ are different from $0$ at a saddle point $x$, then slightly stronger conclusions may be
drawn, including the speed at which the flow moves away from saddle points. We will not elaborate on this point
here.
\end{remark}

\subsection{The strict saddle point property for $L^1$ on $\mathcal{M}_r$}

In this section we establish the strict saddle point property of $L^1$ on $\mathcal{M}_k$ 
by showing that the Riemannian Hessian $\operatorname{Hess} L^1$ at all critical points that are not a global minimizer
has a strictly negative eigenvalue. 
We assume  that $XX^T$ has full rank $d_x=d_0$ and start with an analysis of the critical points.
We first recall the following result of Kawaguchi \cite{kawag16}. 

\begin{theorem}\cite[Theorem 2.3]{kawag16}\label{thm-kawag}
Assume that $XX^T$ and $XY^T$ are of full rank with $d_y\leq d_x$ and that the 
matrix $YX^T(XX^T)^{-1}XY^T$ has $d_y$ distinct eigenvalues. Let $r$ be the 
minimum of the $d_i$. Then the loss function $L^N(W_1,\hdots, W_N)$ has the 
following properties.
\begin{enumerate}
\item It is non-convex and non-concave.

\item Every local minimum is a global minimum.

\item Every critical point that is not a global minimum is a saddle point.

\item If $W_{N-1}\cdots W_2$ has rank $r$ then the Hessian at any saddle point 
has at least one 
negative eigenvalue.
\end{enumerate} 
\end{theorem}
Below we will remove the assumption that $X Y^T$ has full rank and that
$YX^T(XX^T)^{-1}XY^T$ has distinct eigenvalues. Moreover, we will give more
precise information on the strict saddle points.

The following matrix, which is completely determined by the given matrices $X,
Y$ that define $L^1$ and $L^N$ (see \eqref{E:LN} and \eqref{E:LONE}), will play
a central role in our discussion. We define
\begin{equation}\label{def:Q-matrix}
Q:=YX^T(XX^T)^{-\frac 12}
\end{equation}
and let $q := \operatorname{rank}(Q)$ be its rank. We will use a reduced
singular value decomposition 
\[
  Q = U\Sigma V^T
    = \sum_{i=1}^q \sigma_iu_iv_i^T,
\]
of $Q$, where $\sigma_1 \geq \hdots \geq \sigma_q > 0$ are the singular values of $Q$ and  
$U \in \RR^{d_y \times q}$, $V \in \RR^{d_x \times q}$ have orthonormal columns
$u_1, \hdots, u_q$ and $v_1,\hdots,v_q$, respectively. Clearly, it holds $q \leq
n := \min\{d_x,d_y\}$.

Let $k\leq n$ and let $g$ be an arbitrary Riemannian metric on the manifold $\mathcal{M}_k$ of all matrices in  $\RR^{d_y\times d_x}$  of rank $k$, for example it could be the metric induced by the standard metric on $\RR^{d_y\times d_x}$ or the metric $g$ introduced in 
Section~\ref{sec: Riemannian_gradient_flows} for some number of layers $N$.

The next statement is similar in spirit to Kawaguchi's result, Theorem~\ref{thm-kawag}, and follows from \cite{TragerKohnBruna19}.

\begin{proposition}\label{critpoints_general} Let $Q$ be defined by
\eqref{def:Q-matrix} and $q = \operatorname{rank}(Q)$.
\begin{enumerate}
\item	
The critical points of $L^1$ on $\mathcal{M}_k$ are precisely the matrices of
the form 
\begin{equation}
  W =\sum_{j\in J}\sigma_j u_j v_j^T(XX^T)^{-\frac 12},
\label{E:FORK}
\end{equation}
where $J\subseteq \{1,\hdots, q\}$ consists of precisely $k$ elements.
Consequently, if $k>q$, then no such subset $J$ can exist and therefore $L^1$
restricted to $\mathcal{M}_k$ cannot have any critical points.
\item 
If $W$ is a critical point of $L$ (so that $W$ has the form \eqref{E:FORK}), then 
\[
  L^1(W) = \frac12 \left(\tr(YY^T)-\sum_{j\in J}\sigma^2_j\right).
\]
It follows that the critical point $W$ is a global minimizer of $L^1$ on
$\mathcal{M}_k$ if and only if
\[
  \{\sigma_j:j\in J\}=\{\sigma_1,\hdots, \sigma_k\},
\]
i.e., the set $J$ picks precisely the $k$ largest singular values of $Q$. In
particular, if $k = q$, then there cannot be any saddle points. Recall that
there are no critical points if $k>q$ because of (1).
\end{enumerate} 
\end{proposition}
\begin{proof}
For $X=\id$ see the proof of \cite[Theorem 28]{TragerKohnBruna19}. To obtain the
general case we observe that
\begin{equation*}
    L^1(W) = \frac 12 \norm{WX-Y}_F^2
      = \frac 12 \norm{W(XX^T)^{\frac 12}-YX^T(XX^T)^{-\frac 12}}_F^2+C
      = \frac 12 \norm{W(XX^T)^{\frac 12}-Q}_F^2+C,
\end{equation*}
where $C := \frac12 \|Y\|_F^2 - \frac12 \|Q\|_F^2$ does not depend on $W$. Since
$XX^T$ has full rank, the map $W\mapsto W(XX^T)^{\frac 12}$ is invertible (on
any $\mathcal{M}_k$). Therefore the critical points of the map $W\mapsto\frac 12
\norm{W(XX^T)^{\frac 12}-Q}_F^2$ restricted to $\mathcal{M}_k$ are just the
critical points of the map  $W\mapsto\frac 12 \norm{W-Q}_F^2$ (restricted to
$\mathcal{M}_k$) multiplied by  $(XX^T)^{-\frac 12}$. Now we substitute the
results of \cite[Theorem 28]{TragerKohnBruna19} on the critical points of the
map $W\mapsto\frac 12 \norm{W-Q}_F^2$ restricted to $\mathcal{M}_k$ (which are
just as claimed here in the case $X=\id$) and we obtain the claim of the
proposition.
\end{proof}

\begin{proposition}\label{prop:strict_saddle_property_general}
	The function $L^1$ on $\mathcal{M}_k$ for $k \leq n$ 
	satisfies the strict saddle point property. More precisely, all critical points of $L^1$ on $\mathcal{M}_k$ except for the global minimizers are strict saddle points.
\end{proposition}
\begin{proof} If $k\geq q = \operatorname{rank}(Q)$ then there are no saddle points by Proposition \ref{critpoints_general} so that the statement holds trivially.
Therefore, we assume $k < q$ from now on.
	By Proposition \ref{critpoints_general}, it is enough to show that the Riemannian Hessian of $L^1$ has a negative eigenvalue at any point of the   form 
	$$	W =\sum_{j\in J}\sigma_j u_j v_j^T(XX^T)^{-\frac 12},$$
	where $J\subseteq \{1,\hdots, q\}$ consists of  precisely $k$ elements and has the property that there is a $j_0\in J$ with $\sigma_{j_0}<\sigma_k$.
	Thus there is also a $\sigma_{j_1}\in \{\sigma_1,\hdots, \sigma_k\}$ with $\sigma_{j_1}>\sigma_{j_0}$ and $j_1\not\in J$.
	We define for  $t\in (-1,1)$: $$
	u_{j_0}(t)=t u_{j_1} +\sqrt{1-t^2}u_{j_0} \text{ and } v_{j_0}(t)=t v_{j_1} +\sqrt{1-t^2}v_{j_0}.
	$$
	Now consider the curve $\gamma:(-1,1)\rightarrow \mathcal{M}_k$ given by
	$$\gamma(t)=\left(\sigma_{j_0}u_{j_0}(t)v_{j_0}(t)^T+ \sum_{j\in J,j\neq j_0} \sigma_ju_jv_j^T\right)(XX^T)^{-\frac 12}.$$ Obviously we have $\gamma(0)=W$.
	We claim that it is enough to show that $$
	\left.\frac{d^2}{dt^2}L^1(\gamma(t))\right|_{t=0}<0.
	$$
	Indeed, by \eqref{Riemannian-chain:rule} it holds (for any Riemannian metric $g$) that
	$$
	\frac{d^2}{dt^2}L^1(\gamma(t))=g\left( \dot\gamma(t), \Hess^g L^1(\gamma(t))\dot\gamma(t) \right) +g\left( \frac{D}{dt}\dot\gamma(t), \nabla^g L^1(\gamma(t))\right),
	$$
	and since $\nabla^g L^1(\gamma(0))=\nabla^g L^1(W)=0$, it follows that $g\left( \dot\gamma(0), \Hess^g L^1(W)\dot\gamma(0) \right)<0$  if $
	\left.\frac{d^2}{dt^2}L^1(\gamma(t))\right|_{t=0}<0
	$ and hence that $\Hess^g L^1(W)$ has a negative eigenvalue in this case. (Note that $\Hess^g L^1(W)$ is self-adjoint with respect to the scalar product $g$ on $T_W(\mathcal{M}_k)$ and that it cannot be positive semidefinite (wrt.~$g$) if $g\left( \dot\gamma(0), \Hess^g L^1(W)\dot\gamma(0) \right)<0$, hence it has a negative eigenvalue in this case.)

	We note that 
	\begin{equation}\label{L1-curve-general}
	L^1(\gamma(t))=\frac 12 \norm{\gamma(t)X-Y}_F^2=\frac 12\tr(\gamma(t)^T\gamma(t)XX^T-2\gamma(t)XY^T+YY^T).
	\end{equation}
	We compute
	\begin{align*}
	& \left(\sigma_{j_0}v_{j_0}(t)u_{j_0}(t)^T+ \sum_{j\in J,j\neq j_0}\sigma_{j}v_j u_j^T\right) \left(\sigma_{j_0}u_{j_0}(t)v_{j_0}(t)^T+ \sum_{j\in J,j\neq j_0} \sigma_{j}u_jv_j^T\right) \\
	& = \sum_{j \in J \setminus \{j_0\}} \sigma_j^2 v_j v_j^T + \sigma_{j_0}^2 v_{j_0}(t) v_{j_0}(t)^T
	\end{align*}
	so that
	\begin{align*}
	\tr(\gamma(t)^T\gamma(t)XX^T) 
	&= \tr\left( (XX^T)^{-\frac 12} \left( \sum_{j \in J \setminus \{j_0\} } \sigma_j^2 v_j v_j^T + \sigma_{j_0}^2 v_{j_0}(t) v_{j_0}(t)^T
	 \right) (XX^T)^{-\frac 12}XX^T\right)\\
	&=\sum_{j\in J}\sigma_j^2.
	\end{align*}
	In particular, this expression is independent of $t$.
	Further,
	\begin{align*}
	\tr(-2\gamma(t)XY^T)& =-2\tr\left(\left(\sigma_{j_0}u_{j_0}(t)v_{j_0}(t)^T+ \sum_{j\in J,j\neq j_0} \sigma_{j}u_jv_j^T\right)(XX^T)^{-\frac 12}XY^T\right) \displaybreak[2]\\
	&=-2\tr\left(\left(\sigma_{j_0}u_{j_0}(t)v_{j_0}(t)^T+ \sum_{j\in J,j\neq j_0} \sigma_{j}u_jv_j^T\right)Q^T\right)\displaybreak[2]\\
	&=-2\tr\left(\left(\sigma_{j_0}u_{j_0}(t)v_{j_0}(t)^T+ \sum_{j\in J,j\neq j_0} \sigma_{j} u_j v_j^T\right)\sum_{j=1}^q\sigma_{j}v_j u_j^T\right)\displaybreak[2]\\
	&=-2\tr\left(\sigma_{j_0}u_{j_0}(t)v_{j_0}(t)^T(\sigma_{j_0}v_{j_0}u_{j_0}^T+\sigma_{j_1}v_{j_1}u_{j_1}^T)  \right) -2 \sum_{j\in J,j\neq j_0}\sigma_j^2\displaybreak[2]\\
	&= -2(\sigma_{j_0}^2(1-t^2)+t^2\sigma_{j_0}\sigma_{j_1}) -2 \sum_{j\in J,j\neq j_0}\sigma_j^2\displaybreak[2]\\
	&= 2t^2\sigma_{j_0}(\sigma_{j_0}-\sigma_{j_1})-2 \sum_{j\in J}\sigma_j^2.
	\end{align*}
	Together with equation \eqref{L1-curve-general} it follows that 
	$$
	\left.\frac{d^2}{dt^2}L^1(\gamma(t))\right|_{t=0}=2\sigma_{j_0}(\sigma_{j_0}-\sigma_{j_1})<0.
	$$	This concludes the proof.
\end{proof}
We note that a construction similar to the curve constructed in the preceding proof is considered in the proof of \cite[Theorem 28]{TragerKohnBruna19}. However,  it is not discussed there that this implies strictness of the saddle points.

\subsection{Strict saddle points of $L^N$}

Before discussing the strict saddle point property, let us first investigate the relation of the critical points of $L^N$
and the ones of $L^1$ restricted to $\mathcal{M}_r$, where 
\[
r = \min \{d_0,d_1,\hdots,d_N\}.
\]
Throughout this section we assume  that $X X^T$ has full rank. 
\begin{proposition}\label{prop:critical-LN-general}  
	\begin{itemize}
		\item[(a)] 
		Let $(W_1,\hdots,W_N)$ be a critical point of $L^N$. Define $W= W_N \cdots W_1$ and let $k := \operatorname{rank}(W) \leq r$.
		Then $W$ is a critical point of $L^1$ restricted to $\mathcal{M}_k$.
		\item[(b)] Let $W$ be a critical point of $L^1$ restricted to $\mathcal{M}_k$ for some $k\leq r$. Then
		there exists a tuple $(W_1,\hdots,W_N)$ with $W_N \cdots W_1 = W$ that is a critical point of $L^N$.
	\end{itemize}
\end{proposition}
\begin{proof} For (a), let $Z \in T_W(\mathcal{M}_k)$ be arbitrary, i.e., $Z =  WA + BW$ for some matrices 
	$A \in \RR^{d_x \times d_x}$ and $B\in  \RR^{d_y \times d_y}$. It suffices to show that for a curve $\gamma : \RR \to \mathcal{M}_k$ with
	$\gamma(0) = W$ and $\dot{\gamma}(0) = Z$ that $\left.\frac{d}{dt} L^1(\gamma(t))\right|_{t=0}  = 0$.
	We choose the curve
	\begin{equation}\label{def:gammaZ-global}
	\gamma(t) = (W_N + t V_n) \cdot W_{N-1} \cdots W_2 \cdot (W_1 + t V_1),
	\end{equation}
	where $V_1 = W_1 A$ and $V_N =   B W_N$. Then, indeed $\gamma(0) = W_N \cdots W_1 = W$ and 
	$\dot{\gamma}(0) = W_N W_{N-1} \cdots W_1 A + B W_N \cdots W_1 = Z$. Next, observe that
	\begin{align*}
	\left.\frac{d}{dt} L^1(\gamma(t))\right|_{t=0}  &= \left. \frac{d}{dt} L^N(W_1 + tV_1, W_2,\hdots, W_{N-1}, W_N + t V_N) \right|_{t=0}\\
	& = \langle \nabla L^N(W_1,\hdots,W_N), (V_1,0,\hdots,0,V_N)\rangle = 0,
	\end{align*}
	since $(W_1,\hdots,W_N)$ is a critical point of $L^N$. Since $Z$ was arbitrary, this shows (a).
	
	For 
	(b) we first note that by Lemma \ref{L:PREL}, for a point $(W_1,\hdots, W_N)$ to be a critical point of $L^N$, it suffices that  
	\begin{equation}\label{critcond-general}
	(WXX^T-YX^T)W_1^T=0 \quad \text{ and } \quad W_N^T(WXX^T-YX^T)=0.
	\end{equation}
	This is equivalent to 
	\begin{equation}\label{critcond-general2}
	WXX^TW_1^T=Q(XX^T)^{\frac 12}W_1^T \quad \text{ and } \quad W_N^TWXX^T=W_N^TQ(XX^T)^{\frac 12}.
	\end{equation}
	Since $W$ is a critical point of $L^1$ restricted to $\mathcal{M}_k$, we can write
	$$	W =\sum_{j\in J}\sigma_j u_j v_j^T(XX^T)^{-\frac 12},$$
	where $J\subseteq \{1,\hdots, q\}$ consists of $k$ elements, see Proposition \ref{critpoints_general}. We write $J=\{j_{i_1},\hdots,j_{i_k}\}$ to enumerate the elements in $J$.
	For $i,l\in \NN$ with $i\leq l$ we denote by $e_i^{(l)}$ the $i$-th standard unit vector of dimension $l$ (i.e., it has $l$ entries, the $i$-th entry is $1$ and all other entries are $0$).
	Now we define
	\begin{align*}
	W_1:&=\sum_{i=1}^k e^{(d_1)}_iv_{j_i}^T(XX^T)^{-\frac 12},\displaybreak[2]\\
	W_l:&=\sum_{i=1}^k e^{(d_l)}_i(e^{(d_{l-1})}_i)^T  \quad \text{ for } l=2,\hdots, N-1, \displaybreak[2]\\
	W_N:&=\sum_{i=1}^k \sigma_{j_i} u_{j_i} (e^{(d_{N-1})}_i)^T.
	\end{align*}
	Since $k\leq r$ this is well defined and one easily checks that $$W_N\cdots W_1=\sum_{j\in J}\sigma_j u_j v_j^T(XX^T)^{-\frac 12}=W$$ and that the conditions in  \ref{critcond-general2} are fulfilled (recall that  $Q=\sum_{i=1}^q \sigma_iu_iv_i^T$ ).
\end{proof}

Let us now analyze the Hessian of $L^N$ in critical points.

\begin{proposition}\label{prop:critical-LNL1-general}  
Let $(W_1,\hdots,W_N)$ be a critical point of $L^N$ such that $W = W_N \cdots
W_1$ has $\operatorname{rank}(W) = k$. If $W$ is not a global optimum of $L^1$
on $\mathcal{M}_k$ then $(W_1,\hdots,W_N)$ is a strict saddle point of $L^N$.
\end{proposition}
\begin{proof} Since $(W_1,\hdots,W_N)$ is a critical point of $L^N$, the matrix
$W = W_N \cdots W_1$ is a critical point of $L^1$ restricted to $\mathcal{M}_k$,
by Proposition~\ref{prop:critical-LN-general} (a). Since $W$ is not a global
optimum of $L^1$ on $\mathcal{M}_k$ it must be a strict saddle point of $L^1$ on
$\mathcal{M}_k$, by Proposition~\ref{prop:strict_saddle_property_general}.
Therefore, there exists $Z \in T_W(\mathcal{M}_k)$ such that (for some
Riemannian metric $g$) it holds $g(\Hess^g L^1(W) Z,Z) < 0$. Write $Z = WA + BW$
and choose again the curve \eqref{def:gammaZ-global} with $V_1 = W_1 A$ and $V_N
=   B W_N$. Then 
\[
	\left.\frac{d^2}{dt^2} L^1(\gamma(t))\right|_{t = 0} = g_W(\Hess^g L^1(W) Z,Z) < 0.
\]
On the other hand
\begin{align*}
  0 & > \left.\frac{d^2}{dt^2} L^1(\gamma(t))\right|_{t = 0} = \left.\frac{d^2}{dt^2} L^N(W_1 + t V_1,W_2,\hdots,W_{N-1}, W_N + tV_N)\right|_{t = 0} \\
	  &= \langle \Hess L^N(W)(V_1,0,\hdots,0,V_N), (V_1,0,\hdots,0,V_N)\rangle, 
\end{align*}
which implies that $\Hess L^N(W)$ is not positive semidefinite, i.e., has a
negative eigenvalue. In other words, $(W_1,\hdots,W_N)$ is a strict saddle
point.
\end{proof}

We note that the global minimizers of $L^1$ restricted to $\mathcal{M}_k$ for some $k < r$ are not
covered by the above proposition, i.e., the proposition does not identify the corresponding tuples $(W_1,\hdots,W_N)$ (such that the product $W=W_N\cdots W_1$ is a  global minimizer of $L^1$ restricted to $\mathcal{M}_k$) as strict saddle points of $L^N$. (In the language of \cite{TragerKohnBruna19} such points are called spurious local minima and 
they may lead to saddle points of $L^N$, see also Propositions 6 and 7 in \cite{TragerKohnBruna19}.)
The above proposition does not exclude that such points correspond to non-strict saddle points
of $L^N$. In fact, in the special case of $k=0$, the point $(0,\hdots,0)$ is indeed
not a strict saddle point if $N \geq 3$ as shown in the next result, which
extends \cite[Corollary 2.4]{kawag16} to the situation that $X X^T$ does not necessarily need to have distinct eigenvalues.

\begin{proposition} 
	If $X Y^T \neq 0$, the point $(0,\hdots,0)$ is a saddle point of $L^N$, which is strict if $N=2$ and not
	strict if $N \geq 3$.
\end{proposition}
\begin{remark}
Note that if $X Y^T = 0$, then $(0,\hdots,0)$ is a global minimum of $L^N$.
\end{remark}
\begin{proof} 
For convenience, we give a different proof than the one in \cite[Corollary 2.4]{kawag16}.
	It is easy to see that $\nabla_{W_j} L^N(0,\hdots,0) = 0$ for every $j=1,\hdots,N$ so that
	$(0,\hdots,0)$ is a critical point of $L^N$.
	Consider a tuple $(V_1,\hdots,V_N)$ of matrices, set $Z = V_N \cdots V_1$ and 
	\[
	\gamma(t) = (tV_N)\cdot (tV_{N-1}) \cdots (tV_1) = t^N Z.
	\]
	Note that by \eqref{L1-curve-general}
	\begin{align*}
	L^N(tV_1,\hdots,tV_N))  & = L^1(\gamma(t)) = \frac{1}{2} 
	\tr(\gamma(t)^T\gamma(t)XX^T-2\gamma(t)XY^T+YY^T) \\
	& =  \frac{1}{2} t^{2N} \tr(Z^T Z X X^T) - t^{N} \tr(Z X Y^T)+ \frac{1}{2} \tr(YY^T).
	\end{align*}
	Hence,
	$$
    \frac{d^2}{dt^2}L^N(tV_1,\hdots,tV_N)) 
      = \frac{d^2}{dt^2}L^1(\gamma(t)) 
      = N(2N-1) t^{2N-2} \tr(Z^T Z X X^T) - N(N-1) t^{N-2} \tr(Z X Y^T).
	$$
  Note that $\tr(Z^T Z X X^T) \geq 0$. Recall also that $N\geq 2$. Since $X Y^T
  \neq 0$, there clearly exist matrices $(V_1,\hdots,V_N)$ such $\tr(Z X Y^T) >
  0$ for $Z= V_N \cdots V_1$, so that $L^N(tV_1,\hdots,tV_N) < L^N(0,\hdots,0)$
  for small enough $t$. Hence, $(0,\hdots,0)$ is not a local minimum, but a
  saddle point. Moreover,
	\begin{align*}
	& \langle \Hess L^N(0,\hdots,0)(V_1,\hdots,V_N), (V_1,\hdots,V_N) \rangle
	= \left.\frac{d^2}{dt^2} L^1(\gamma(t))\right|_{t = 0} \\
	&= \left\{ \begin{array}{ll} -2 \tr(Z X Y^T) & \mbox{ if } N = 2\\
	0 & \mbox{ if } N\geq 3
	\end{array} \right.
	\end{align*}
	If $N=2$, we can find matrices $V_1,V_2$ such that $\tr(Z X Y^T) > 0$ for $Z = V_2 V_1$ so that
	$(0,0)$ is a strict saddle for $N=2$. If $N \geq 3$ it follows that $\Hess L^N(0,\hdots,0) = 0$ so that
	$(0,\hdots,0)$ is not a strict saddle.
\end{proof}

In the case $N=2$, the following result implies that all local minima of $L^2$ are global and
all saddle points of $L^2$ are strict.

\begin{proposition}\label{prop:N=2strictsaddles-general} Let $N=2$ and $k < \min\{r,q\}$, where  $r=\min\{d_0,d_1,d_2\}$ and $q = \operatorname{rank}(Q)$. 
	Let $W$ be a global minimum of $L^1$ restricted to $\mathcal{M}_k$, i.e., 
		$	W =\sum_{j\in J}\sigma_j u_j v_j^T(XX^T)^{-\frac 12},$
	where $|J|=k$ and $ \{\sigma_j:j\in J\}=\{\sigma_1,\hdots, \sigma_k\}$. Then any critical point $(W_1,W_2) \in \RR^{d_1 \times d_0} \times \RR^{d_2 \times d_1}$ such that 
	$W_2 \cdot W_1 = W$ is a strict saddle point of $L^2$.
\end{proposition}
\begin{proof} 
	For $\kappa\in \RR\setminus\{0\}$ and  $u\in \RR^{d_2}$, $v\in \RR^{d_1}$, $w\in \RR^{d_0}$ with $u^Tu=1$ and $v^Tv=1$  
	we define
	the curve
	\[
	\gamma(t) = (W_2 + t\kappa uv^T)\cdot(W_1+ t \kappa^{-1}vw^T) = W + t (\kappa uv^T W_1 +  \kappa^{-1}W_2 vw^T) + t^2 uw^T.
	\]
	Then by \eqref{L1-curve-general}
	\begin{align*}
	& L^2(W_1 + t\kappa^{-1}vw^T, W_2 + t\kappa uv^T)  = L^1(\gamma(t)) =  \frac{1}{2} 
	\tr(\gamma(t)^T\gamma(t)XX^T-2\gamma(t)XY^T+YY^T).
	\end{align*}
	We compute 
	\begin{align*}
	\gamma(t)^T\gamma(t)=t^2 &\left(W_1^Tvu^TW_2vw^T+(W_1^Tvu^TW_2vw^T)^T+ \kappa^{2}W_1^Tvv^TW_1+ \kappa^{-2}wv^TW_2^TW_2vw^T+wu^TW\right. \\ & \left.+W^Tuw^T\right) + \text{ terms which are not of order } t^2.
	\end{align*}
	It follows that 
	\begin{align*}
	\left.\frac{d^2}{dt^2}L^1(\gamma(t))\right|_{t=0} =& \tr   \left((  W_1^Tvu^TW_2vw^T+(W_1^Tvu^TW_2vw^T)^T \right. + \kappa^{2}W_1^Tvv^TW_1  \\ & \left.  + \kappa^{-2}wv^TW_2^TW_2vw^T\right.  \left. +wu^TW + W^Tuw^T )XX^T -2uw^TXY^T   \right).
	\end{align*}	
Let us now choose the vectors $u,v,w$.	
	Note that since $(W_1,W_2)$ is a critical point of $L^2$, we have 
	$W_2^T(WXX^T-YX^T)=0$ by  Lemma~\ref{L:PREL}, point~1, and hence
	$$
	W_2^T(\sum_{j\in J}\sigma_j u_j v_j^T-\sum_{j=1}^q\sigma_j u_j v_j^T)(XX^T)^{\frac 12}=0.
	$$
	Since $XX^T$ has full rank it follows that for any $j_0\in \{1,\hdots,q\}\setminus J$ we have $W_2^Tu_{j_0}=0$.
	Since $k<q$ such a $j_0$ exists. Thus we may choose  $j_0\in \{1,\hdots,q\}\setminus J$ and define $u=u_{j_0}$ and $w=(XX^T)^{-\frac 12}v_{j_0}$.
	
	If the kernel of $W_1^T$ is trivial then $d_1\leq d_0$ and $W_1$ has rank $d_1$.
	It follows that then the kernel of $W_2$ cannot be trivial since otherwise $W_2$ would be injective and the rank of $W=W_2W_1$ would be $d_1$. But the rank of $W$ is $k<r\leq d_1$. 
	Hence we may choose $v$ as follows:
	We choose $v$ to be an element of the kernel of $W_1^T$ with $\|v\|_2 =1$ if such a $v$ exists and otherwise we choose  $v$ to be an element of the kernel of $W_2$ with $\|v\|_2 = 1$.
	
	With these choices for $u,v,w$ we have $W_1^Tvu^TW_2vw^T = 0$ and $W^Tuw^T= W_1^TW_2^Tu_{j_0} w^T = 0$ so that 
	\begin{align*}
	\left.\frac{d^2}{dt^2}L^1(\gamma(t))\right|_{t=0} =\tr   \left((  \kappa^{2}W_1^Tvv^TW_1    + \kappa^{-2}wv^TW_2^TW_2vw^T)XX^T -2uw^TXY^T   \right),
	\end{align*}
	where at least one of the terms $W_1^Tvv^TW_1 $ and $ wv^TW_2^TW_2vw^T$ vanishes.
	We have 
	$$
	\tr(uw^TXY^T)=w^TXY^Tu=v_{j_0}^T(XX^T)^{-\frac 12}XY^Tu_{j_0}=v_{j_0}^TQ^Tu_{j_0}= v_{j_0}^T\sum_{j=1}^q\sigma_j v_ju_j^Tu_{j_0}=\sigma_{j_0}.
	$$
	Hence 
		\begin{align*}
	\left.\frac{d^2}{dt^2}L^1(\gamma(t))\right|_{t=0} =\kappa^{2}\tr   \left( W_1^Tvv^TW_1XX^T\right)    + \kappa^{-2}\tr\left(wv^TW_2^TW_2vw^TXX^T\right) -2\sigma_{j_0}.
	\end{align*}
	Since $\sigma_{j_0}>0$ and since at least one of the terms $W_1^Tvv^TW_1XX^T$ and 
	$wv^TW_2^TW_2vw^TXX^T$ vanishes, we can always choose $\kappa>0$ such that 
	$	\left.\frac{d^2}{dt^2}L^1(\gamma(t))\right|_{t=0}<0$.

	As in the proof of Propositon \ref{prop:strict_saddle_property_general}, this shows that $(W_1,W_2)$ is a strict saddle point.
\end{proof}

\subsection{Convergence to global minimizers}

We now state the main result of this article about convergence to global minimizers. Part (b) of Theorem \ref{thm:main-general} for two layers generalizes 
a result in \cite[Section 4]{chlico18}, where it is assumed that $d_x \geq d_y$ and $d_y \leq \min\{d_1,\hdots,d_{N-1}\}$ on top of
some mild technical assumptions on matrices formed with $X$ and $Y$ (see Assumptions 3 and 4 of \cite{chlico18}).

\begin{theorem}\label{thm:main-general}
Assume  that  $XX^T$ has full rank, let $q = \operatorname{rank}(Q)$, $r = \min\{d_0,\hdots,d_N\}$ and let $\bar{r} := \min\{q,r\}$.
\begin{itemize}
\item[(a)] 
For almost all initial values $W_1(0),\hdots, W_N(0)$, the flow
(\ref{gradflow}) converges to a critical point $(W_1,\hdots, W_N)$ of $L^N$ such
that $W:=W_N\cdots W_1$ is a global minimizer of $L^1$ on the manifold
$\mathcal{M}_k$ of matrices in $\RR^{d_N\times d_0}$ of rank $k:=\rank(W)$, where
$k$ lies  between $0$ and  $\bar{r}$ and depends  on the initialization.
\item[(b)] 
For $N=2$,	for almost all initial values  $W_1(0),\hdots, W_N(0)$, the flow
(\ref{gradflow}) converges to a global minimizer of $L^N$ on $\RR^{d_0\times
d_1}\times\hdots\times \RR^{d_{N-1}\times d_N}$. 
\end{itemize}
	
\end{theorem}

By {\em ``for almost all initial values''} we mean that there exists a set $N$
with Lebesgue measure zero in $ \RR^{d_0\times d_1}\times\hdots\times
\RR^{d_{N-1}\times d_N} $ such that the statement holds for all initial
values outside $N$.

\begin{proof}
By Theorem \ref{globconv}, under the flow \eqref{gradflow}, the curve
$(W_1(t),\hdots, W_N(t))$ converges to some critical point $(W_1,\hdots, W_N)$
of $L^N$. Let $k$ be the rank of $W:=W_N\cdots W_1$. Then $k \leq r$, by
construction. But also $k \leq q$ because, if $(W_1, \ldots, W_N)$ is a critical
point of $L^N$, then $W$ as above is a critical point of $L^1$ restricted to
$\mathcal{M}_k$, by Proposition~\ref{prop:critical-LN-general} (a). But we know
that there are no critical points of $L^1$ in $\mathcal{M}_k$ with rank larger
than $q$ because of Proposition~\ref{critpoints_general} (a). This proves that
$0 \leq k \leq \bar{r}$.

If $W$ is not a global minimizer of $L^1$ restricted to $\mathcal{M}_k$, then
$(W_1, \ldots, W_N)$ must be strict a saddle point of $L^N$ because of
Proposition~\ref{prop:critical-LNL1-general}. 
By Theorem~\ref{thm:avoid-saddles} only a negligible set of initial values
$W_1(0), \ldots, W_N(0)$ converges to a strict saddle point of $L^N$. All other
initial values therefore converge to a limit point $(W_1, \ldots, W_N)$ for
which $W = W_N \cdots W_1$ is a global minimizer of $L^1$ restricted to
$\mathcal{M}_k$. This proves part (a) of Theorem~\ref{thm:main-general}. Note
that $W$ being a minimizer of $L^1$ restricted to $\mathcal{M}_k$ does not imply
that the corresponding matrix tuple $(W_1, \ldots, W_N)$ is a minimizer of
$L^N$. This happens only if the rank of $W$ is as large as possible, i.e., if $k
= r$.

In the case $N=2$, Proposition~\ref{prop:N=2strictsaddles-general} shows that if
the limit $(W_1,W_2)$ has the property that $W = W_2 W_1$ is a global minimizer
of $L^1$ in $\mathcal{M}_k$ but $k < \bar{r}$, then $(W_1,W_2)$ is a strict
saddle point of $L^2$. But we already know that the set of initial values
$W_1(0), W_2(0)$ that converge to a strict saddle point of $L^2$ is negligible.
We conclude that generically the solution of \eqref{gradflow} converges to a
limit $(W_1,W_2)$ for which $W=W_2W_1$ is a global minimizer of $L^1$ on
$\mathcal{M}_k$ with $k = \bar{r}$, which implies that $(W_1, W_2)$ is a global
minimizer of $L^N$.
\end{proof}

Balanced initial values $(W_1(0),\hdots,W_N(0))$ are of special interest as they
give rise to a Riemannian gradient flow on $\mathcal{M}_k$. Unfortunately,
Theorem~\ref{thm:main-general} does not allow to make conclusions about the set
of balanced initial values because this is a set of Lebesgue measure zero in
$\RR^{d_1 \times d_0} \times \cdots \times \RR^{d_{N} \times d_{N-1}}$. We are
nevertheless able to derive the following convergence result by applying
Theorem~\ref{thm:avoid-saddles} to the Riemannian gradient flow on
$\mathcal{M}_k$.

\begin{theorem}\label{thm:main_result_balanced} Assume that $XX^T$ has full rank and let $N\geq 2$. 
%
Then for almost all initializations $W(0)\in \RR^{d_y\times d_x}$ on $\mathcal{M}_k$,
the  flow $W(t)$  on  $\mathcal{M}_k$ solving \eqref{riemflow} (cf. Theorem \ref{thm:balanced_flow_on_M_k}) converges to a global minimum of
$L^1$ restricted to $\mathcal{M}_k$ or to a critical point on some
$\mathcal{M}_{\ell}$, where $\ell<k$. Note that for $k > \operatorname{rank}(Q)$ there is no global minimum of $L^1$ on $\mathcal{M}_k$ so that
then the second option applies.
Here  {\em ``for almost all $W(0)$''}
means for all $W(0)$ up to a set of measure zero. 
\end{theorem}

\begin{proof}
	
As in the proof of Theorem \ref{thm:balanced_flow_on_M_k} we  write $W(t)=W_N(t)\cdots W_1(t)$, where the $W_i(t)\in \RR^{d_{i} \times d_{i-1}}$ (for some $d_i\geq k$ with $d_0=d_x$ and $d_N=d_y$) are balanced and solve \eqref{gradflow}.
Since the tuple $(W_1(t),\hdots,
W_N(t))$ satisfies \eqref{gradflow}, it converges to a critical point of $L^N$ by
Theorem~\ref{globconv}, hence the flow $W(t)$ converges for all initial values to
some $W$ that is a critical point of $L^1$ on some $\mathcal{M}_{\ell}$, where
$\ell\leq k$; see Proposition \ref{prop:critical-LN-general}. Since by
Proposition \ref{prop:strict_saddle_property_general} all critical points of
$L^1$ on $\mathcal{M}_k$ except for the global minimizers are strict saddle
points, the  claim follows from Theorem  \ref{thm:avoid-saddles}, whose
assumptions are satisfied by Proposition \ref{C1metric} and Theorem 
\ref{thm:balanced_flow_on_M_k}.
\end{proof}


The reason why we cannot choose $k=\bar{r}$ in Theorem \ref{thm:main-general} (a), i.e., state that the flow for $N \geq 3$ converges to the global minimum of 
$L^1$ on $\mathcal{M}_{\bar{r}}$ for almost all initializations is that not all saddle points of $L^N$ are necessarily strict for $N \geq3$. Nevertheless, 
we conjecture a more precise version of the previous result in the spirit of Theorem~\ref{optconv1}.
Part (a) below is a strengthened version of the overfitting conjecture in \cite{chlico18}, where additional assumptions are made.

\begin{conjecture}	\label{con:main}
Assume  that    $XX^T$  has full rank. 
	\begin{itemize}
		\item[(a)] The statement in Theorem \ref{thm:main-general} (b) also holds for $N>2$.
		\item[(b)] Consider the autoencoder case $X=Y$ and let $d=d_0=d_N$.
		Let 
		$r = \min_{i=1,\hdots,N} d_i$.
		Let $\lambda_1\geq\hdots\geq\lambda_d$ be the eigenvalues of $XX^T$ and let $u_1,\hdots, u_d$ be corresponding
		orthonormal eigenvectors. Let $U_r$  be the matrix with columns $u_1,\hdots, u_r$. Assume  that $\lambda_r>\lambda_{r+1}$. Assume further that $W(0)U_r$ has rank $r$ and that for all $i \in \{1,\hdots, r\}$ we have 
		\begin{equation}\label{eqn:conjecture}
		u_i^T W(0)u_i>0,
		\end{equation}
		where $W(t)=W_N(t)\cdots W_1(t)$. Then $W(t)$ converges to $\sum_{i=1}^r u_iu_i^T$.
\item[(c)] In  Theorem \ref{thm:main_result_balanced}, if $k\leq \rank(Q)$, convergence to a critical point on some $\mathcal{M}_{\ell}$, where $\ell<k$, happens only for a set of  initial values that has measure zero.
	\end{itemize}

\end{conjecture}

\begin{remark}\label{rem:conjecturecondition}
	Without the condition that $
	u_i^T W(0)u_i>0
	$ for all $i \in \{1,\hdots, r\}$, the above conjecture (b) is wrong.
\end{remark}
\begin{proof}
	Indeed, in  the autoencoder  case with $N=2$ and  $r=1$ with  $W_1(0)=u_1^T$ and $W_2(0)=-u_1$ (which is a balanced starting condition and $W(0)U_1$ has obviously rank $1$), we  show that $W_1, W_2$ and $W$ all converge to the zero-matrix of their respective size.
	Write $W_1=(\alpha_1,\hdots, \alpha_d)$ and 
	$W_2=(\beta_1,\hdots, \beta_d)^T$.
	We may assume that $XX^T$ is a diagonal matrix with entries  $\lambda_1\geq\hdots\geq\lambda_d>0$. (In particular,   the $u_i$ are given by the standard unit vectors $u_i=e_i$.)	Then the system (\ref{gradflow}), see also (\ref{flow2layers}), becomes
	\begin{align}
	\begin{split}
	\dot \alpha_j=-\lambda_j\alpha_j\sum_{i=1}^d \beta_i^2 +\lambda_j\beta_j,\ \ \ \alpha_j(0)=\delta_{j1},
	\\
	\dot \beta_j=-\beta_j\sum_{i=1}^d \lambda_i\alpha_i^2 +\lambda_j\alpha_j,\  \ \ \beta_j(0)=-\delta_{j1}.
	\end{split}
	\end{align}
	This system is solved by the following functions:
	\begin{align}
	\begin{split}
	\alpha_1(t)=\frac{1}{\sqrt{2e^{2\lambda_1 t}-1}}, \ \ \  
	\alpha_j(t)=0 \text{ for all } j\geq 2, \\
	\beta_1(t)=\frac{-1}{\sqrt{2e^{2\lambda_1 t}-1}},\ \  \  
	\beta_j(t)=0 \text{ for all  } j\geq 2. \\
	\end{split}
	\end{align}
	Obviously, all $\alpha_j$ and $\beta_j$ converge to $0$
	as $t$ tends to infinity. From this the claim follows.
	(By Theorem \ref{saddlepoints}, this equilibrium is not stable, so this behavior may not be obvious  in numerical simulations.)
\end{proof}

\section{Numerical results}\label{sec:numerics}

We numerically study the convergence of gradient flows in the linear supervised learning setting as a proof of concept of the convergence results presented above in both the general supervised learning case and the special case of autoencoders. Moreover, in the autoencoder case the experiments also computationally explore the conjecture (Conjecture \ref{con:main}) of the manuscript.

\subsection{Autoencoder case}
We study the gradient flow \eqref{gradflow} in the autoencoder setting, where $Y=X\in \RR^{d_x\times m}$ in \eqref{eqNNMinimization} for  different dimensions of $X$ (i.e., $d_x$ and $m$) 
and different values of the number $N$ of layers, where we typically use $N \in \{2, 5, 10, 20\}$.
A Runge-Kutta method (RK4) is used to solve the gradient flow differential equation with appropriate step sizes $t_n = t_0 + nh$ for large $n$ and $h\in (0,1)$. The experiments fall into two categories based on initial conditions of the gradient flow: a) {\em balanced} -- where the balanced conditions are satisfied; and b) {\em non-balanced} -- where the balanced conditions are not satisfied. Under a) we investigate the general case in these balanced conditions where condition \eqref{eqn:conjecture} of Conjecture~\ref{con:main} is satisfied, but also a special case were the balanced conditions are satisfied but condition \eqref{eqn:conjecture} of Conjecture~\ref{con:main} is not satisfied. 

The results in summary, considering $W = W_N\cdots W_1$ as the limiting solution of the gradient flow, that is $W = \lim_{t \to \infty} W(t)$, where $W(t) = W_N(t)\cdots W_1(t)$: We show that with balanced initial conditions, the solutions of the gradient flow converges to $U_rU_r^T$, where the columns of $U_r$ are the $r$ eigenvectors corresponding to the $r$ largest eigenvalues of $XX^T$. The convergence rates decrease with an increase in either $d$ or $N$ or both. We see similar results for the non-balanced case.

\subsubsection{Balanced initial conditions}\label{sec:bic}
In this section and Section~\ref{sec:nic} the data matrix $X\in \RR^{d_x\times m}$ is generated with columns drawn i.i.d.~from a Gaussian distribution, i.e., $x_i \sim \mathcal{N}(0,\sigma^2 I_{d_x})$, where $\sigma = 1/\sqrt{d_x}$. Random realization of $X$ with sizes $d_x=d$ and $m = 3d$ are varied to investigate different dimensions of the input data, i.e., with $2N\leq d \leq 20N$. For each fixed $d$, the dimensions $d_j$ of the $W_j \in \mathbb{R}^{d_j \times d_{j-1}} $ for $j=1,\ldots, N$ are selected as follows: We set $d_1 = r = [d/2]$, where $[\cdot]$ rounds to the nearest integer,
and put $d_j = [r+(d-r)(j-1)/(N-1)]$, $j=2,\hdots,N$ (generating an integer ``grid'' of numbers between $d_1= r$ and 
$d_N = d_x = d$).

In the first set of experiments, we consider a general case of the balanced initial conditions, precisely $W_{j+1}^T(0)W_{j+1}(0) = W_{j}(0)W_{j}^T(0), \ j=1,\ldots,N-1,$ where condition \eqref{eqn:conjecture} of Conjecture \ref{con:main} is satisfied. The dimensions of the $W_j$ and their initializations are as follows.
Recall, $W_j\in\RR^{d_j\times d_{j-1}}$ for $j=1,\ldots,N$ where $d_N = d_0 = d_x = d$ and $d_1 = r$ is the rank of $W=W_N\cdots W_1$. We randomly generate $d_j \times d_{j}$ orthogonal matrices $V_j$ and then form $W_j(0) = V_j I_{d_j d_{1}}U_{j-1}^T$ for $j = 1,\ldots, N$, where $U_{j} \in \RR^{d_j\times d_1}$ is composed of the $d_1$ columns of $V_j$, and $I_{ab}$ is the (rectangular) $a\times b$ identity matrix.
For all the values of $N$ and the different ranks of $W$ considered, Figure \ref{fig:gen_balance} shows that the limit of $W(t)$ as $t\rightarrow \infty$ is $U_rU_r^T$, where the columns of $U_r$ are $r$ eigenvectors of $XX^T$ corresponding to the largest $r$ eigenvalues of $XX^T$. This agrees with the theoretical results obtained for the autoencoder setting.
\begin{figure}[h]
	\centering
	\includegraphics[width=0.45\textwidth,height=0.3\textwidth]{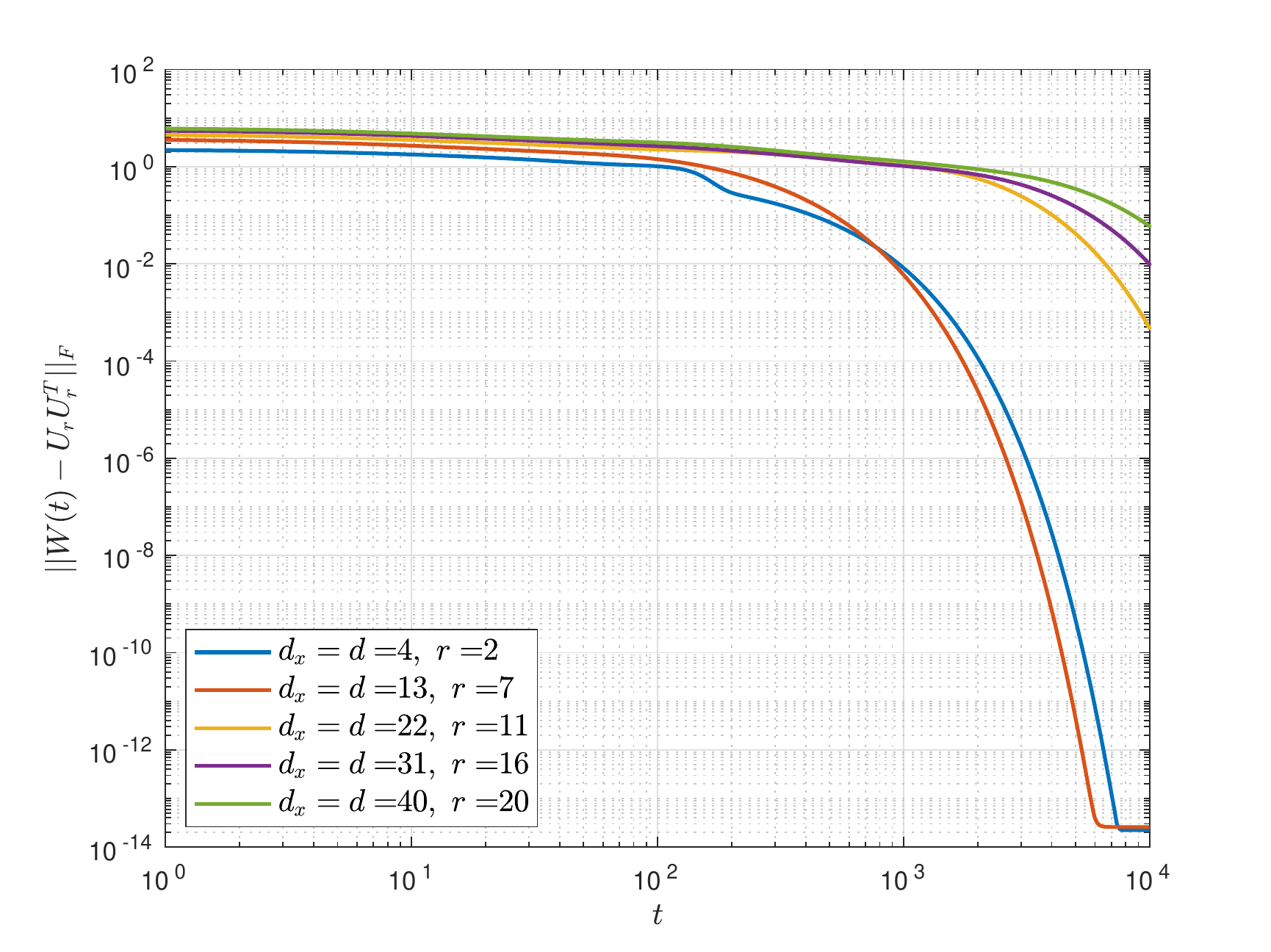}
	\includegraphics[width=0.45\textwidth,height=0.3\textwidth]{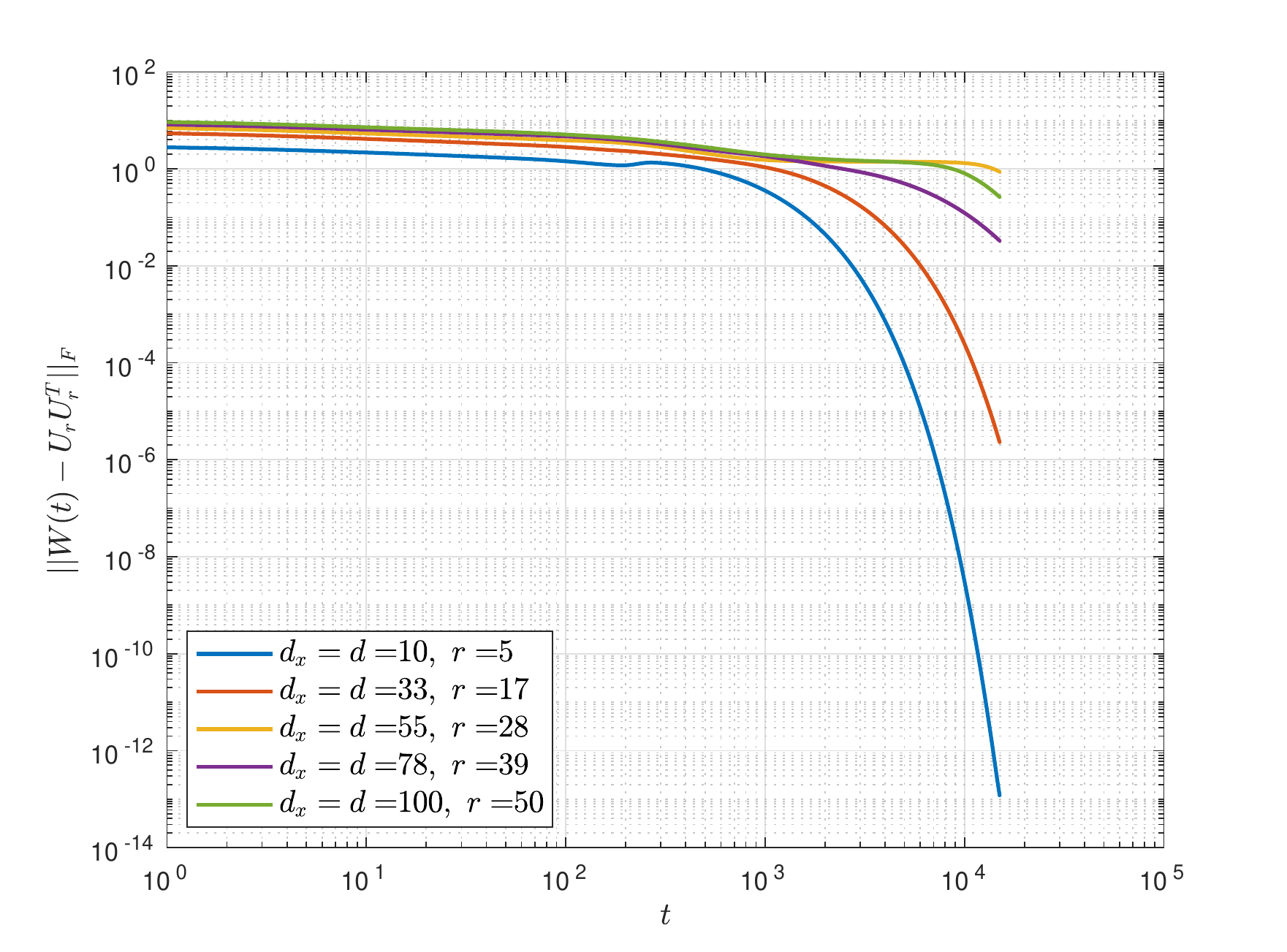} 
	\includegraphics[width=0.45\textwidth,height=0.3\textwidth]{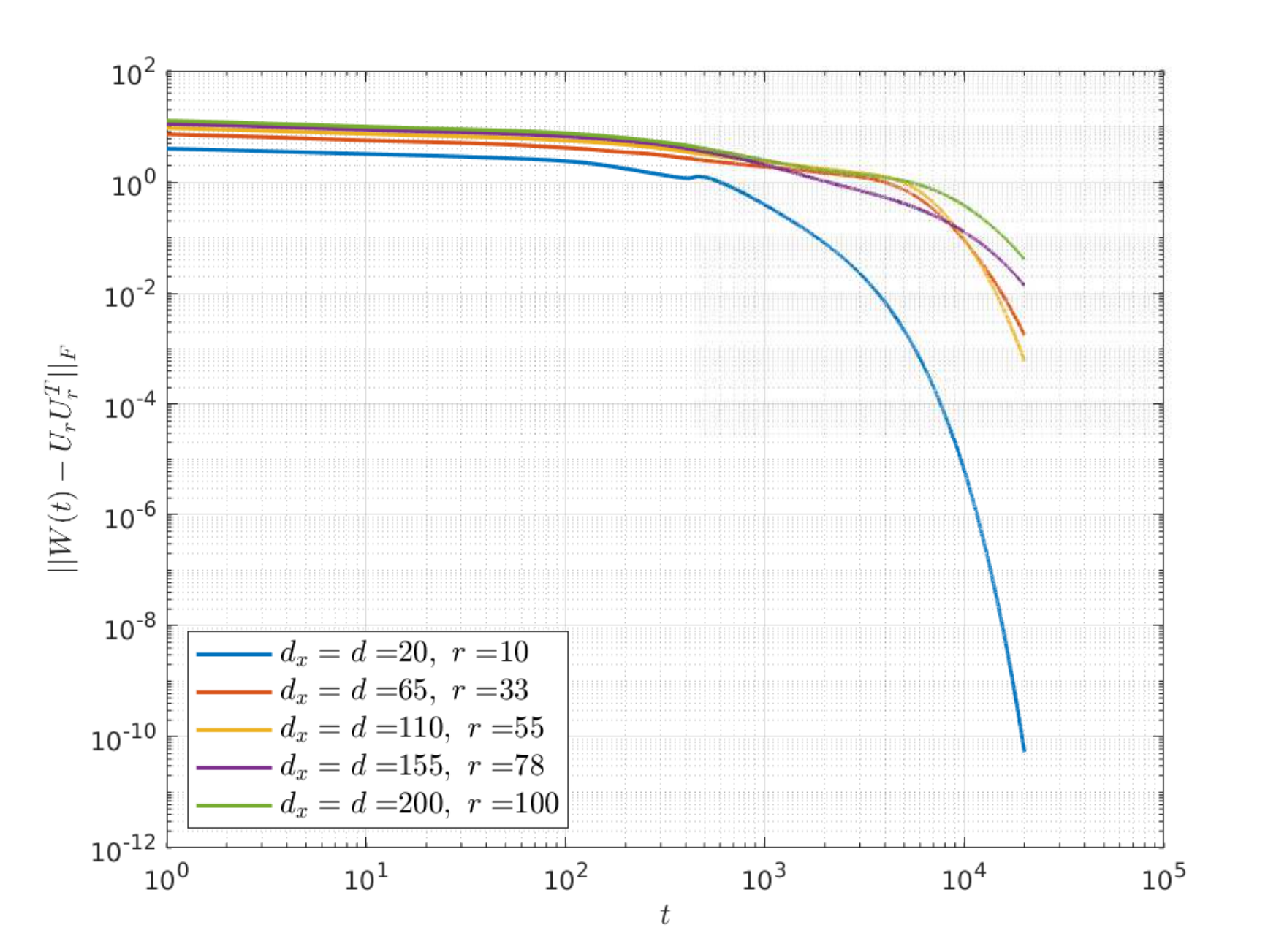} 
	\includegraphics[width=0.45\textwidth,height=0.3\textwidth]{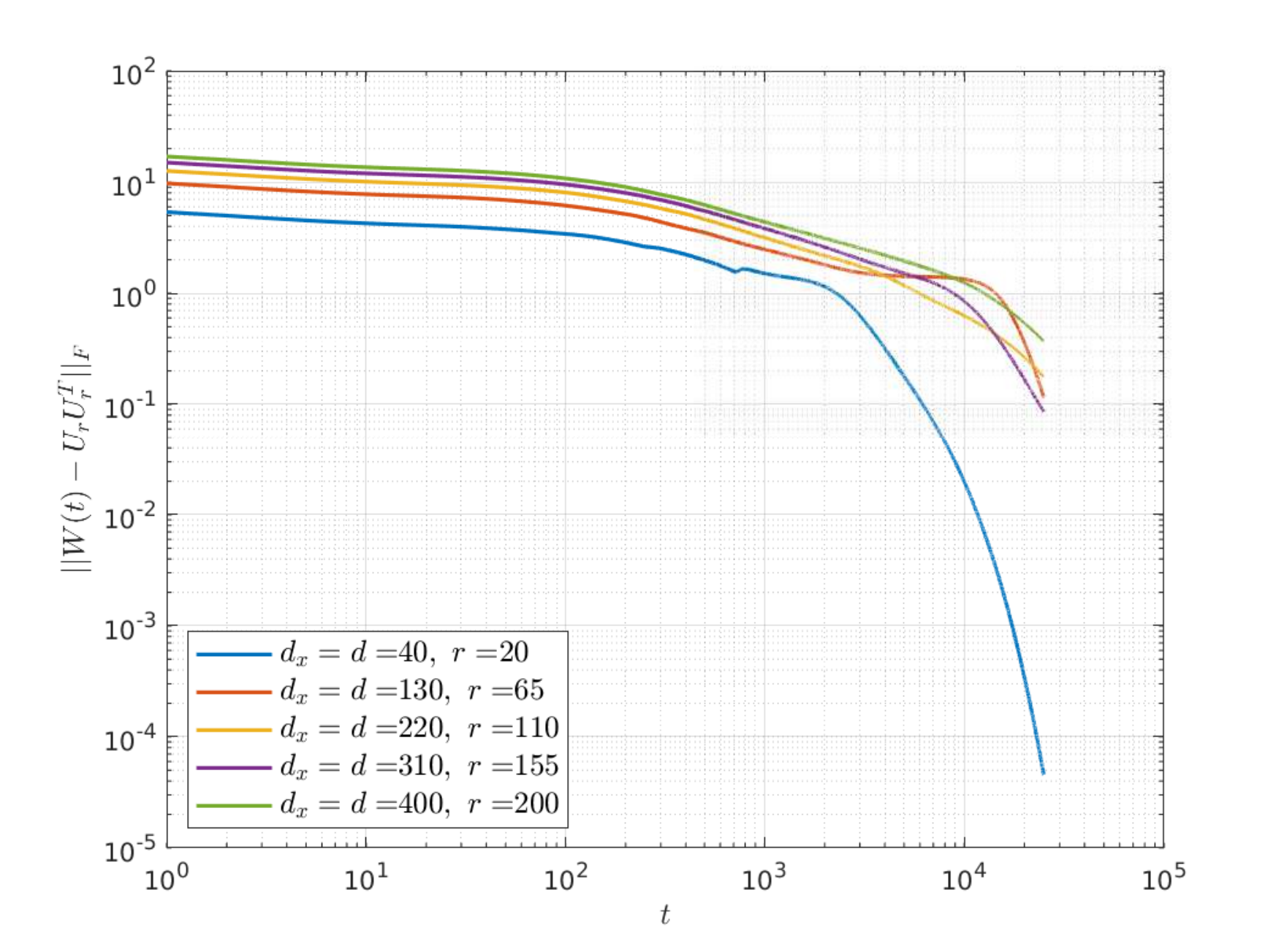}
	\vspace{-1mm}
	\caption{Convergence of solutions for the general balanced case. Error between $W(t)$ and $U_rU_r^T$ for different $r$ and $d$ values. {\em Top left panel}:  $N=2$; {\em top right panel}:  $N=5$; {\em bottom left panel}:  $N=10$; {\em bottom right panel}:  $N=20$.}
	\label{fig:gen_balance}\vspace{-2mm}
\end{figure}

\begin{figure}[h]
	\centering
	\includegraphics[width=0.45\textwidth,height=0.3\textwidth]{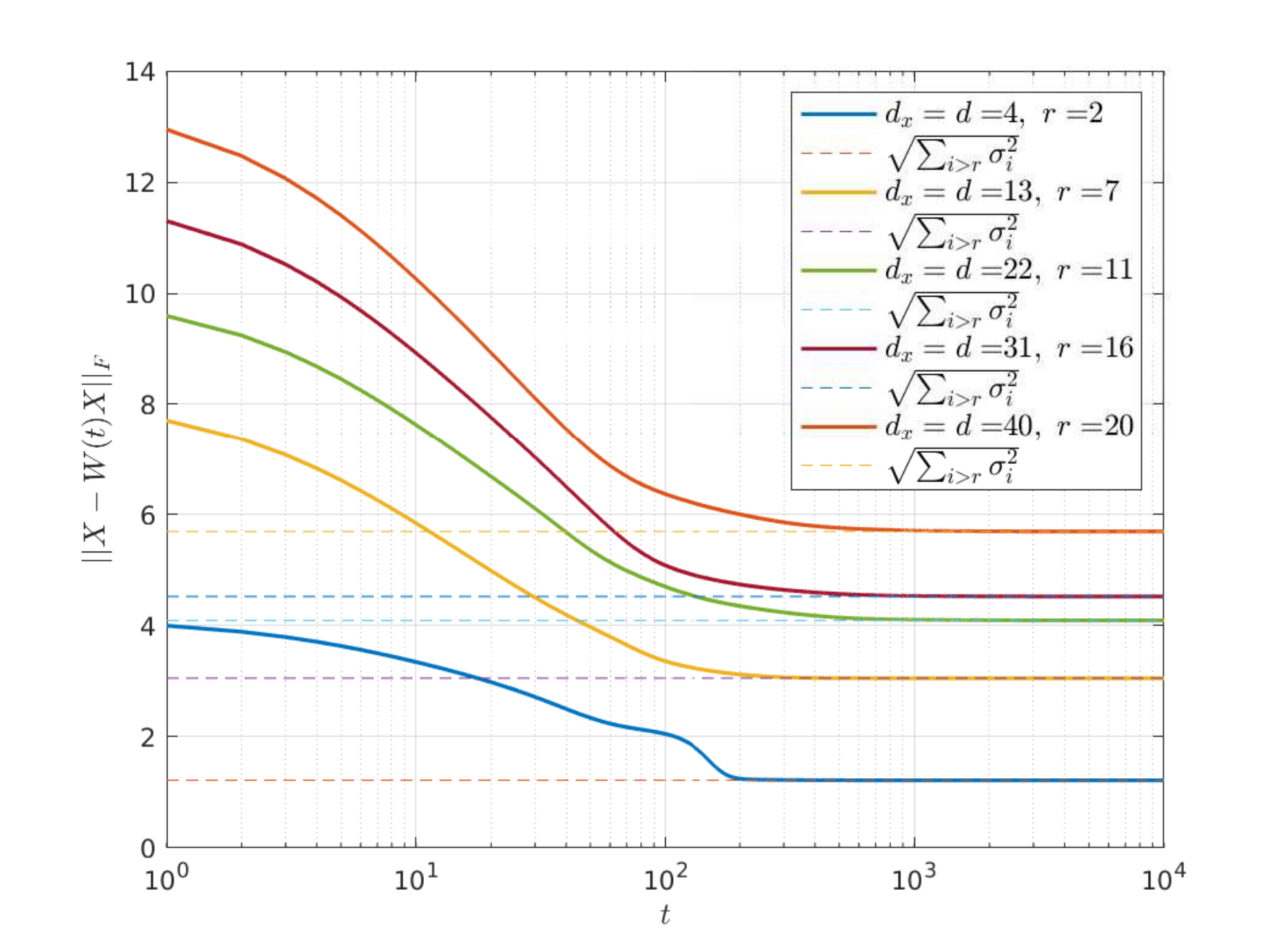} 
	\includegraphics[width=0.45\textwidth,height=0.3\textwidth]{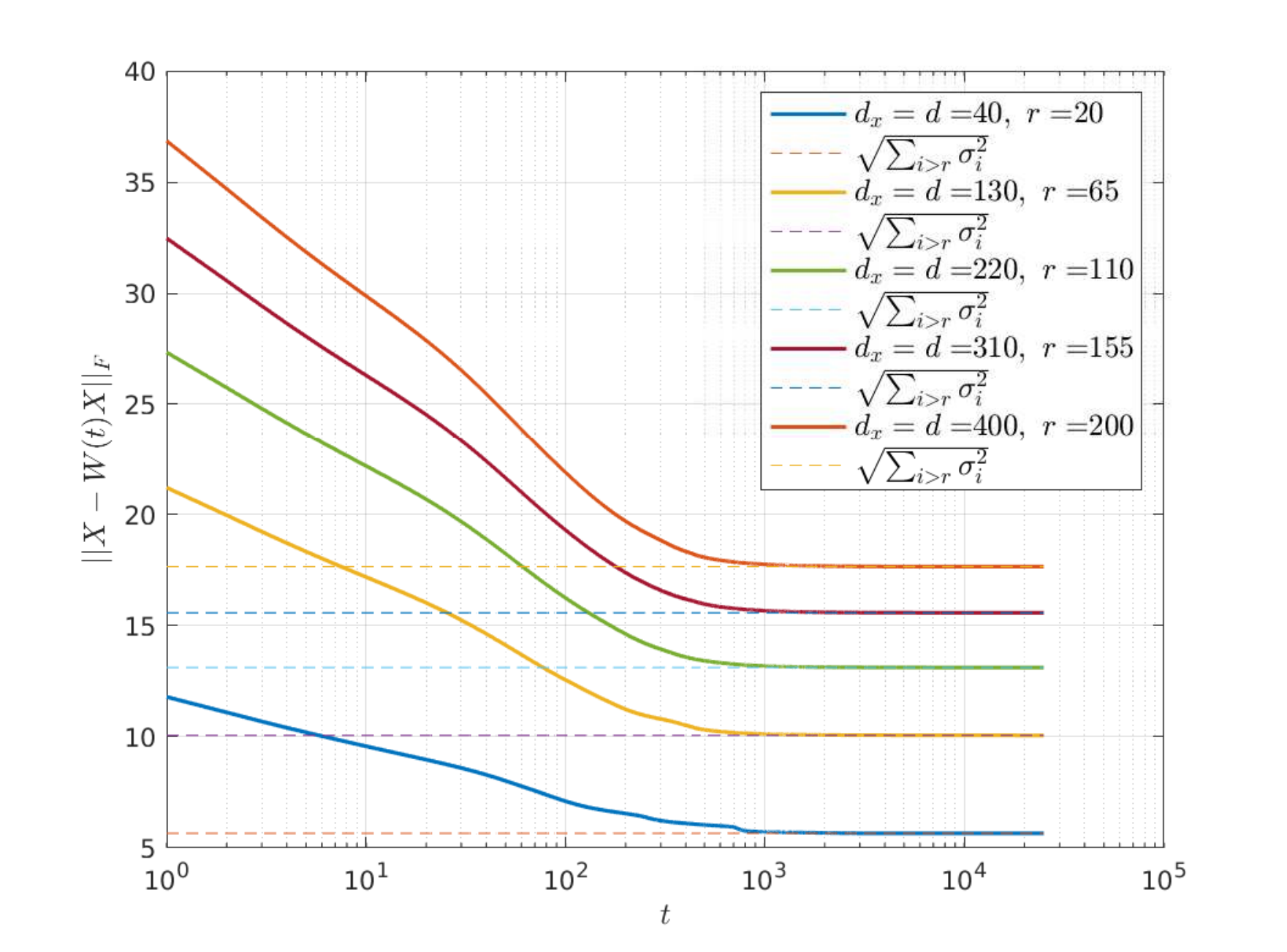} 
	\vspace{-1mm}
	\caption{Convergence of solutions for the general balanced case. Errors between $X$ and $W(t)X$ for different $r$ and $d$ values. {\em Left panel}: $N=2$; {\em right panel}: $N=20$.}
	\label{fig:gen_balance2}\vspace{-4mm}
\end{figure}

In addition, when $W(t)$ converges to $U_rU_r^T$ then $\|X-W(t)X\|_F$ converges to $\sqrt{\sum_{i>r}\sigma_i^2}$. This is also tested and confirmed for $N=2, 5, 10, 20$, but for the purpose of saving space we show results for $N=2$ and $N=20$ in Figure~\ref{fig:gen_balance2}. This depicts convergence of the functional $L^1(W(t))$ to the optimal error, which is the square-root of the sum of the tail eigenvalues of $XX^T$ of order greater than $r$. 
Moreover, in the autoencoder setting when $N=2$ we showed in Lemma \ref{equilpoints} that the optimal solutions are $W_2=W_1^T$. This is also confirmed in the numerics as can be seen in the left panel plot of Figure~\ref{fig:ww1w2diff}.
\begin{figure}[h]
	\centering
	\includegraphics[width=0.45\textwidth,height=0.3\textwidth]{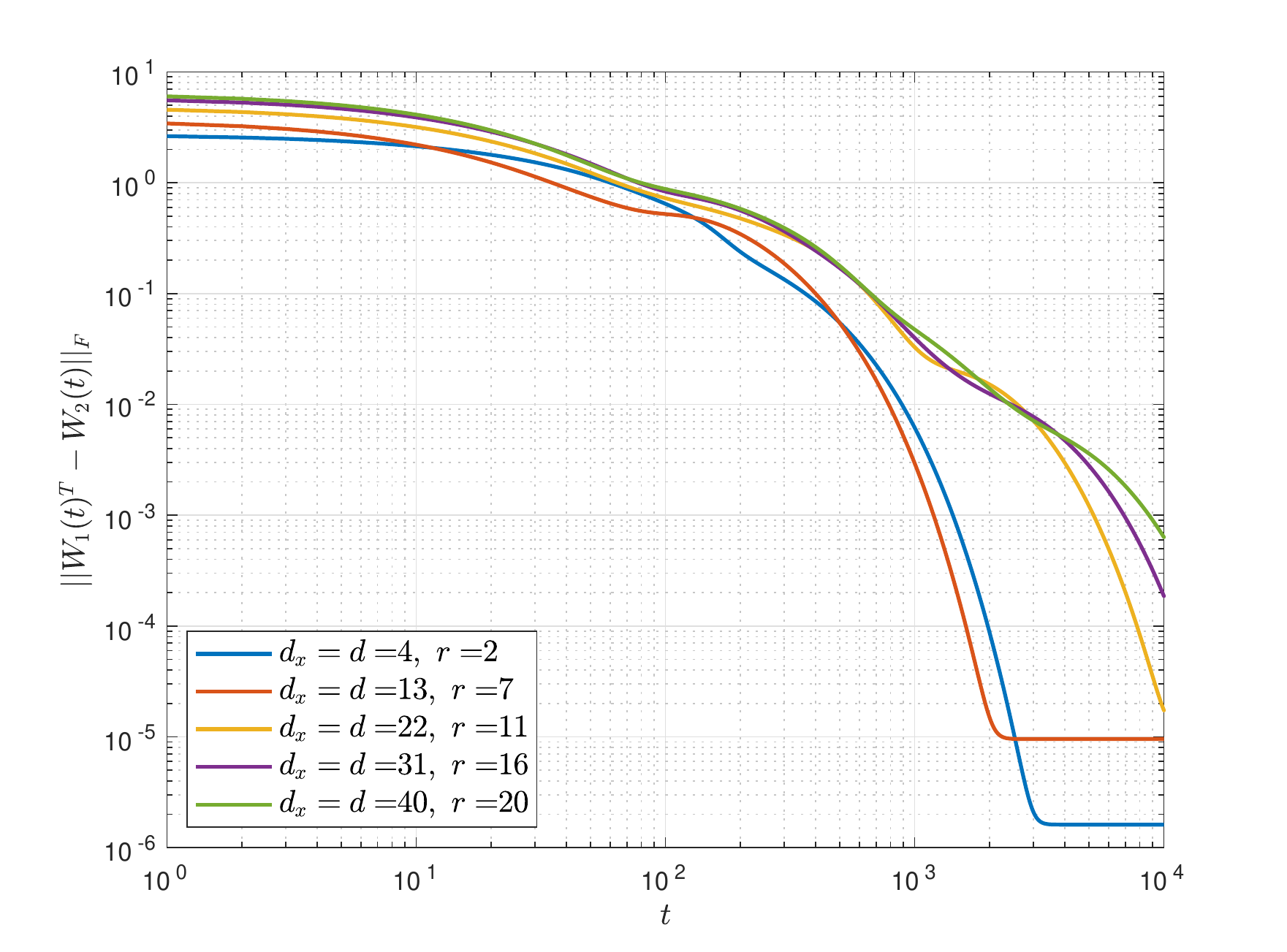}
	\includegraphics[width=0.45\textwidth,height=0.3\textwidth]{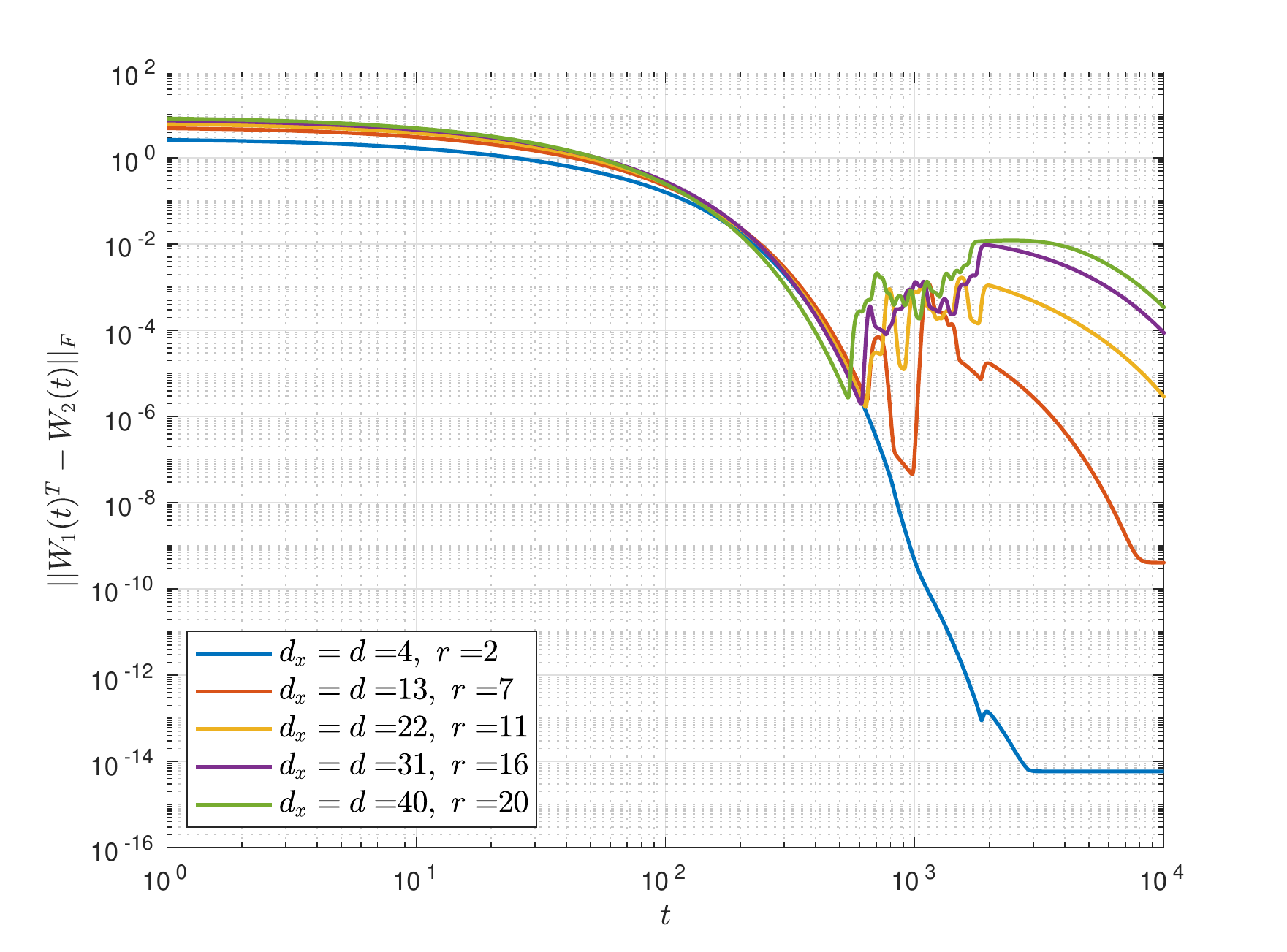}  
	\vspace{-1mm}
	\caption{Difference between $W_1(t)$ and $W_2(t)^T$ in the $N=2$ settings, for {\em left panel:} general balanced case; {\em right panel:} special balanced case.}
	\label{fig:ww1w2diff}\vspace{-2mm}
\end{figure}
\begin{figure}[h]
	\centering 
	\includegraphics[width=0.45\textwidth,height=0.3\textwidth]{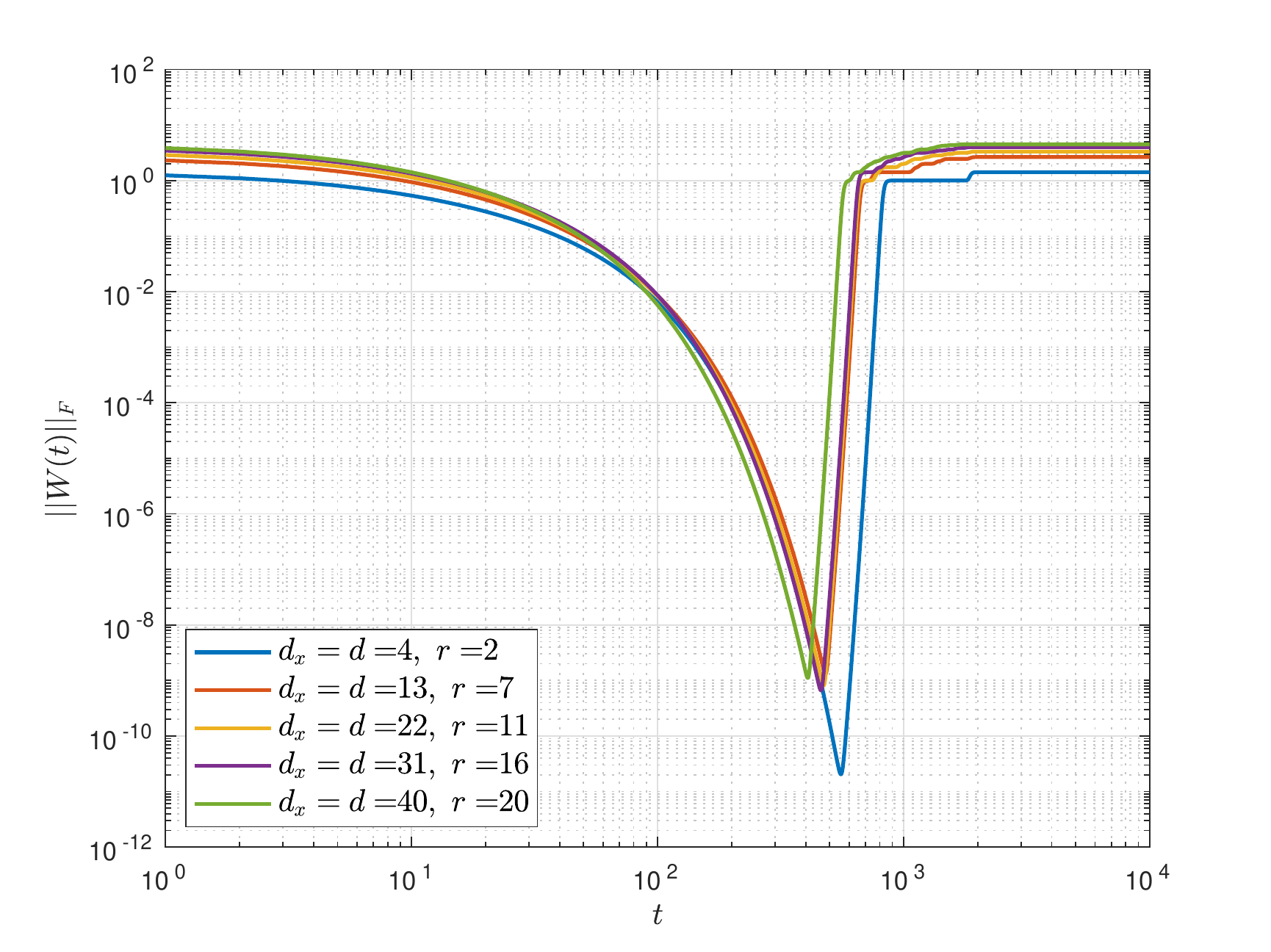} 
	\includegraphics[width=0.45\textwidth,height=0.3\textwidth]{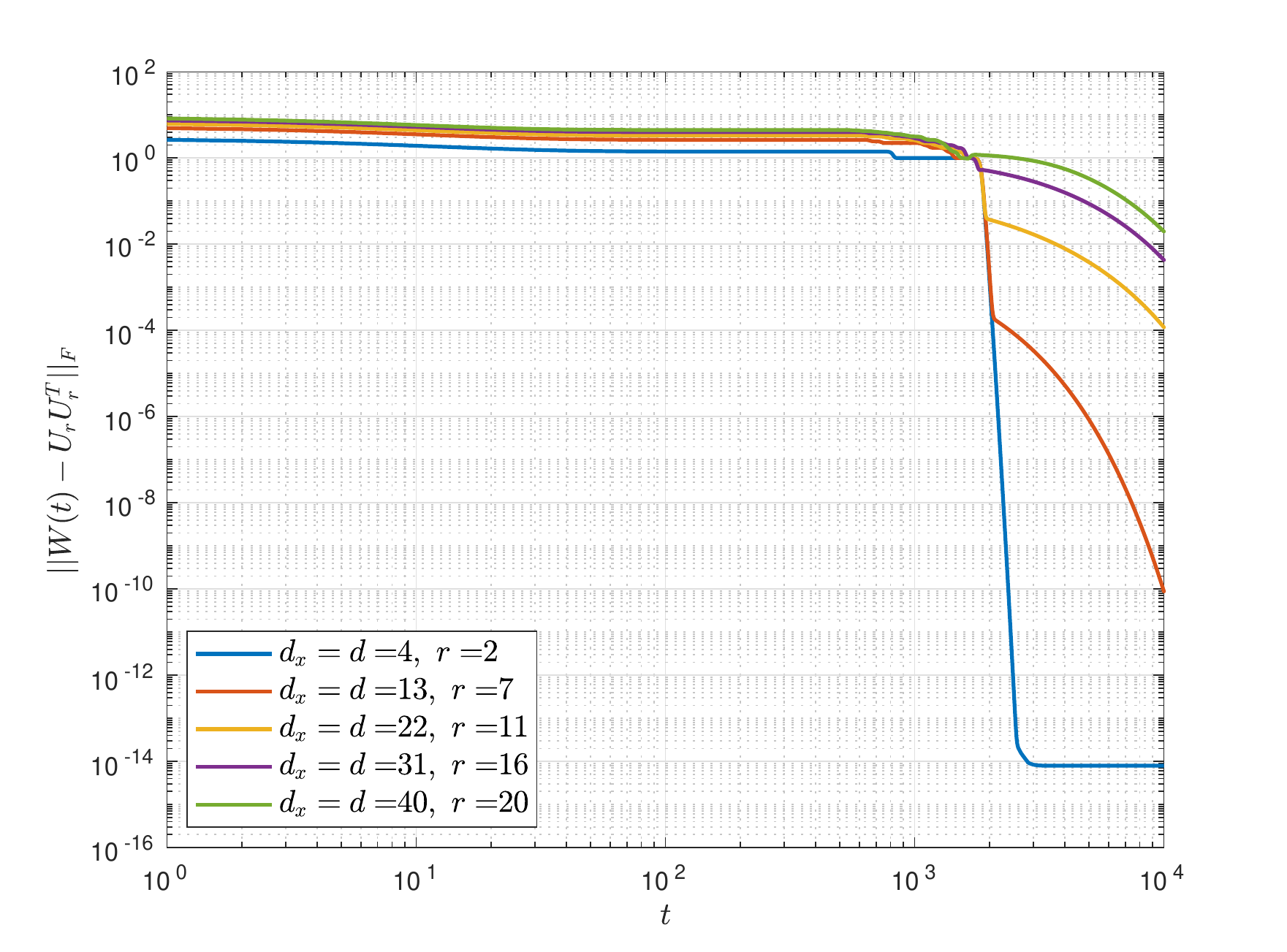} 
	\vspace{-1mm}
	\caption{In the special balanced case, {\em left panel:} norm of $W(t)$; {\em right panel:} errors between $W(t)$ and $U_rU_r^T$ for different $r$ and $d$ values.}
	\label{fig:pathobalance}\vspace{-2mm}
\end{figure}

In the second set of experiments, we attempt to test Conjecture~\ref{con:main} by constructing pathological examples, where we have balanced initial conditions, but $W(0)$ violates condition \eqref{eqn:conjecture} of Conjecture \ref{con:main}. Precisely, in the case $N=2$  we take $W_1(0) = V_r^T$ and $W_2(0) = -V_r$, where 
the columns of $V_r$ are the top $r$ eigenvectors of $XX^T$. Such $W(0)$ clearly violates the condition of the conjecture $u_i^TW(0)u_i > 0$ for all $i\in[r]$. 

The hypothesis is that in such a setting the solution will not converge to the optimal solution proposed in Conjecture \ref{con:main}. Remark \ref{rem:conjecturecondition} showed that in such a case the solution should converge to $0$, that is $\lim_{t\rightarrow \infty}W(t) = 0$. This can be seen in the left panel plot of 
Figure~\ref{fig:pathobalance}. The dip in the left panel shows that $W(t)$ is approaching zero in a first phase. However, probably due to numerical errors the flow escapes the equilibrium point at zero. In fact, zero is an unstable point (a strict saddle point), so that, numerically, the flow will hardly converge to zero. The right panel plot of Figure~\ref{fig:pathobalance} shows very slow convergence to $U_rU_r^T$. Moreover, the limiting solutions (despite slow convergence) satisfy $W_2=W_1^T$ as shown in the right plot of Figure \ref{fig:ww1w2diff}.

\subsubsection{Non-balanced initial conditions}\label{sec:nic}
For $W_j(0)$, $j = 1,\ldots,N$, we randomly generate Gaussian matrices. The two plots in Figure~\ref{fig:nonbalance1} and the left panel plot of Figure~\ref{fig:nonbalance2} show that $W(t)$ converges to $U_rU_r^T$. As in the balanced case we can confirm that $\|X-W(t)X\|_F$ converges to $\sqrt{\sum_{i>r}\sigma_i^2}$. On the other hand, for $N=2$ in this case we see that $W_2(t)$ does not converge to $W_1(t)^T$ in contrast to the balanced case, as can be seen in the right panel plot of Figure~\ref{fig:nonbalance2}.
\begin{figure}[h]
	\centering 
	\includegraphics[width=0.45\textwidth,height=0.3\textwidth]{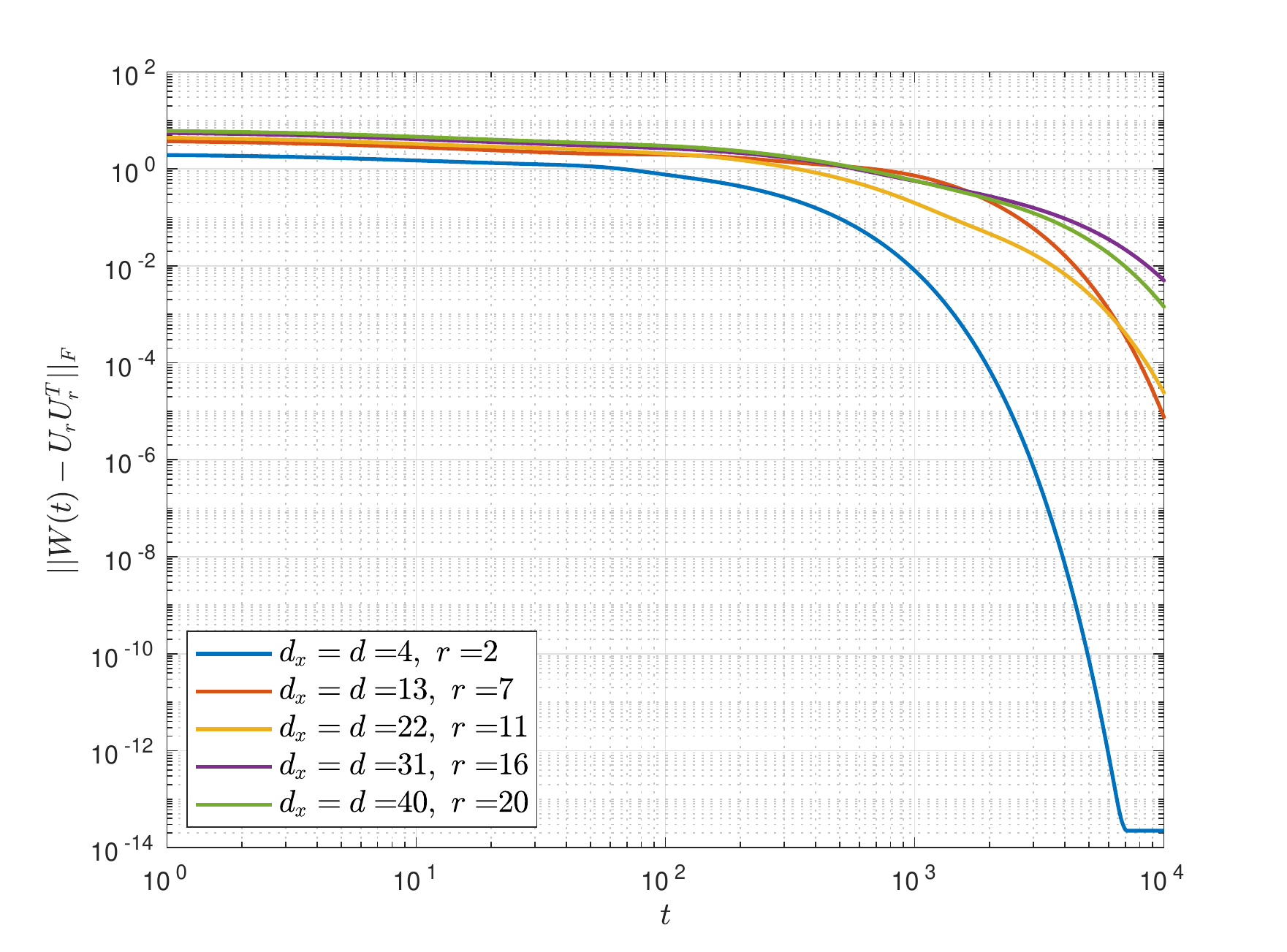}
	\includegraphics[width=0.45\textwidth,height=0.3\textwidth]{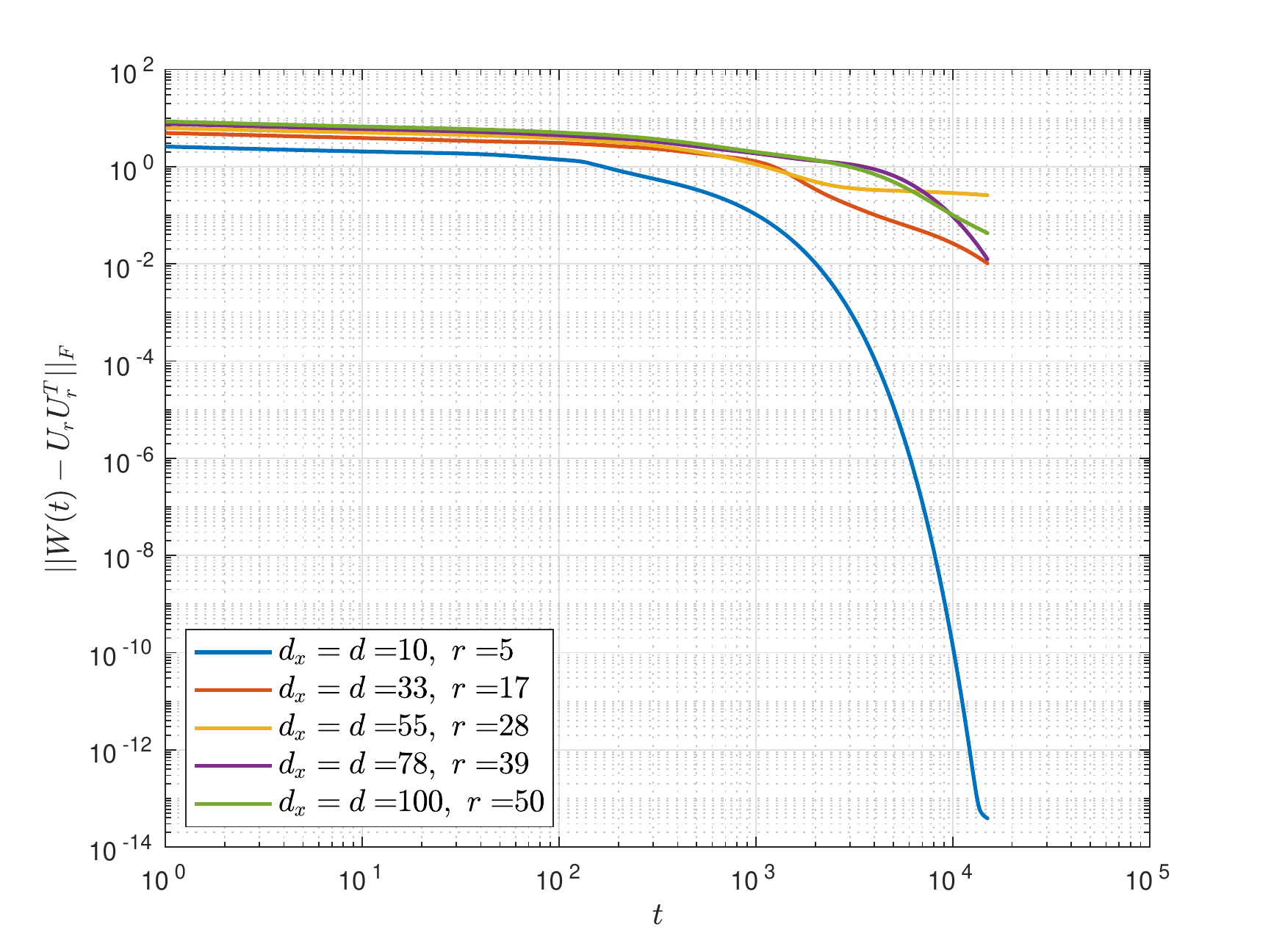} 
	\caption{Non-balanced case convergence of solutions of the gradient flow. Errors between $W(t)$ and $U_rU_r^T$ for different $r$ and $d$ values for {\em left panel:} $N=2$, {\em right panel:} $N=5$.}
	\label{fig:nonbalance1}
\end{figure}

\begin{figure}[h]
	\centering  
	\includegraphics[width=0.45\textwidth,height=0.3\textwidth]{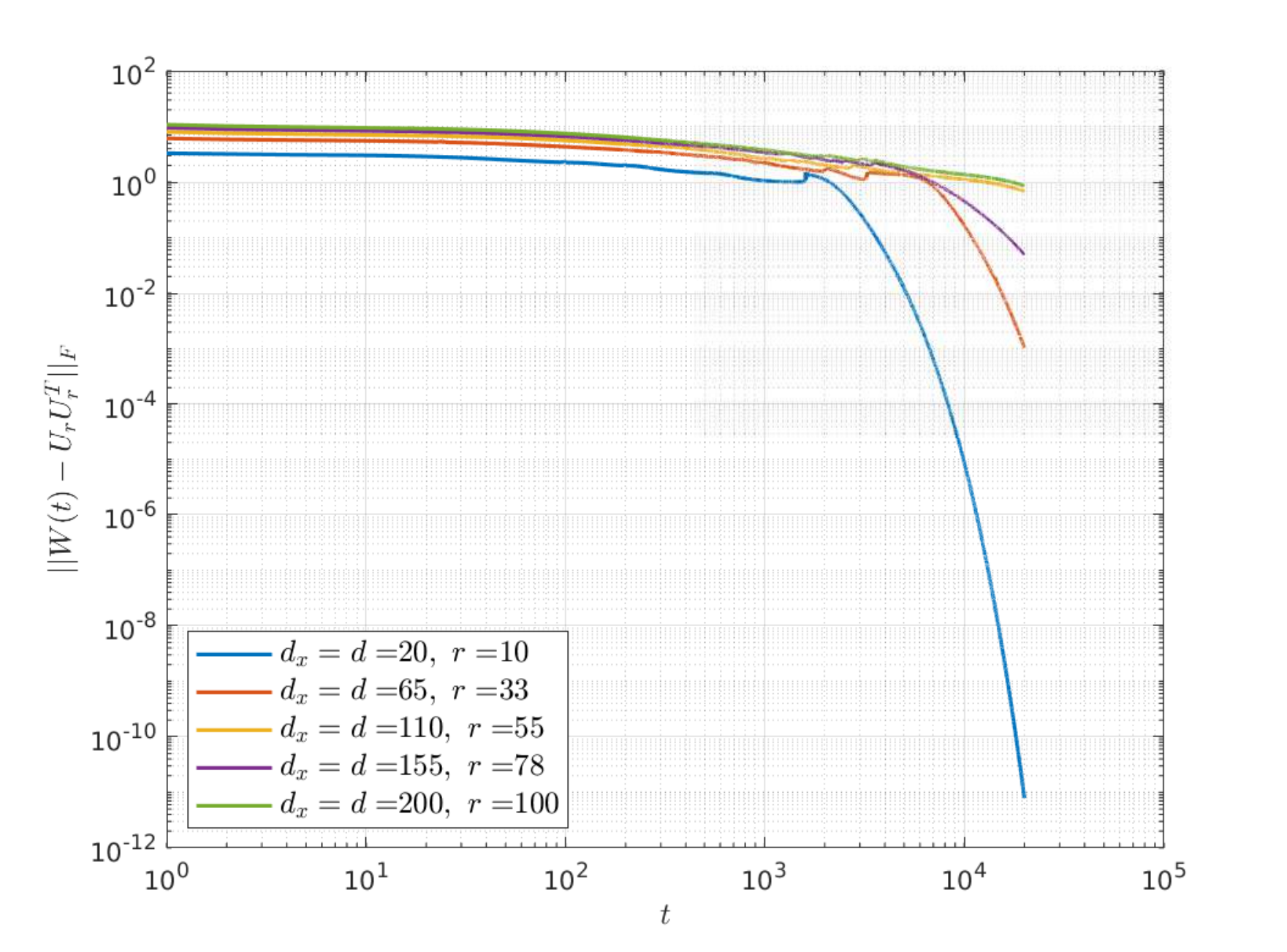} 
	\includegraphics[width=0.45\textwidth,height=0.3\textwidth]{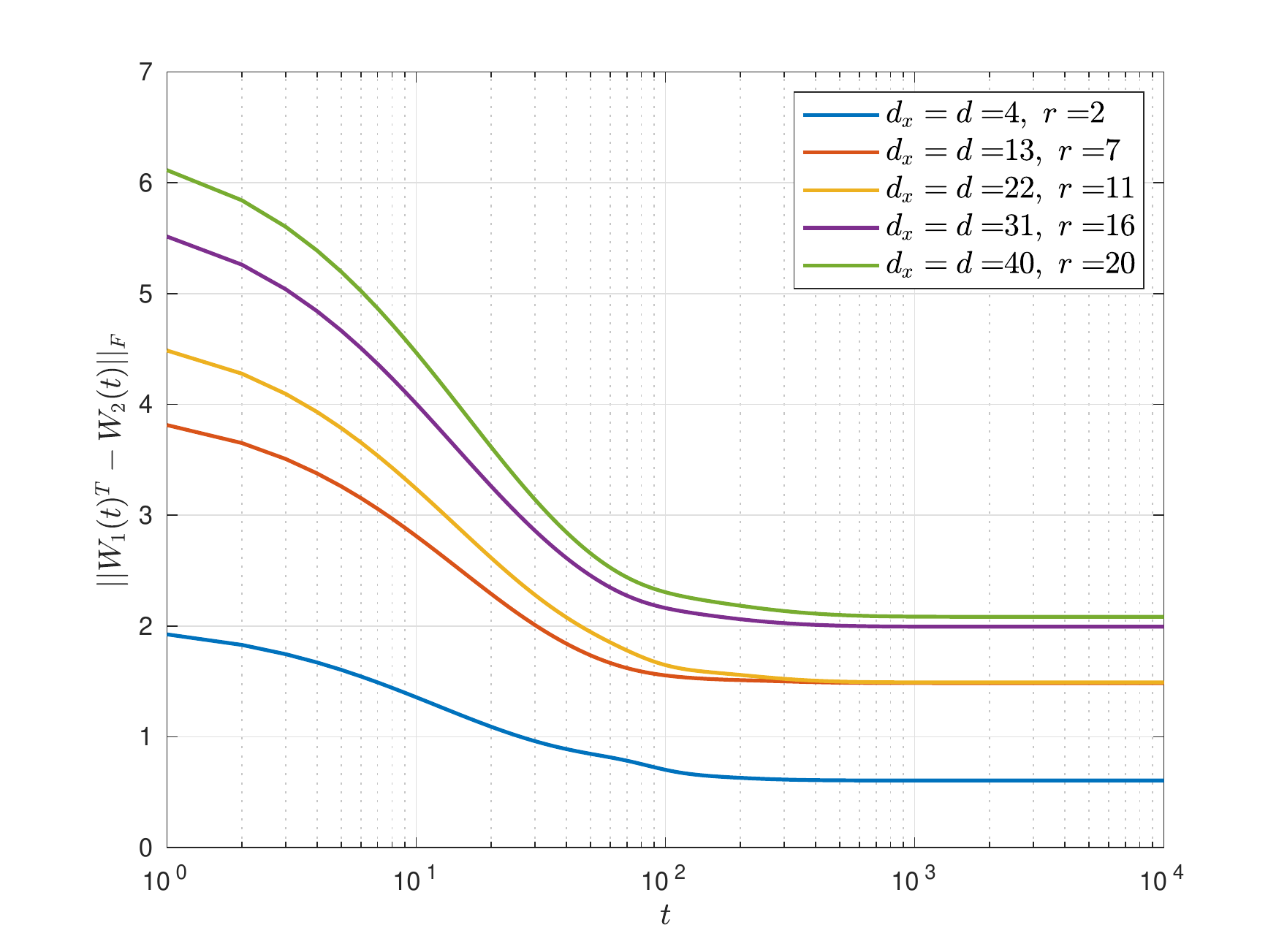}
	\caption{Non-balanced case convergence. {\em Left panel:} Errors between $W(t)$ and $U_rU_r^T$ for different $r$ and $d$ values for $N=10$. {\em Right panel:} Errors between $W_2(t)$ and $W_1(t)^T$.} 
	\label{fig:nonbalance2}
\end{figure}

\subsubsection{Convergence rates}\label{sec:convrates}
Here the data matrix $X\in \RR^{d_x\times m}$ is generated with columns drawn i.i.d.from a Gaussian distribution, i.e., $x_i \sim \mathcal{N}(0,\sigma^2 I_{d_x})$, where $\sigma = 1/\sqrt{d_x}$. Random realization of $X$ with two different values for $d_x$ (as in above $m = 3d$) and different $r$, the rank of $W(t)$, are used. For each fixed $d$, the dimensions $d_j$ of the $W_j \in \mathbb{R}^{d_j \times d_{j-1}}$ are selected using an arbitrarily chosen $r$ and setting $d_j = [r + (d-r)(j-1)/(N-1)]$ for $j=1,\ldots, N$.
The value of $r$  is stated in the caption of the figures.
The experiments show very rapid convergence of the solutions but also the dependence of the convergence rate on $N$, $d_x$, and $r$. We investigate this for different values of $N$, $d_x$ and $r$, in both the balanced and non-balanced cases. Convergence plots for the balanced initial conditions are shown in Figure~\ref{fig:convrate_bic}, depicting smooth convergence. Similarly, we have convergence rates of the non-balanced case in Figure~\ref{fig:convrate_nic}. These plots also show a slightly faster convergence for the balanced case than for the non-balanced case.
\begin{figure}[h!]
	\centering 
	\includegraphics[width=0.45\textwidth,height=0.3\textwidth]{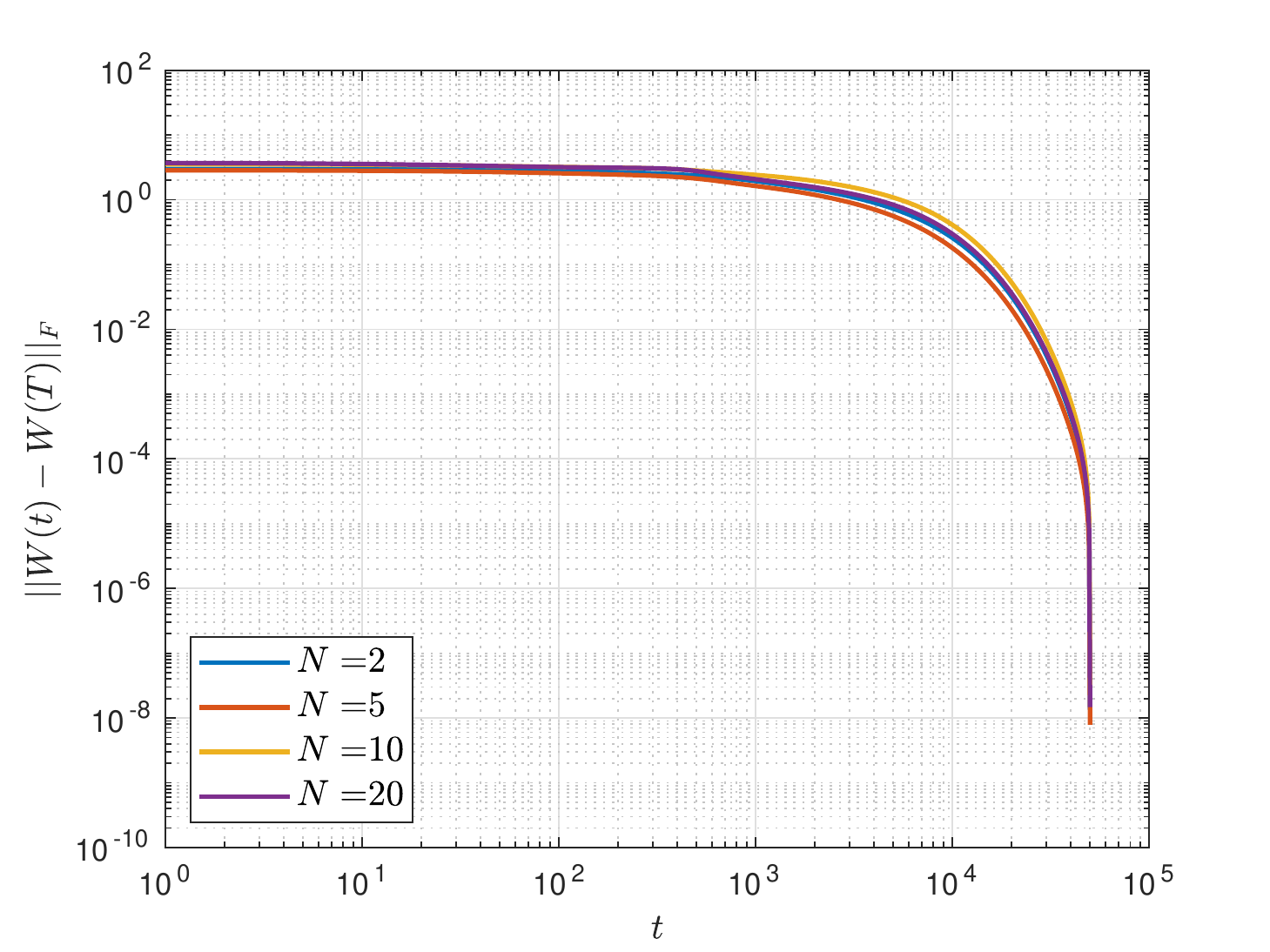}
	\includegraphics[width=0.45\textwidth,height=0.3\textwidth]{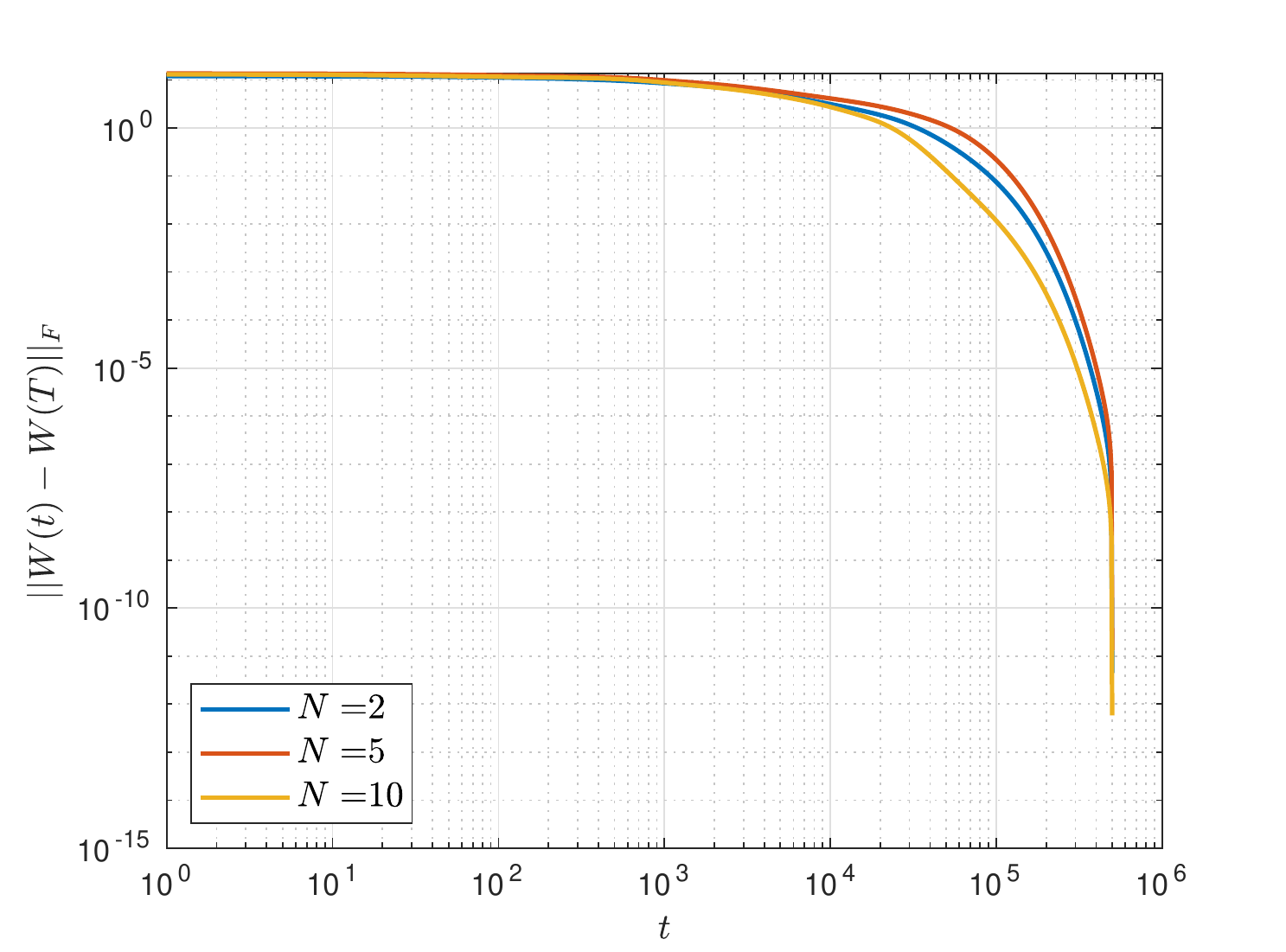}  
	\caption{Convergence rates of solutions of the gradient flow in the autoencoder case with balanced initial conditions -- errors between $W(t)$ and $W(T)$ for different $N$ values, where $T$ is the final time. Dimensions {\em Left panel}: $d_x = 20$, $r=1$; {\em Right panel}: $d_x = 200$, $r=10$.}
	\label{fig:convrate_bic}
\end{figure}
\begin{figure}[h!]
	\centering  
	\includegraphics[width=0.45\textwidth,height=0.3\textwidth]{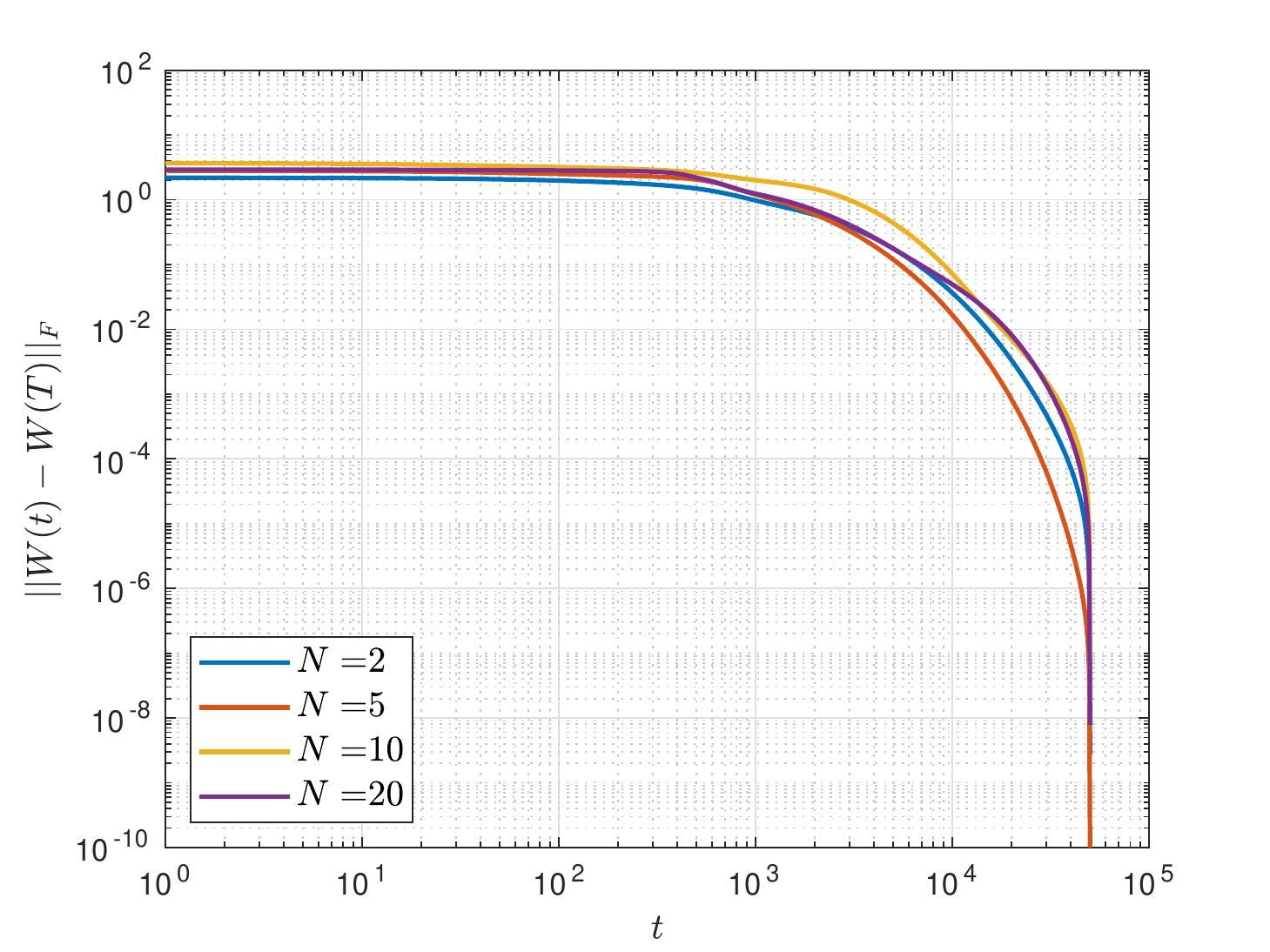} 
	\includegraphics[width=0.45\textwidth,height=0.3\textwidth]{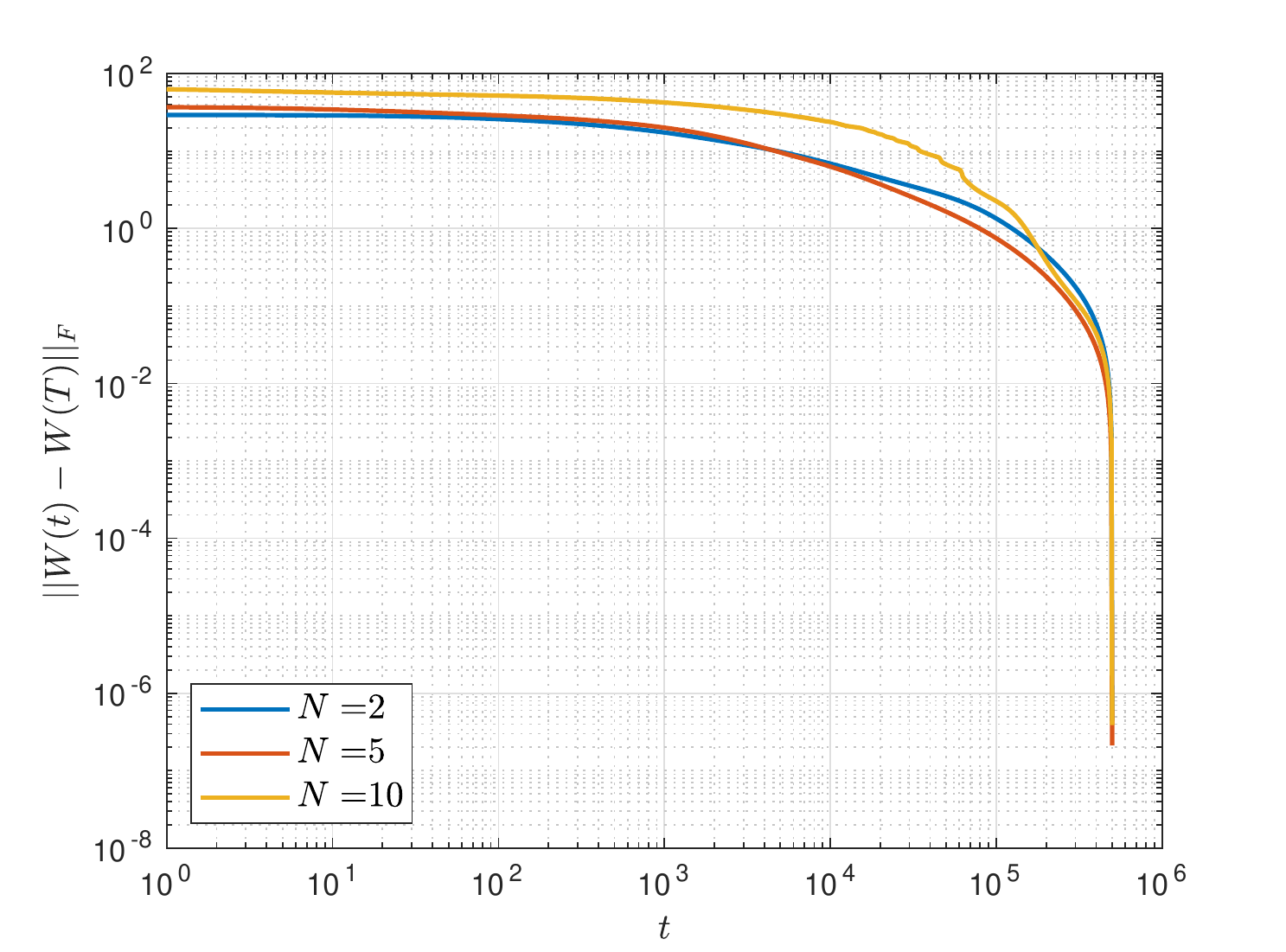} 
	\caption{Convergence rates of solutions of the gradient flow in the autoencoder case with non-balanced initial conditions -- errors between $W(t)$ and $W(T)$ for different $N$ values, where $T$ is the final time. Dimensions {\em Left panel}: $d_x = 20$, $r=1$; {\em Right panel}: $d_x = 200$, $r=10$.}
	\label{fig:convrate_nic}
\end{figure}

\subsection{General supervised learning case}\label{sec:supervisedlearning}
Experiments were also conducted to test the results in the general supervised learning setting to support theoretical results in Theorem~\ref{thm-kawag} and Propositions \ref{critpoints_general} and \ref{prop:strict_saddle_property_general}. We show results for $N=2,5,10,20$, and two sets of values for $d_x$ and $r$ (rank of $W(t)$ and $\widetilde{W}$, the true parameters). The data matrix $X$ is generated as in the autoencoder case and $Y = \widetilde{W}X$, where $\widetilde{W}=\widetilde{W}_N\cdots \widetilde{W}_1$, with $\widetilde{W}_j\in\RR^{d_j\times d_{j-1}}$ for $j=1,\ldots,N$ with $d_N = d_0 = d_x = d$ and $d_1 = r$ is the rank of $\widetilde{W}$. The entries of $\widetilde{W}_j$ are randomly generated independently from a Gaussian distribution with standard deviation $\sigma = 1/\sqrt{d_j}$. The dimensions $d_j\times d_{j-1}$ of the $W_j$ for $j=1,\ldots, N$, are again selected respectively in an integer grid, i.e., $d_j = [r + (d_x-r)(j-1)/(N-1)]$, 
where $r$ is arbitrarily fixed. The initial conditions are generated as was done in the autoencoder case.
We investigate the convergence rates for the balanced and non-balanced initial conditions of the gradient flows. The results of the experiments are plotted in Figures \ref{fig:SLconv_bic} and \ref{fig:SLconv_nic}. In these plots $k$ is the rank of $Q \in \mathbb{R}^{d_y \times d_x}$ defined in \eqref{def:Q-matrix}, and 
$Q = U_k \Sigma_k V_k$ is the (reduced) singular value decomposition, i.e., 
$U_k \in \mathbb{R}^{d_x \times k}$ and $V_k \in \mathbb{R}^{d_y \times k}$ 
have orthonormal columns and $\Sigma_k \in \mathbb{R}^{k \times k}$ is a 
diagonal matrix containing the non-zero singular values of $Q$.
\begin{figure}[h!]
	\centering 
	\includegraphics[width=0.45\textwidth,height=0.3\textwidth]{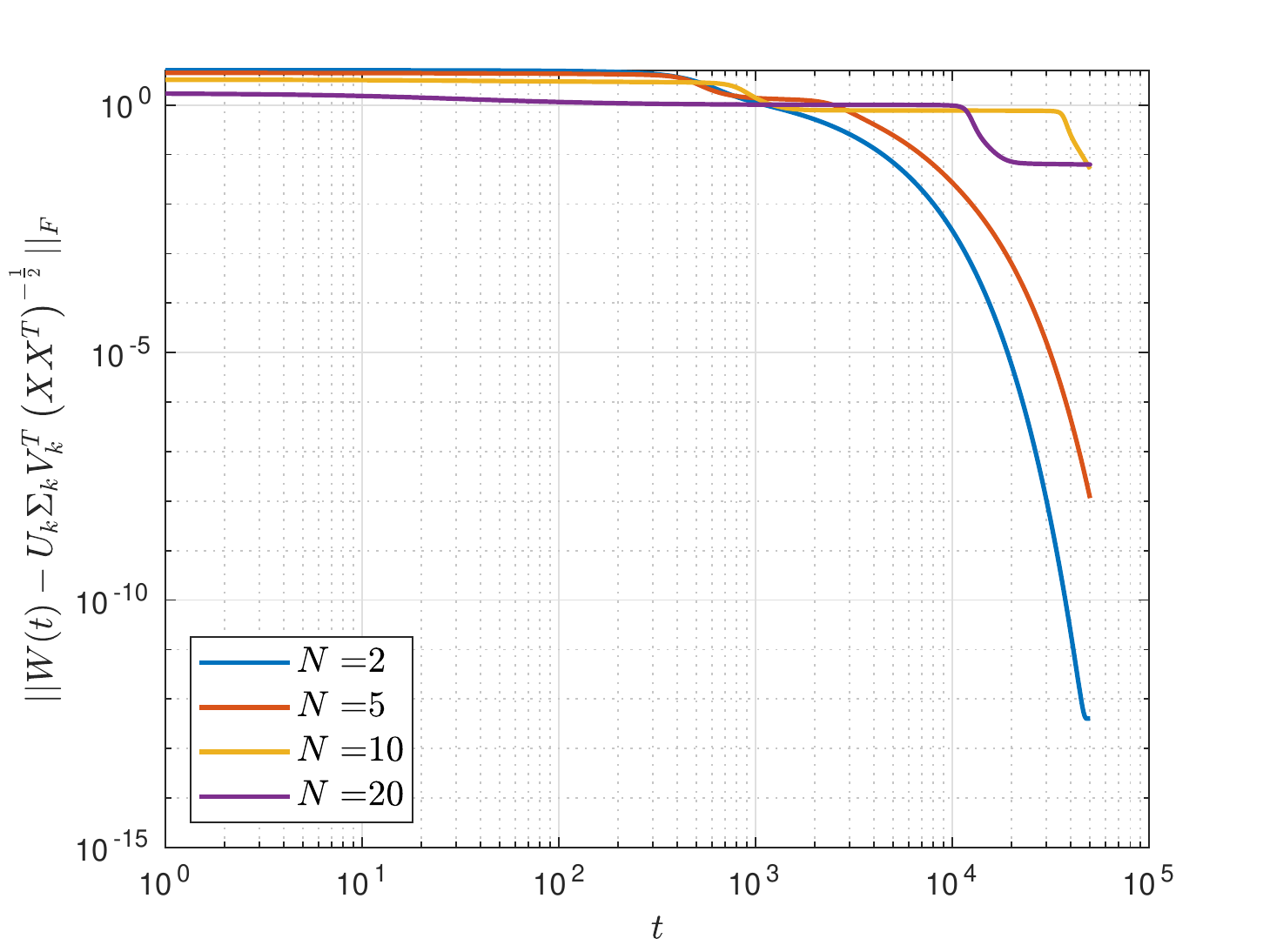}
	\includegraphics[width=0.45\textwidth,height=0.3\textwidth]{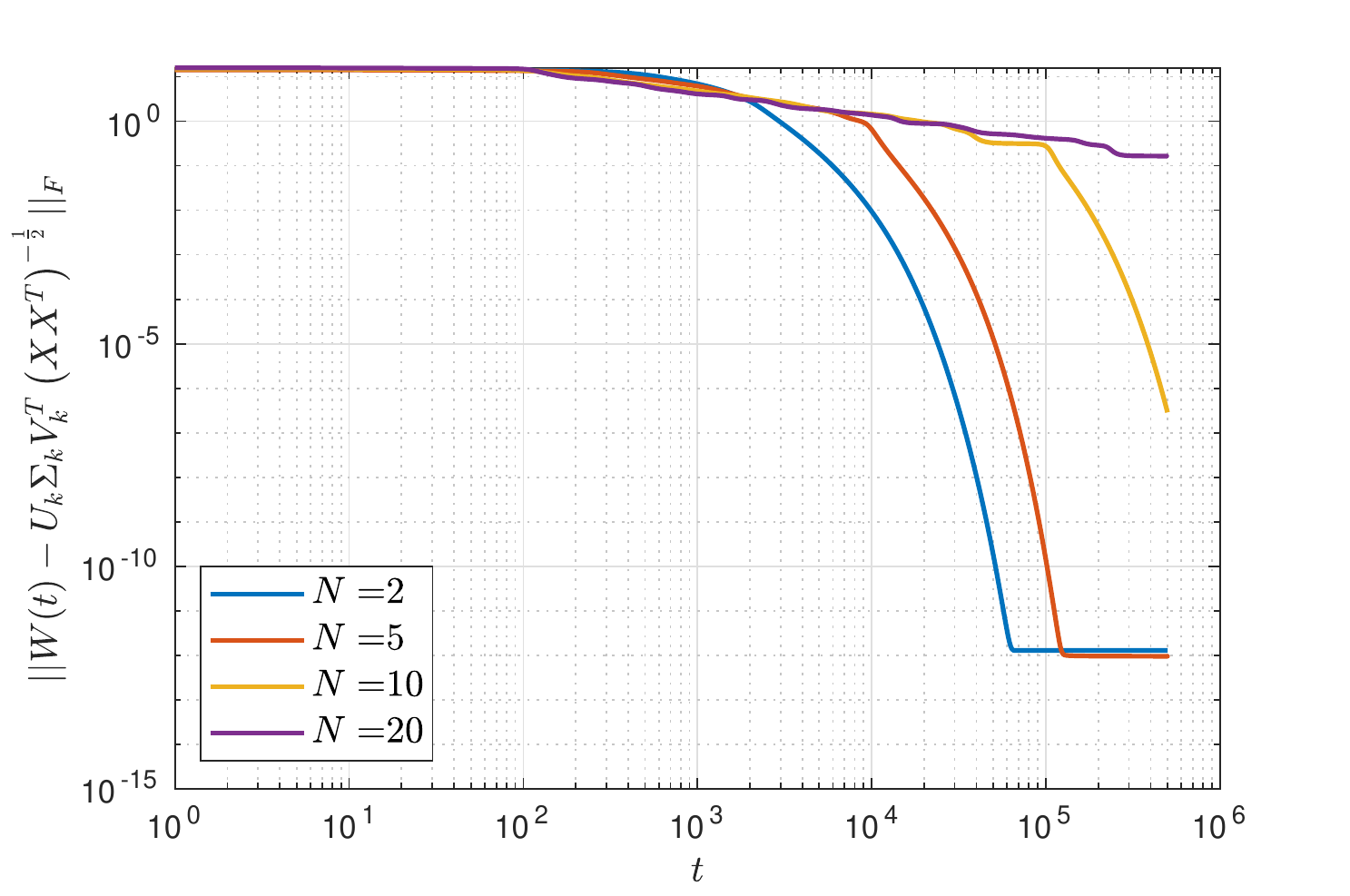} 
	\caption{Convergence rates of solutions of the gradient flow of the general supervised learning problem depicted by convergence to $W$ in $(1)$ of Proposition~\ref{critpoints_general} with balanced initial conditions for {\em left panel:} $d_x = 20$, $r=2$; {\em right panel:} $d_x = 200$, $r=20$.}
	\label{fig:SLconv_bic}
\end{figure}

With balanced initial conditions the plots of Figure~\ref{fig:SLconv_bic} show convergence rates of the flow 
to $W$ in $(1)$ of Proposition \ref{critpoints_general}. With non-balanced initial conditions the plots of 
Figure~\ref{fig:SLconv_nic} show convergence rates to $W$ in $(1)$ of Proposition~\ref{critpoints_general}. 
These results show rapid convergence of the flow and the dependence of the convergence rate on $N$, $r$ and $d_x$ with either balanced or non-balanced initial conditions. Note that $W$ in $(1)$ of 
Proposition~\ref{critpoints_general} is the same as the true parameters $\widetilde{W}$. 
This can be seen by comparing the left panel plot of 
Figure~\ref{fig:SLconv_bic} to the left panel plot of Figure~\ref{fig:SLconvtrue} 
and the left panel plot of Figure~\ref{fig:SLconv_nic} to the right panel plot of Figure~\ref{fig:SLconvtrue}.

\begin{figure}[h!]
	\centering 
	\includegraphics[width=0.45\textwidth,height=0.3\textwidth]{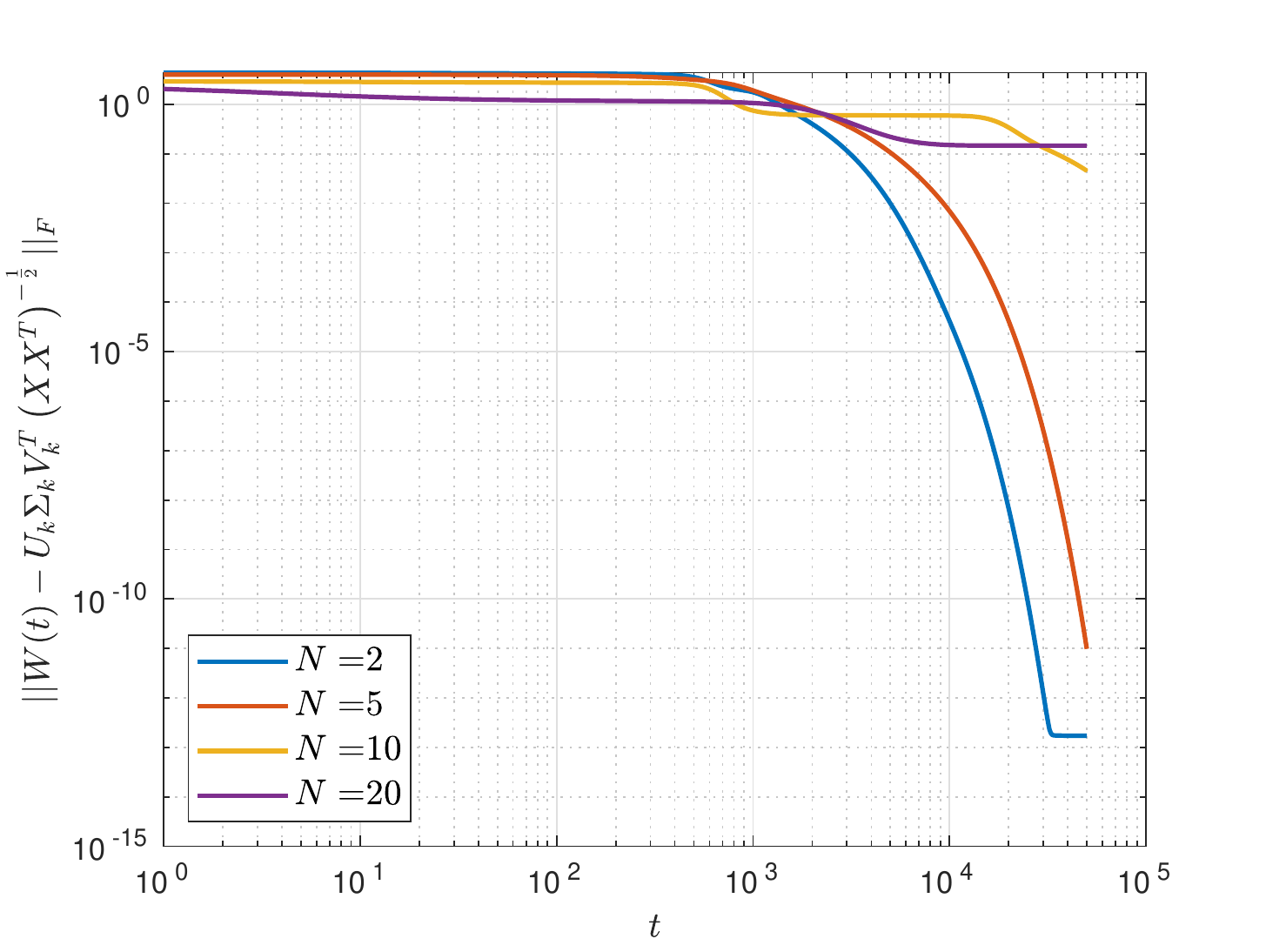} 
	\includegraphics[width=0.45\textwidth,height=0.3\textwidth]{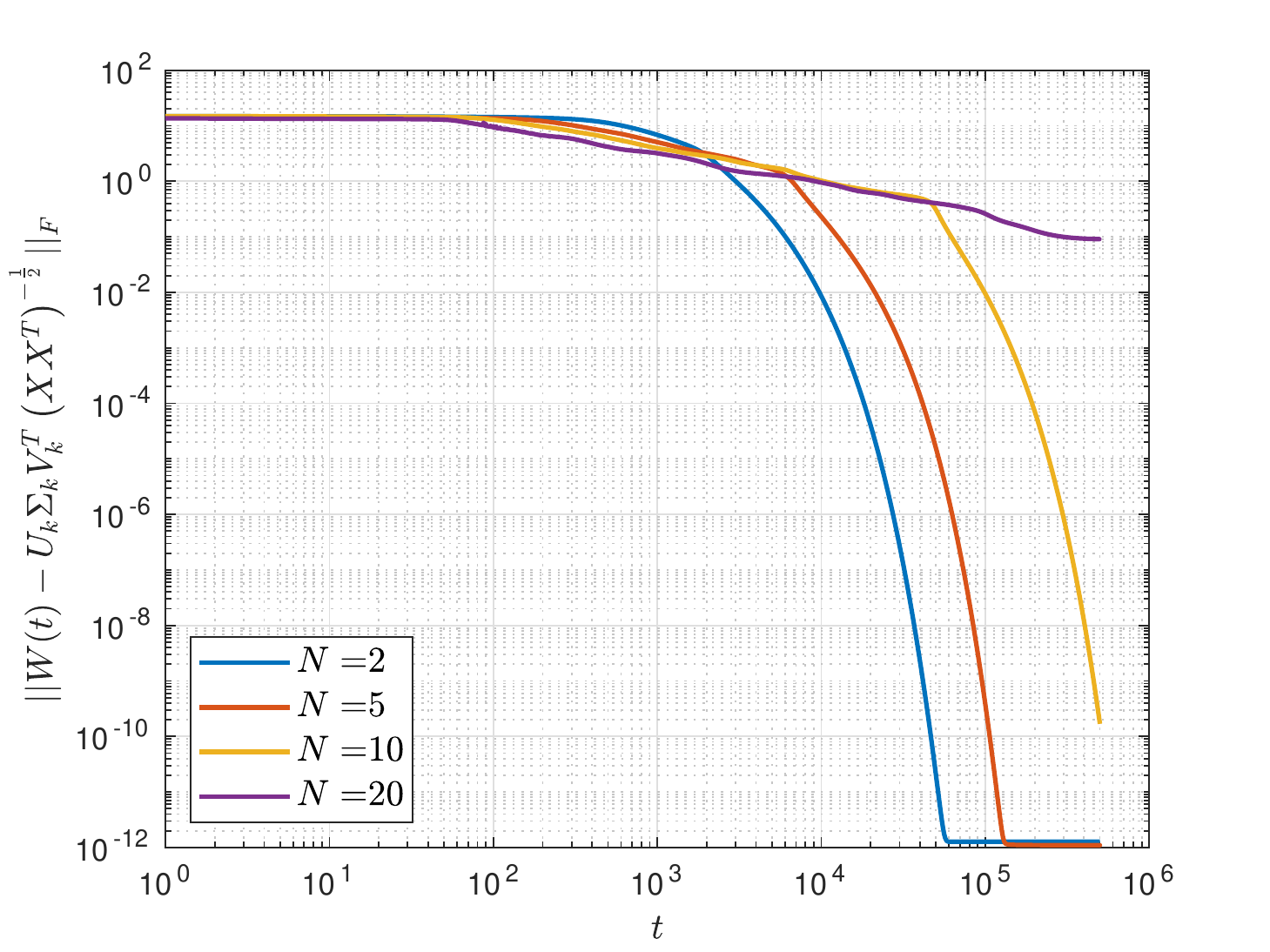} 
	\caption{Convergence rates of solutions of the gradient flow of the general supervised learning problem depicted by convergence to $W$ in $(1)$ of Proposition \ref{critpoints_general} with non-balanced initial conditions for {\em left panel:} $d_x = 20$, $r=2$; {\em right panel:} $d_x = 200$, $r=20$.}
	\label{fig:SLconv_nic}
\end{figure}

Convergence is slower for larger $N$, and it seems not to depend on the initial conditions, 
balanced or non-balanced, see the plots of Figures~\ref{fig:SLconv_bic} and \ref{fig:SLconv_nic}. 
Equivalently, this can be seen from the error of the supervised learning loss shown in the plots of 
Figure~\ref{fig:SLerror} for balanced initial conditions.
There is much stronger dependence on $N$ in this setting than in the autoencoder setting.
\begin{figure}[h!]
	\centering 
	\includegraphics[width=0.45\textwidth,height=0.3\textwidth]{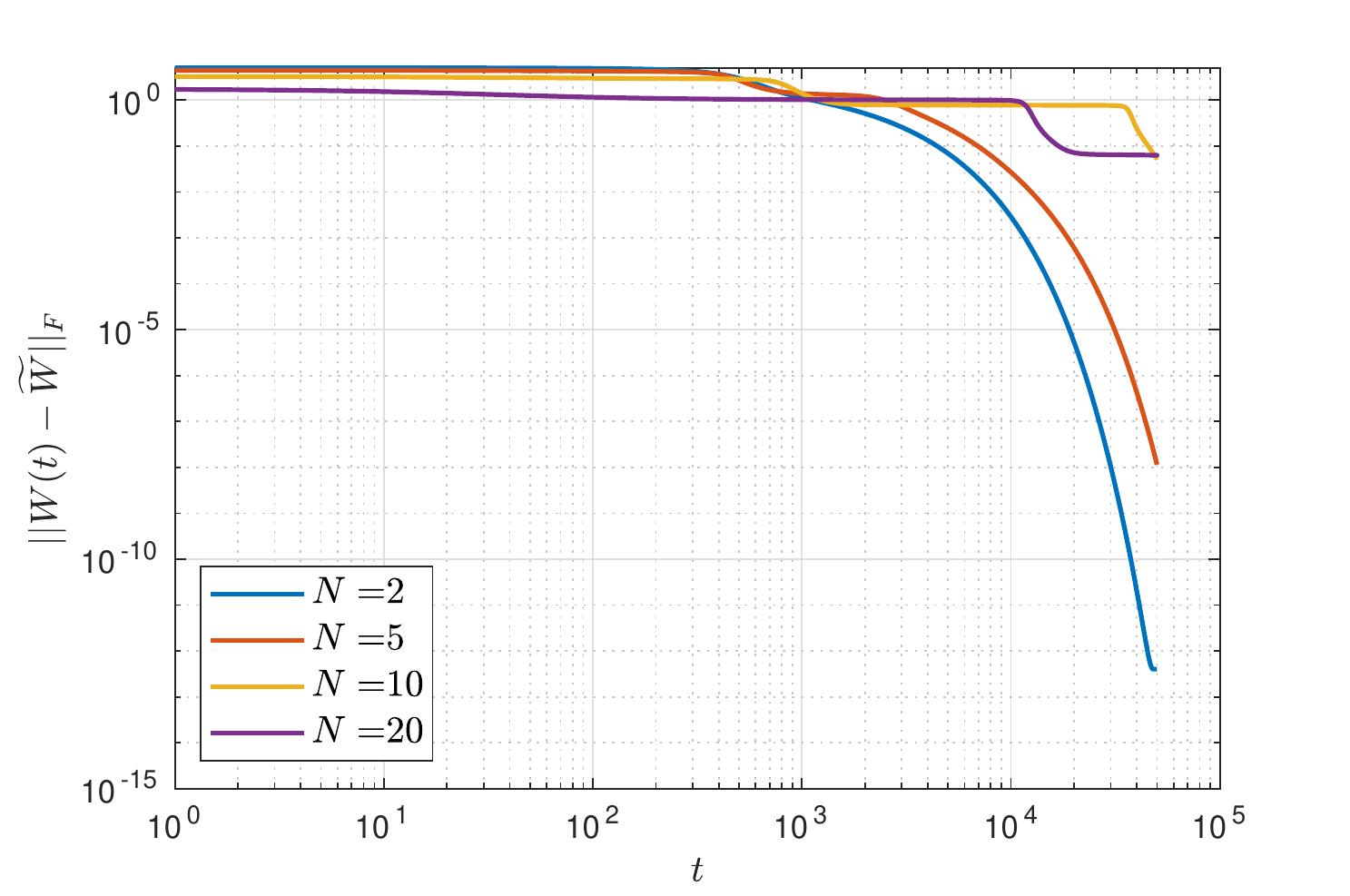}
	\includegraphics[width=0.45\textwidth,height=0.3\textwidth]{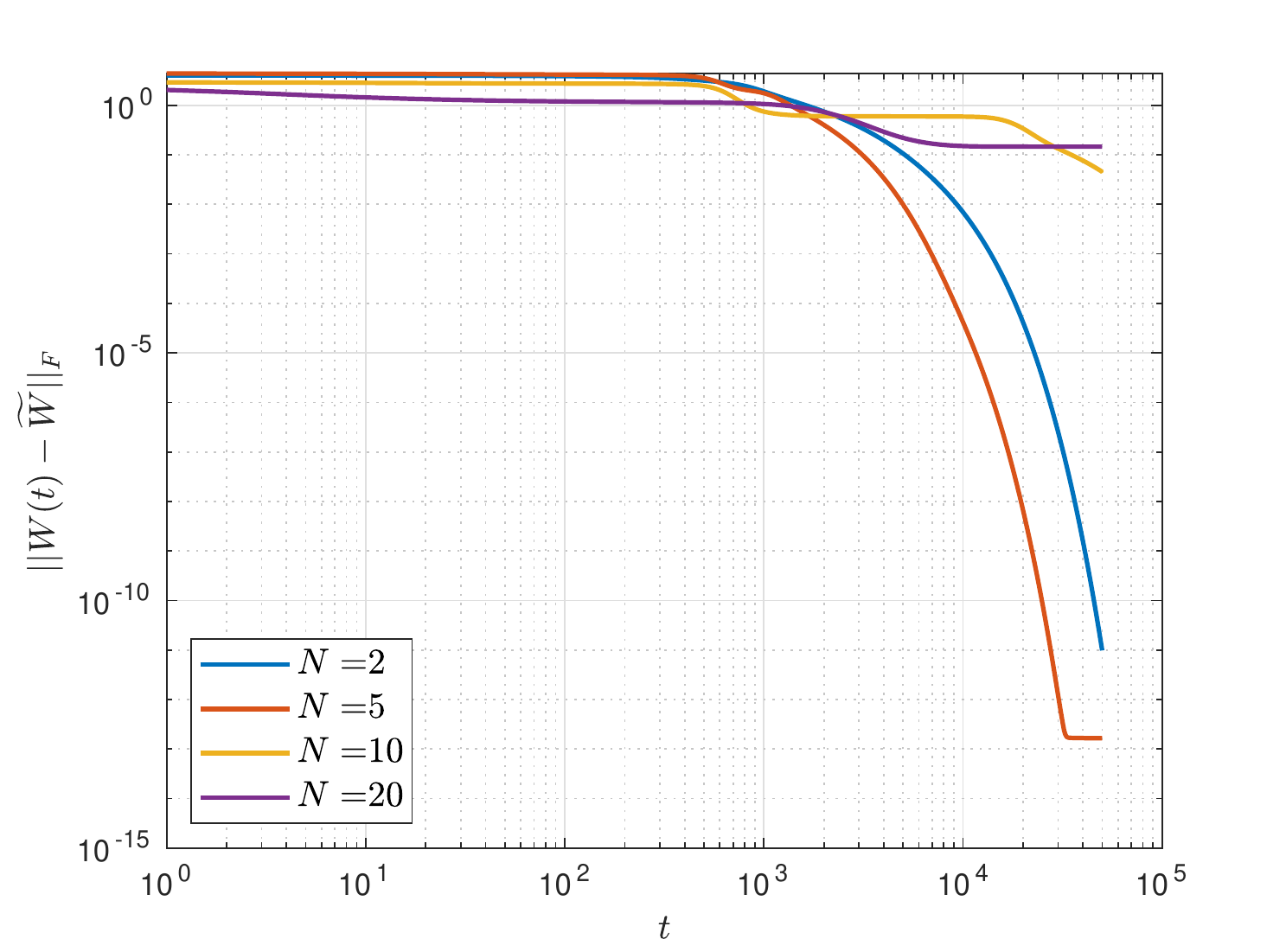}
	\caption{Convergence to the true parameters $\widetilde{W}$ for ($d_x = 20$, $r=2$) with {\em left panel:} balanced initial conditions; {\em right panel:} non-balanced initial conditions.}
	\label{fig:SLconvtrue}
\end{figure}

\begin{figure}[h!]
	\centering 
	\includegraphics[width=0.45\textwidth,height=0.3\textwidth]{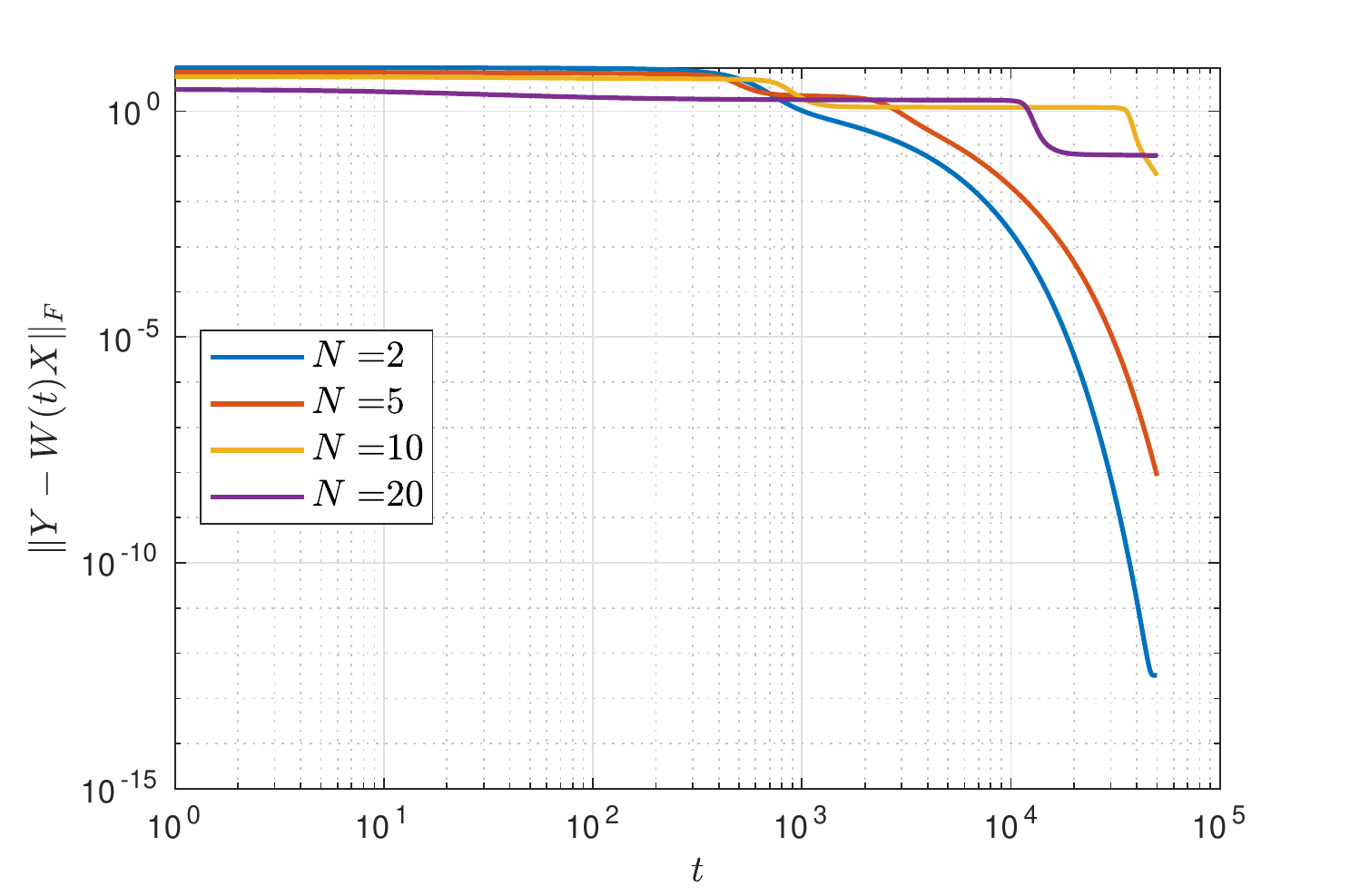}
	\includegraphics[width=0.45\textwidth,height=0.3\textwidth]{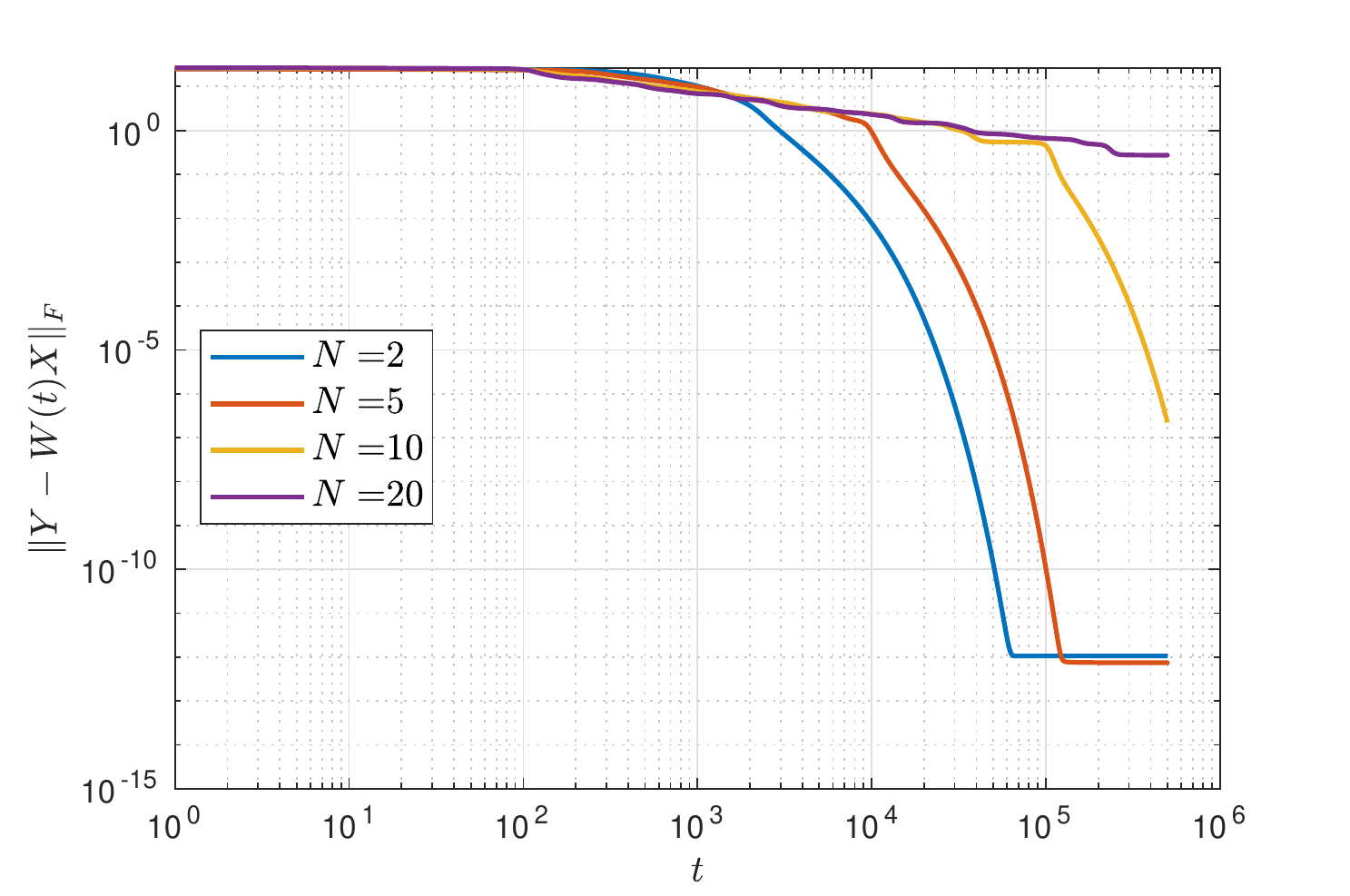} 
	\caption{General supervised learning errors with balanced initial conditions for dimensions {\em left panel:} $d_x = 20$, $r=2$; {\em right panel:} $d_x = 200$, $r=20$.}
	\label{fig:SLerror}
\end{figure}

\subsection{Conclusion}\label{sec:conclusion}
To conclude the numerical section we summarise our results as follows. In the autoencoder case we confirmed that the solutions of the gradient flow converges to $U_rU_r^T$, while in the general supervised learning case we confirmed convergence of the flow to $W$ in $(1)$ of Proposition~\ref{critpoints_general}. Such convergence occurs with either balanced or non-balanced initial conditions albeit a slight faster convergence in the balanced than in the non-balanced. Secondly, in the autoencoder case we numerically confirmed the hypothesis of Conjecture \ref{con:main} and that $W_2(t) = W_1(t)^T$ as claimed for $N=2$ with balanced initial conditions, which does not necessarily hold with non-balanced initial conditions. Moreover, in both the autoencoder and the general supervised learning setting we see that as the size ($N,d_x,r$) of the problem instance  increases the convergence rates decrease. In the autoencoder case we saw stronger dependence in $d_x$ and $r$ than in the general supervised learning case. On the other hand the dependence on $N$ seems to be stronger in the general supervised learning case than in the autoencoder case.


\appendix 
\label{appendix}

\section{Proof of Proposition~\ref{metric-explicit}}
\label{appendix-metric-explicit}

The proof is based on the next lemma, which follows from \cite{bhuc09}.
\begin{lemma}\label{lem:inv} Let $A \in \RR^{m \times m}$, $B\in \RR^{n \times n}$, be positive definite matrices and $Y \in \RR^{m \times n}$. Then, for $p \in \NN$, $p \geq 2$, the solution $X \in \RR^{m \times n}$
of the matrix equation
\begin{equation}\label{op-eq1}
A^{p-1} X + A^{p-2} X B + \cdots + A X B^{p-2} + X B^{p-1} = Y
\end{equation}
satisfies
\begin{align}
X & = \frac{\sin(\pi/p)}{\pi} \int_0^\infty (t \id_m + A^p)^{-1} Y (t \id_n +B^p)^{-1} t^{1/p} dt \label{solution1}\\
   & = \frac{1}{p \Gamma(1-1/p)} \int_0^\infty \int_0^t  e^{-s A^p} Y e^{-(t-s)B^p} ds \, t^{-(1+1/p)} dt. \label{solution2}
\end{align}
\end{lemma}
\begin{proof} The first formula \eqref{solution1} is shown for $n=m$ in \cite{bhuc09}, for matrices with eigenvalues in $\{z \in \CC: z \neq 0, - \pi/p < \arg z < \pi/p\}$.
Positive definite matrices clearly have their eigenvalues in this set. Formula \eqref{solution1} extends to squares $A,B$ of possibly different dimensions $m,n$.
In fact, \cite{bhuc09} first proves \eqref{solution1} for $A = B$, see \cite[eq.~(8)]{bhuc09} and then extends the solution by the ``Berberian trick'' which introduces the block matrices
\[
\widetilde{A} =  \left( \begin{matrix} A & 0 \\ 0 & B\end{matrix} \right), \quad \widetilde{Y} =  \left( \begin{matrix} 0 & Y \\ 0 & 0\end{matrix} \right),
\quad \widetilde{X} = \left( \begin{matrix} X_{11} & X_{12} \\ X_{21} & X_{22} \end{matrix} \right).
\]
If $\widetilde{X}$ solves $\sum_{i=1}^{p} \widetilde{A}^{p-i} \widetilde{X} \widetilde{A}^{i-1} = \widetilde{Y}$, then the submatrix $X=X_{12}$ solves \eqref{op-eq1}
and one obtains \eqref{solution1}. This argument works for general $m,n$ so that \eqref{solution1} holds under the conditions of the lemma.

In the case that $A = B$, \cite[eq.~(15)]{bhuc09} implies \eqref{solution2}. The general case of possibly different $A,B$ is then established again by the Berberian trick as above. 
\end{proof} 

\begin{proof}[Proof of Proposition~\ref{metric-explicit}] 
We aim at applying Lemma~\ref{lem:inv} for the matrices $A = (WW^T)^{1/N}$ and $B = (W^T W)^{1/N}$ in order to obtain a formula for $\bar{\mathcal{A}}_W^{-1}$. Unfortunately, in the rank deficient case, these matrices are only positive semi-definite and not positive definite. We will overcome this problem by using an approximation argument. For $u > 0$, the matrices $(u\id_{d_y} + WW^T)^{1/N}$ and $(u\id_{d_x} + W^T W)^{1/N}$ are positive definite and hence, the linear operator 
$$
\mathcal{A}_{W,u} : \RR^{d_y \times d_x} \to \RR^{d_y \times d_x}, \quad \mathcal{A}_{W,u}(Z) = \sum_{j=1}^{N} (u \id_{d_y} + W W^T)^{\frac{N-j}{N}} Z (u\id_{d_x} + W^TW)^{\frac{j-1}{N}}
$$
is invertible by Lemma~\ref{lem:inv} with inverse $\mathcal{A}_{W,u}^{-1} : \RR^{d_y \times d_x} \to \RR^{d_y \times d_x}$,
\begin{equation}\label{inverseAW}
\mathcal{A}_{W,u}^{-1}(Z) = \frac{\sin(\pi/N)}{\pi} \int_0^\infty \left( (t + u)\id_{d_y} + WW^T\right)^{-1} Z \left((t+ u) \id_{d_x} + W^T W\right)^{-1}  t^{1/N} dt. 
\end{equation}
Furthermore, $\mathcal{A}_{W,u}$ maps $T_W(\mathcal{M}_k)$ into $T_W(\mathcal{M}_k)$ for all $u \geq 0$. Indeed, let $W = U \Sigma V^T$ be the full singular value decomposition of $W$ with $U$, $V$ being (square) orthogonal matrices and $\Sigma= \diag(\sigma_1,\hdots,\sigma_k,0,\hdots,0)$.
If $Z = WA + BW \in T_W(\mathcal{M}_k)$ then 
for $\alpha,\beta \geq 0$,
\begin{align*}
  & (u \id_{d_y} + W W^T)^{\alpha} Z (u\id_{d_x} + W^TW)^{\beta} \\
  & \quad 
    = U \diag( ( u+ \sigma_1^2)^{\alpha}, \hdots, ( u+ \sigma_k^2)^{\alpha}, 
      u^{\alpha},\hdots, u^{\alpha})) U^T U \Sigma V^T A  (u\id_{d_x} + W^TW)^{\beta} 
\\
  & \quad\quad
    + (u\id_{d_y} + W W^T)^{\alpha} B U \Sigma V^T V 
      \diag( ( u+ \sigma_1^2)^{\beta}, \hdots, ( u+ \sigma_k^2)^{\beta} , 
        u^{\beta},\hdots, u^{\beta})) V^T
\\
  & \quad
    = U \Sigma \diag(  ( u+ \sigma_1^2)^{\alpha}, \hdots, ( u+ \sigma_k^2)^{\alpha}, 
      0 ,\hdots, 0) V^T A (u\id_{d_x} + W^TW)^{\beta}  
\\
  & \quad\quad
    + (u\id_{d_y} + W W^T)^{\alpha} B U \Sigma \diag( ( u+ \sigma_1^2)^{\beta}, \hdots,  ( u+ \sigma_k^2)^{\beta} , 0, \hdots, 0)  V^T
\\
  & \quad
    = W V \diag(( u+ \sigma_1^2)^{\alpha}, \hdots, ( u+ \sigma_k^2)^{\alpha}, 0, \hdots,0) V^T A
      (u\id_{d_x} + W^TW)^{\beta} 
\\
  & \quad\quad
    + (u\id_{d_y} + W W^T)^{\alpha} B U \diag(( u+ \sigma_1^2)^{\beta}, \hdots, ( u+ \sigma_k^2)^{\beta}, 
      0, \hdots,0) U^T W.
\end{align*}
The last expression is clearly an element of $T_W(\mathcal{M}_k)$ and by the formula for $\mathcal{A}_{W,u}$ this implies that 
this operator maps the tangent space $T_W(\mathcal{M}_k)$ into itself. Let us denote $\bar{\mathcal{A}}_{W,u} : T_W(\mathcal{M}_k) \to T_W(\mathcal{M}_k)$
the corresponding restriction and by $\bar{\mathcal{A}}^{-1}_{W,u} : T_W(\mathcal{M}_k) \to T_W(\mathcal{M}_k)$ the restriction of the inverse map to 
$T_W(\mathcal{M}_k)$. Clearly, \eqref{inverseAW} still holds for the restriction $\bar{\mathcal{A}}^{-1}_{W,u}$. Lemma~\ref{PosDefLemma} implies that the restrictions
$\bar{\mathcal{A}}_{W,u}$ and $\bar{\mathcal{A}}^{-1}_{W,u}$ are both well-defined also for $u=0$ with $\bar{\mathcal{A}}_{W,0} = \bar{\mathcal{A}}_{W}$
and  $\bar{\mathcal{A}}^{-1}_{W,0} = \bar{\mathcal{A}}^{-1}_{W}$. Moreover, the map $u \mapsto \bar{\mathcal{A}}_{W,u}$ is continuous in $u \geq 0$ and, hence,
also $u \mapsto \bar{\mathcal{A}}^{-1}_{W,u}$ is continuous in $u \geq 0$. We claim that also the right hand side of \eqref{inverseAW} with $Z \in T_W(\mathcal{M}_k)$
is well-defined and continuous for all $u \geq 0$, which will give an inversion formula for $\bar{\mathcal{A}}^{-1}_{W}$ by setting $u=0$. 

In order to show continuity of the right hand side of \eqref{inverseAW} in $u$,
we investigate uniform integrability of the integrand for  $Z \in
T_W(\mathcal{M}_k)$. By Lemma~\ref{lem:proj-tangent} we can write $Z = P_W(Z) =
U P_k U^T Z (\id_{d_x} - V P_k V^T) + Z V P_k V^T$, where $P_k$ is the diagonal matrix
from the lemma and $W = U\Sigma V^T$ is the (full) singular value decomposition
of $W$. In particular, the matrix $\Sigma \in \RR^{d_y \times d_x}$ has the singular values $\sigma_1
\geq \ldots \geq \sigma_k > 0$ on the diagonal, with all other entries equal to
zero. For simplicity we write $E = \id_{d_x} - V P_k V^T$. Denoting
\begin{align*}
\Lambda_{t,u} & := ((t+u) \id_{d_y} + \Sigma \Sigma^T)^{-1} = \diag\left( \frac{1}{t+u+\sigma_1^2}, \hdots, \frac{1}{t+u+\sigma_k^2}, \frac{1}{t+u} ,\hdots, \frac{1}{t+u}\right)\displaybreak[2] \in \RR^{d_y \times d_y},\\
\widetilde{\Lambda}_{t,u} & := ((t+u) \id_{d_x} + \Sigma^T \Sigma)^{-1} = \diag\left( \frac{1}{t+u+\sigma_1^2}, \hdots, \frac{1}{t+u+\sigma_k^2}, \frac{1}{t+u} ,\hdots, \frac{1}{t+u}\right)\displaybreak[2]  \in \RR^{d_x \times d_x},\\
K_{t,u} & := ((t+u) \id_{d_y} + \Sigma \Sigma^T)^{-1} P_k = \diag\left( \frac{1}{t+u+\sigma_1^2}, \hdots, \frac{1}{t+u+\sigma_k^2}, 0 ,\hdots, 0\right) \in \RR^{d_y \times d_y},
\displaybreak[2]\\
\widetilde{K}_{t,u} & := P_k ((t+u) \id_{d_x} + \Sigma^T \Sigma)^{-1} = \diag\left( \frac{1}{t+u+\sigma_1^2}, \hdots, \frac{1}{t+u+\sigma_k^2}, 0 ,\hdots, 0\right)
\in \RR^{d_x \times d_x}
\end{align*}
we have
\begin{align*}
 \left( (t + u)\id_{d_y} + WW^T\right)^{-1} Z \left((t+ u) \id_{d_x} + W^T W\right)^{-1} 
 & = U K_{t,u}  U^T Z E V \widetilde{\Lambda}_{t,u} V^T + U \Lambda_{t,u} U^T  Z  V \widetilde{K}_{t,u} V^T. 
\end{align*}
Taking the spectral norm gives 
\[
\| \left( (t + u)\id_{d_y} + WW^T\right)^{-1} Z \left((t+ u) \id_{d_x} + W^T W\right)^{-1} \|_{2 \to 2} \leq 2 \|Z\|_{2 \to 2} \sigma_k^{-2} (t+u)^{-1} 
\leq 2 \|Z\|_{2 \to 2} \sigma_k^{-2} t^{-1}.
\] 
This estimate will be good enough for $t \to 0$, for any $u>0$. For $t \to
\infty$ we need a second estimate
\begin{align*}
  & \| \left( (t + u)\id_{d_y} + WW^T\right)^{-1} Z \left((t+ u) \id_{d_x} + W^T W\right)^{-1} \|_{2 \to 2} 
\\
  & \quad
    \leq \| \left((t + u)\id_{d_y} + WW^T\right)^{-1} \|_{2 \to 2} \|Z\|_{2 \to 2} 
      \|\left((t + u)\id_{d_x} + W^TW\right)^{-1}\|_{2 \to 2} 
\\
  & \quad
    \leq \|Z\|_{2 \to 2} (t+u)^{-2} \leq  \|Z\|_{2 \to 2} t^{-2}, \vspace{-0.1cm}
\end{align*}
which holds uniformly in $u>0$ and follows from the fact that $W W^T$ and $W^T
W$ are positive semidefinite.

Altogether, for $Z \in T_W(\mathcal{M}_k)$ the integrand in \eqref{inverseAW} satisfies \vspace{-0.1cm}
\[
\| \left( (t + u)\id_{d_y} + WW^T\right)^{-1} Z \left((t+ u) \id_{d_x} + W^T W\right)^{-1} t^{1/N} \|_{2 \to 2} 
\leq \|Z\|_{2 \to 2} \min\{ \sigma_k^{-2} t^{-1+1/N}, t^{-2+1/N}\}.\vspace{-0.1cm}
\]
The latter function is integrable over $t \in (0,\infty)$ since $N\geq 2$, and
hence, for all $u \geq 0$, the integrand in \eqref{inverseAW} is uniformly
dominated by an integrable function. By Lebesgue's dominated convergence theorem
and continuity of $u \mapsto \left( (t + u)\id_{d_y} + WW^T\right)^{-1} Z \left((t+ u)
\id_{d_x} + W^T W\right)^{-1} t^{1/N}$ for all $t \in (0,\infty)$, the function $u
\mapsto \bar{\mathcal{A}}_{W,u}^{-1}(Z)$ is continuous for all $Z \in
T_W(\mathcal{M}_k)$. Altogether, we showed that
\[
\bar{\mathcal{A}}_{W}^{-1}(Z) = \frac{\sin(\pi/N)}{\pi} \int_0^\infty \left( t \id_{d_y} + WW^T\right)^{-1} Z \left(t\id_{d_x} + W^T W\right)^{-1}  t^{1/N} dt, \quad Z \in T_W(\mathcal{M}_k),\vspace{-0.2cm}
\]
and this implies \eqref{metric_F1}.


 A similar argument based on \eqref{solution2} proves \eqref{metric_F2}. In the case $N=2$, it can be shown as in \cite[Theorem VII.2.3]{bhatia97} and with the approximation argument above that
 \[
 \bar{\mathcal{A}}_W^{-1}(Z) = \int_0^{\infty} e^{-t(WW^T)^{\frac 1 2}}Z e^{-t(W^TW)^{\frac 1 2}} \,dt,\vspace{-0.2cm}
 \]
 which implies \eqref{metric_F3N2}. 
\end{proof}

\section{Proof of Proposition~\ref{C1metric}}
\label{appendix-Prop-metric-C1}

Given a system $\{\gamma_1,\hdots,\gamma_n\}$ 
 of sufficiently smooth local coordinates on $\mathcal{M}_k$, where $n$ is the dimension of $\mathcal{M}_k$, the vectors $\xi_j(W):=\frac{\partial}{\partial \gamma_j}(W)$, $j=1,\hdots,n$, form a basis of the tangent space $T_W(\mathcal{M}_k)$. We need to show that the maps
 \[
W \mapsto F_{i,j}(W) := g_W \left(  \xi_i(W),  \xi_j(W) \right)
 \]
 are continuously differentiable for all $i,j$. Note that the vector fields $\xi_j(W)$ are smooth in $W$.
 We consider the representation \eqref{metric_F1} and introduce the functions
 \[
 K_{i,j,t}(W) :=  \tr\left( (t + WW^T)^{-1} \xi_i(W) (t + W^T W)^{-1} \xi_j^T(W)\right),
 \]	
 where we write $(t+WW^T)$ for $(t \id_{d_y} + WW^T)$ and likewise $(t + W^T W) = (t\id_{d_x} + W^T W)$. For a function $f$ 
 we denote the differential of $f$ at $A$ applied to $Y$ 
 by $Df(A)[Y]$.  Denoting $\phi_t(W) = (t + WW^T)^{-1}$ and $\psi_t(W) = (t+W^TW)^{-1}$ 
 the product rule gives, for $W \in \mathcal{M}_k$ and $Y \in T_W(\mathcal{M}_k)$, 
 \begin{align}
 D K_{i,j,t}(W)[Y] & = \tr(D \phi_t(W)[Y] \xi_i(W) \psi_t(W) \xi_j^T(W)) + 
 \tr(\phi_t(W) D\xi_i(W)[Y] \psi_t(W) \xi_j^T(W)) \label{DK1}\\
 & + \tr(\phi_t(W)\xi_i(W) D\psi_t(W)[Y] \xi_j^T(W)) + \tr(\phi_t(W)\xi_i(W) \psi_t(W) D\xi_j^T(W)[Y]). \label{DK2}
 \end{align}
 The differential of the function $\phi(A) = A^{-1}$ satisfies $D\phi(A)[Y] = - A^{-1} Y A^{-1}$ so that 
 \begin{align*}
 D \phi_t(W)[Y] & = -(t+WW^T)^{-1}(W Y^T + Y W^T) (t+WW^T)^{-1}, \\ 
 D \psi_t(W)[Y] & = -(t+W^TW)^{-1}(W^T Y + Y^T W) (t+W^TW)^{-1}.
 \end{align*}
Let $W = U \Sigma V^T$ be the (full) singular value decomposition of $W$, i.e.,
$U \in \RR^{d_y \times d_y}$, $V \in \RR^{d_x \times d_x}$ with $U^T U = \id_{d_y}$,
$V^T V = \id_{d_x}$ and $\Sigma = \diag(\sigma_1,\hdots,\sigma_k,0,\hdots,0) \in
\RR^{d_y \times d_x}$ with $\sigma_1\geq \cdots \geq \sigma_k > 0$. The first
term on the right hand side of \eqref{DK1} satisfies
\begin{align*}
& \tr((D \phi_t(W)[Y] \xi_i(W) \psi_t(W) \xi_j^T(W))\\
 & = - \tr\left( (t+WW^T)^{-1}(W Y^T + Y W^T) (t+WW^T)^{-1} \xi_i(W) (t + W^T W)^{-1} \xi_j^T(W)\right).
\end{align*}
Note that $ (t+WW^T)^{-1}(W Y^T + Y W^T) (t+WW^T)^{-1} = Q + Q^T$ with
\begin{align*}
& Q=  (t+WW^T)^{-1}W Y^T (t+WW^T)^{-1}\\
&= U \diag\left(\frac{1}{t+\sigma_1^2}, \hdots,\frac{1}{t+\sigma_k^2},\frac{1}{t},\hdots,\frac{1}{t}\right)U^T(U \Sigma V^T Y^T)  U \diag\left(\frac{1}{t+\sigma_1^2}, \hdots,\frac{1}{t+\sigma_k^2},\frac{1}{t},\hdots,\frac{1}{t}\right)U^T\\
& = U \diag\left( \frac{\sigma_1}{t+\sigma_1^2}, \hdots, \frac{\sigma_k}{t + \sigma_k^2}, 0, \hdots,0\right) V^T Y^T U \diag\left(\frac{1}{t+\sigma_1^2}, \hdots,\frac{1}{t+\sigma_k^2},\frac{1}{t},\hdots,\frac{1}{t}\right)U^T.
\end{align*}
By Lemma~\ref{lem:proj-tangent} it holds
\begin{equation}\label{def:zeta}
\xi_i(W) = P_W(\xi_i(W)) = U P_k U^T \xi_i(W) (\id_{d_x} - V P_k V^T) + \xi_i(W) V P_k V^T =  \zeta^{1}_i(W) + \zeta^2_i(W).
\end{equation}
where $P_k = \diag(1,\hdots,1,0,\hdots,0)$ (with $k$ ones on the diagonal), $\zeta^{1}_i(W) = U P_k U^T \xi_i(W) (\id_{d_x} - V P_k V^T)$ and 
$\zeta^2_i(W) = \xi_i(W) V P_k V^T$. Note that also $U$ and $V$ are functions of $W$, which may be non-unique, but in this case, we just fix one choice. 
Then we have
\begin{align*}
 \xi_i(W) (t + W^T W)^{-1} \xi_j^T(W)  
 & = \overbrace{\zeta^1_i(W) (t+W^T W)^{-1} (\zeta^1_j(W))^T}^{=:E_1}  + \overbrace{\zeta^1_i(W) (t+W^T W)^{-1} (\zeta^2_j(W))^T}^{=:E_2}\\
 & + \underbrace{\zeta^2_i(W) (t + W^T W)^{-1} (\zeta^1_j(W))^T }_{=:E_3} + \underbrace{\zeta^2_i(W)(t+W^T W)^{-1} (\zeta^2_j(W))^T}_{=:E_4}.  
\end{align*}
Using cyclicity of the trace, we obtain
\begin{align*}
\tr(Q E_1) &= \tr\left(\diag\left( \frac{\sigma_1}{t+\sigma_1^2}, \hdots, \frac{\sigma_k}{t + \sigma_k^2}, 0, \hdots,0\right) V^T Y^T U \diag\left(\frac{1}{t+\sigma_1^2}, \hdots,\frac{1}{t+\sigma_k^2},0,\hdots,0 \right)  \right.\\
& \phantom{\tr(}  \left. \times  U^T \xi_i(W) (\id_{d_x} - V P_k V^T) V\diag\left(\frac{1}{t + \sigma_1^2},\hdots,\frac{1}{t+\sigma_k^2}, \frac{1}{t},\hdots,\frac{1}{t}\right) V^T (\id_{d_x} - V P_k V^T) \xi_j^T(W) U \right), \displaybreak[2]\\
\tr(Q E_2) & =  \tr\left(\diag\left( \frac{\sigma_1}{t+\sigma_1^2}, \hdots, \frac{\sigma_k}{t + \sigma_k^2}, 0, \hdots,0\right) V^T Y^T U \diag\left(\frac{1}{t+\sigma_1^2}, \hdots,\frac{1}{t+\sigma_k^2},0,\hdots,0 \right)  \right.\\
& \phantom{\tr(}  \left. \times  U^T \xi_i(W) (\id_{d_x} - V P_k V^T) V \diag\left(\frac{1}{t + \sigma_1^2},\hdots,\frac{1}{t+\sigma_k^2}, 0,\hdots,0\right) V^T \xi_j^T(W) U \right),\displaybreak[2] \\
\tr(Q E_3) & =  \tr\left(\diag\left( \frac{\sigma_1}{t+\sigma_1^2}, \hdots, \frac{\sigma_k}{t + \sigma_k^2}, 0, \hdots,0\right) V^T Y^T U \diag\left(\frac{1}{t+\sigma_1^2}, \hdots,\frac{1}{t+\sigma_k^2},\frac{1}{t},\hdots,\frac{1}{t} \right)  \right.\\
& \phantom{\tr(}  \left. \times U^T \xi_i(W) V \diag\left(\frac{1}{t + \sigma_1^2},\hdots,\frac{1}{t+\sigma_k^2}, 0,\hdots,0 \right) V^T (\id_{d_x} - V P_k V^T) \xi_j^T(W) U \right),\displaybreak[2] \\
\tr(Q E_4) & = \tr\left(\diag\left( \frac{\sigma_1}{t+\sigma_1^2}, \hdots, \frac{\sigma_k}{t + \sigma_k^2}, 0, \hdots,0\right) V^T Y^T U \diag\left(\frac{1}{t+\sigma_1^2}, \hdots,\frac{1}{t+\sigma_k^2},\frac{1}{t},\hdots,\frac{1}{t} \right)   \right.\\
& \phantom{\tr(}  \left. \times U^T \xi_i(W) V \diag\left(\frac{1}{t + \sigma_1^2},\hdots,\frac{1}{t+\sigma_k^2}, 0,\hdots,0 \right) V^T \xi_j^T(W) U\right).
\end{align*}
Note that $1/(t+\sigma_i^2) \leq \min\{ 1/t, 1/\sigma_i^2\}$ for all $t>0$.
Using Cauchy-Schwarz inequality for the Frobenius inner product, the fact that
$\|AB\|_F \leq \|A\|_{2\to 2} \|B\|_F$, and unitarity of $U$ and $V$, we obtain
\[
|\tr (Q E_\ell)| \leq \|Y\|_F \| \xi_i(W)\|_F \|\xi_j(W)\|_F \sigma_k^{-3} t^{-1}, \quad \ell = 1,2,3,4.
\]
By continuity of $\xi_i$ and $\xi_j$, it follows that there exists a neighborhood $\mathcal{U} \subset \mathcal{M}_k$ around a fixed $W_0 \in \mathcal{M}_k$ (in which $\sigma_k(W) \geq c >0$ for some $c> 0$ and all $W \in \mathcal{U}$) and a constant $C> 0$ (depending only on the neighborhood) such that
\[
|\tr(Q \xi_i(W) (t + W^T W)^{-1} \xi_j^T(W) )| \leq C \|Y\|_F t^{-1} \quad \mbox{ for all } t > 0 \mbox{ and } W \in \mathcal{U}.
\]
In the same way, one shows the above inequality for $Q$ replaced by $Q^T$ and hence
\[
|\tr((D \phi_t(W)[Y] \xi_i(W) \psi_t(W) \xi_j^T(W))|
\leq 2C \|Y\|_F t^{-1} \quad \mbox{ for all } t > 0 \mbox{ and } W \in \mathcal{U}.
\]
Moreover, by the Cauchy-Schwarz inequality for the Frobenius inner product and since $W W^T$ as well as $W^T W$ are positive semidefinite, it holds
\begin{align}
  & |\tr((D \phi_t(W)[Y] \xi_i(W) \psi_t(W) \xi_j^T(W))| \notag 
\\
  & \quad
    \leq \| (t+WW^T)^{-1}(W Y^T + Y W^T) (t+WW^T)^{-1} \|_F 
      \| \xi_i(W) (t + W^T W)^{-1} \xi_j^T(W) \|_F\notag \displaybreak[2]
\\
  & \quad
    \leq \|(t+WW^T)^{-1}\|_{2 \to 2}^2 \|(t + W^T W)^{-1}\|_{2 \to 2} 
      \|W Y^T + Y W^T \|_F \|\xi_i(W)\|_F \|\xi_j(W)\|_F \notag \displaybreak[2]
\\
  & \quad
    \leq 2 t^{-3}  \|W Y^T\|_F \|\xi_i(W)\|_F \|\xi_j(W)\|_F. 
\label{B1large-t}
\end{align}
Altogether, it holds, for a suitable constant $C_1> 0$,
\[
  |\tr((D \phi_t(W)[Y] \xi_i(W) \psi_t(W) \xi_j^T(W))| \leq C_1 \|Y\|_F \min\{ t^{-1}, t^{-3} \} \quad \mbox{ for all } t > 0 \mbox{ and } W \in \mathcal{U}.
\]
Let us now consider the second term on the right hand side of \eqref{DK1}. As in \eqref{B1large-t}, we obtain
\[
| \tr(\phi_t(W) D\xi_i(W)[Y] \psi_t(W) \xi_j^T(W))| \leq t^{-2} \|D\xi_i(W)[Y]\|_F \| \xi_j^T(W) \|_F.
\]
By Lemma~\ref{lem:proj-tangent} we can write $\xi_j(W) = P_W(\xi_j(W)) = \zeta^1_j(W) + \zeta^2_j(W)$ with
$\zeta^1_j(W) = U P_k U^T \xi_j(W) (\id_{d_x} - V P_k V^T)$ and $\zeta^2_j(W) = \xi_j(W) V P_k V^T$ as in \eqref{def:zeta}. With $E = \id_{d_x} - V P_k V^T$
this gives
\begin{align*}
& \tr(\phi_t(W) D\xi_i(W)[Y] \psi_t(W) \xi_j^T(W)) \\
 & = \tr\left( (t+ WW^T)^{-1} D \xi_i(W)[Y](t+ W^TW)^{-1} 
 (\zeta^1_j(W) + \zeta^2_j(W))^T \right) \displaybreak[2]\\
 &= \tr\left( \diag\left(\frac{1}{t+ \sigma_1^2},\hdots,\frac{1}{t+ \sigma_k^2},0,\hdots,0\right) U^T D\xi_i(W)[Y] V \diag\left(\frac{1}{t+ \sigma_1^2},\hdots,\frac{1}{t+ \sigma_k^2},\frac{1}{t},\hdots,\frac{1}{t}\right) V^T E \xi_j^T(W)  U \right) \displaybreak[2] \\
 & + \tr\left( \diag\left(\frac{1}{t+ \sigma_1^2},\hdots,\frac{1}{t+ \sigma_k^2},\frac{1}{t},\hdots,\frac{1}{t}\right) U^T D\xi_i(W)[Y] V \diag\left(\frac{1}{t+ \sigma_1^2},\hdots,\frac{1}{t+ \sigma_k^2},0,\hdots,0 \right) V^T \xi_j^T(W) U \right).
\end{align*}
By the Cauchy-Schwarz inequality for the Frobenius inner product it follows that
\[
| \tr(\phi_t(W) D\xi_i(W)[Y] \psi_t(W) \xi_j^T(W)) | \leq 2 \| D \xi_i(W)[Y] \|_F \|\xi_j(W)\|_F \sigma_k^{-2} t^{-1}.
\]
Since the $\xi_i$ are continuously differentiable
there exists $C_2>0$ such that
\[
|\tr(\phi_t(W) D\xi_i(W)[Y] \psi_t(W) \xi_j^T(W))| \leq C_2 \|Y\|_F \min\{ t^{-1}, t^{-2}\} \quad \mbox{ for all } t > 0 \mbox{ and } W \in \mathcal{U}.
\]
The terms in \eqref{DK2} can be bounded in the same way as the ones in \eqref{DK1}, hence
\[
|D K_{i,j,t}(W)[Y]| \leq C' \|Y\|_F \min\{t^{-1},t^{-2}\} \quad \mbox{ for all } t > 0 \mbox{ and } W \in \mathcal{U}. 
\]
It follows that
$
\int_0^\infty |D K_{i,j,t}(W)[Y]| t^{1/N} dt
$
exists and is uniformly bounded in $W \in \mathcal{U}$. By Lebesgue's dominated convergence
theorem, it follows that we can interchange integration and differentiation, and hence, all directional derivatives of $F_{i,j}$ at $W$ in the direction of $Y$ exist and are continuous 
since $D K_{i,j,t}(W)[Y]$ is continuous in $W$ for all $Y \in T_W(\mathcal{M}_k)$. Hence, by \eqref{metric_F1}, $F_{i,j}$ is (totally) continuously differentiable for all $i,j$ with
\[
D F_{i,j}(W)[Y] = \frac{\sin(\pi/N)}{\pi} \int_0^\infty D K_{i,j,t}(W)[Y] t^{1/N} dt.
\]
Hence, the metric is of class $C^1$ as claimed.  

 \section{Some results on flows on manifolds}
\label{appendix:flows}

Here we summarize some notions and results on flows on manifolds that can be found in \cite[Chapter IV,§2]{Lang99} to which we also refer for more details.

Let $p\geq 2$ be an integer or $p=\infty$, let $\mathcal M$ be a (finite dimensional) $C^p$-manifold and let $\xi$ be a vector field of class $C^{p-1}$ on $\mathcal M$.
An \emph{integral curve} for $\xi$ with initial condition $x_0\in \mathcal M$ is a $C^{p-1}$-curve
$$\gamma: J \to \mathcal M,$$
where $J$ is an open intervall containing $0$ such that $$
\dot \gamma(t)=\xi(\gamma(t))\ \ \forall t\in J \ \text{ and } \gamma(0)=x_0.
$$
\begin{theorem}[Theorem 2.1 in Chapter IV,§2 in \cite{Lang99}]\label{uniqueness}
	If $\gamma_1: J_1 \to \mathcal M$ and $\gamma_2: J_2 \to \mathcal M$ are integral curves for $\xi$ with the same initial condition then $\gamma_1=\gamma_2$ on $J_1\cap J_2$.
\end{theorem}

Let $D(\xi)\subseteq \RR\times \mathcal M$ be the set of all pairs $(t,x_0)$ such that for any $x_0\in \mathcal M$ the set $$J(x_0):=\{t\in \RR \ | \ (t,x_0)\in D(\xi)\}$$ is the maximal open existence interval  of an integral curve for $\xi$ with initial condition $x_0$. For any $x_0\in\mathcal M$, this interval is non-empty (locally one can argue as in the case $\mathcal M=\RR^n$).

A \emph{global flow} for $\xi$ is a mapping $$
\alpha: D(\xi)\to \mathcal M
$$ such that for all $x_0\in \mathcal M$ the map $t\mapsto \alpha(t,x_0)$ for $ t\in J(x_0)$ is an integral curve for $\xi$ with initial condition $x_0$, i.e. $\alpha(0,x_0)=x_0$ and $\dot \alpha(t,x_0)=\xi(\alpha(t,x_0))$ for any $t\in J(x_0)$. Note that by Theorem \ref{uniqueness} there is only one such mapping  $\alpha$.
For $t\in \RR$ let $$D_t(\xi):=\{x\in \mathcal M \ | \ (t,x)\in D(\xi)\}$$
and define the map $\alpha_t\colon D_t(\xi)\to \mathcal{M}$ by $\alpha_t(x)= \alpha(t,x)$.

\begin{theorem}[Theorems 2.6 and 2.9 in Chapter IV,§2 in \cite{Lang99}]\label{flowprop}
$\hbox{}$
  \begin{enumerate}
		\item The set $D(\xi)$ is open  in $\RR\times \mathcal M $  and $\alpha$ is a $C^{p-1}$-morphism.
		\item For any $t\in \RR$, the set  $D_t(\xi)$ is open in $\mathcal M$ and (for $D_t(\xi)$ non-empty) $\alpha_t$ defines a diffeomorphism of $D_t(\xi)$ onto an open subset of $\mathcal M$ (namely $\alpha_t(D_t(\xi))=D_{-t}(\xi))$ and $\alpha_t^{-1}=\alpha_{-t}$).
	\end{enumerate}
\end{theorem}

\section{The non-symmetric autoencoder case for $N=2$}
\label{appendix:nonsymmetric}

Here we consider  the optimization problem (\ref{eqNNMinimization}) with $N=2$ and the  additional constraint that $Y=X$, but we do not assume that  $W_2=W_1^T$. We also assume balanced starting conditions, i.e., $W_2(0)^TW_2(0)=W_1(0)W_1(0)^T$.
In this special case, we can use a more direct approach
than in Section~\ref{sec:saddle-points} to establish some additional explicit statements below.

We write again $d$ for $d_x=d_y$ and  $r$ for $d_1$.
The equations for the flow here are:
\begin{align}\label{flow2layers}
\begin{split}
\dot{W_1} &= -W_2^TW_2W_1XX^T+W_2^T XX^T,\\
\dot{W_2} &= -W_2W_1XX^TW_1^T+XX^TW_1^T.
\end{split}
\end{align}

Next we analyze the equilibrium points of the flow (\ref{flow2layers}) and of the product $W=W_2W_1$ again assuming balanced initial conditions.
We begin by exploring the equilibrium points of the flow
(\ref{flow2layers}) by setting the expressions in (\ref{flow2layers}) equal to zero:
\begin{align}\label{equilib}
\begin{split}
-W_2^TW_2W_1XX^T+W_2^T XX^T &= 0,\\
-W_2W_1XX^TW_1^T+XX^TW_1^T &= 0.
\end{split}
\end{align}
If $W_2\in\RR^{d\times r}$ is the zero matrix then (since $XX^T$ has full rank) it  follows that
(\ref{equilib}) is solved if and only if  $W_1$ is the $r\times d$ zero-matrix, hence $W$ is the $d\times d$ zero-matrix.
The following lemma characterizes the non-trivial solutions. 
(The second part of the lemma is a special case of Proposition~\ref{critpoints_general} below.)
\begin{lemma}\label{equilpoints}
	The balanced  nonzero solutions (i.e. solutions with $W_2\neq 0$)  of (\ref{equilib}) are precisely the matrices of the form 
	\begin{align}
	W_2 =UV^T,  \quad 
	W_1=W_2^T=VU^T, \quad 
	W =W_2W_1=UU^T, \label{equi3}
	\end{align}
	where
		$U\in \RR^{d\times k}$ for some $1 \leq k\leq r$ and where the columns of $U$ are orthonormal eigenvectors of $XX^T$ and 
		 	$V\in \RR^{r\times k}$ has orthonormal columns.

In particular, the equilibrium points for $W=W_2W_1$ are precisely the matrices of the form 
\begin{equation}\label{equi-points}
W =\sum_{j=1}^k u_ju_j^T,
\end{equation}
where $k\in \{1,\hdots,r\}$
and $u_1,\hdots, u_k$ (the columns of $U$ above) are orthonormal eigenvectors of $XX^T$. 
%
\end{lemma}
\begin{remark} Note that $W=0$ is also an equilibrium point of 
(\ref{flow2layers}) 
which formally corresponds to taking $k=0$ in \eqref{equi-points}.
\end{remark}
\begin{proof}
	Since $W_2\neq 0$,  the rank $k$ of  $W_2$ is at least  $1$.
The balancedness condition $W_2^TW_2=W_1W_1^T$ implies that $W_1$ and $W_2$ have the same singular values.
Since $XX^T$ has full rank, the first equation of (\ref{equilib}) yields $W_2^T=W_2^TW_2W_1$. 
Again due to balancedness, this shows that $W_2^T=W_1W_1^TW_1$. It follows that all positive singular values of $W_1$ and of $W_2$ are equal to $1$ and that $W_2=W_1^T$. 
The second  equation of (\ref{equilib}) thus gives  the equation 
\begin{equation}\label{equilibsimple}
(I_d-W_2W_2^T)XX^TW_2=0. 
\end{equation}
(The equilibrium points of full rank $r$ could now be obtained using   \cite[Propositon 4.1]{yahemo94}  again, but we are interested in all solutions here.) 
Since the positive singular values of $W_2$ are all equal to $1$, it follows that we can write
$$
W_2W_2^T=\sum_{i=1}^k u_iu_i^T,
$$
where the $u_i$ are orthonormal. We extend the system $u_1,\hdots, u_k$ to an orthonormal basis $u_1,\hdots, u_d$
of $\RR^d$.
From (\ref{equilibsimple}) we obtain
$(I_d-W_2W_2^T)XX^TW_2W_2^T=0$, hence
$
\sum_{j=k+1}^d u_ju_j^T XX^T \sum_{i=1}^k u_iu_i^T=0
$
and consequently 
$$
\sum_{j=k+1}^d  \sum_{i=1}^k (u_j^T XX^Tu_i)u_ju_i^T=0.
$$
It follows that for all $j\in\{k+1,\hdots,d\}$ and for all  $i\in\{1,\hdots,k\}$ we have $u_j^T XX^Tu_i=0$.
This in turn implies that $XX^T$ maps the span of $u_1,\hdots, u_k$ into itself and also maps the  span of $u_{k+1},\hdots, u_d$ into itself.
This implies that we can choose    $u_1,\hdots, u_d$ 
as orthonormal eigenvectors of $XX^T$.
Thus we can indeed write the (reduced) singular value decomposition of $W_2$ as $W_2=UV^T$, where the columns 
 $u_1,\hdots, u_k$ of $U$  are orthonormal eigenvectors of $XX^T$ and where $V$ is  as in the statement of the lemma. Since $W_1=W_2^T$ and $W=W_2W_1$, it follows that $W_1=VU^T$ and $W=UU^T$ as claimed. Altogether, we have shown that the equations 
 \eqref{equi3} are necessary for having a balanced solution of \eqref{equilib}.
One easily checks that $W_1,W_2$ defined by 
\eqref{equi3} also satisfy 
\eqref{equilib}, which shows sufficiency. This completes the proof.
\end{proof}

\begin{corollary}
	Consider a linear autoencoder with one hidden layer of size $r$ with balanced initial conditions and assume that $XX^T$ has eigenvalues $\lambda_1\geq \hdots \geq  \lambda_d>0$ and corresponding orthonormal eigenvectors $u_1,\hdots, u_d$.
	 \begin{enumerate}
	\item The flow  $W(t)$ always converges to an equilibrium point of the form $	W=\sum_{j\in J_W} u_ju_j^T,$ where $J_W$ is a (possibly empty) subset of $\{1,\hdots, d\}$ of at most $r$ elements. 
	\item The flow $W_2(t)$ converges to $UV^T=:W_2$, where  the columns of $U$ are the $u_j, j\in J_W$ and $V\in\RR^{r\times k} $ has orthonormal columns ($k=|J_W|$) . Furthermore $W_1(t)$ converges to $W_2^T$.
	\item If $L^1(W(0))< \frac 1 2\sum_{i=r, i\neq r+1}^d \lambda_i$ then $W(t)$ converges to the optimal equilibrium $	W=\sum_{j=1}^r u_ju_j^T$.
\item If $\lambda_r>\lambda_{r+1}$, then there is an open neighbourhood of the optimal equilibrium point in  which we have convergence of the flow $W(t)$ to the optimal equilibrium point.
\end{enumerate}
\end{corollary} 

\begin{proof}
	The first and the second point follow from Lemma \ref{equilpoints} together with Theorem \ref{globconv}. (Note that if $(W_1, W_2)$ is an equilibrium point to which the flow converges then $W_1,W_2$ are  balanced since we assume that the flow has balanced initial conditions.) 
	To prove the third point, note that the loss of an  equilibrium point $	W=\sum_{j\in J_W} u_ju_j^T$ is given by $L^1(W)=\frac 1 2\sum_{i\in K_W}\lambda_i$, where $K_W=\{1,\hdots, d\}\setminus J_W$. This sum is minimal for $J_W=\{1,\hdots, r\}$. Among the remaining possible $J_W$, the value of   $L^1(W)$ is minimal for $J_W=\{1,\hdots, r+1\}\setminus \{r\}$, i.e., $K_W=\{r,\hdots, d\}\setminus \{r+1\}$. Since the value of $L^1(W(t))$  monotonically decreases as $t$ increases (as follows e.g. from equation (\ref{riemflow})),  the claim now follows from the first point. 
	The last point follows from the third point.
\end{proof}

The following result is an analogue to Theorem~\ref{nonstable}.
\begin{theorem}\label{saddlepoints}
If $k\leq r$ and $u_1,\hdots, u_k$ are orthonormal eigenvectors of $XX^T$ which do  not form a system of  eigenvectors to the $r$ largest eigenvalues of $XX^T$ (in particular for $k<r$), in any  neighborhood of the equilibrium point $W =\sum_{j=1}^k u_ju_j^T$ there is some $\widetilde W$ of rank at most $r$ for which $L^1(\widetilde{W})<L^1(W)$. In particular, the equilibrium in $W$ is non-stable.
\end{theorem}
\begin{proof}
	If $k<r$ and $W =\sum_{j=1}^k u_ju_j^T$ for orthonormal eigenvectors $u_j$ of $XX^T$ then for any additional eigenvector $u_{k+1}$ orthonormal to the $u_j$ and for any $\varepsilon >0 $, we can choose $\widetilde{W}=W+\varepsilon u_{k+1}u_{k+1}^T$ to obtain $L^1(\widetilde{W})<L^1(W)$.
	Let now $k=r$. This case can be treated analogously to the proof of Theorem \ref{nonstable}:
	let $u_i$ be one of the eigenvectors  $u_1,\hdots, u_r$  whose eigenvalue does not belong to the $r$ largest eigenvalues of $XX^T$. Let $v$ be an eigenvector of $XX^T$ of unit length which is orthogonal to the eigenvectors   $u_1,\hdots, u_r$  and whose eigenvalue   belongs to the $r$ largest eigenvalues of $XX^T$. Now for any $\varepsilon\in [0,1]$ consider $u_i(\varepsilon):=\varepsilon v +\sqrt{1-\varepsilon^2}u_{i}$. Then $W(\varepsilon):=u_i(\varepsilon)u_i(\varepsilon)^T+ \sum_{j=1,j\neq i}^r u_ju_j^T$ satisfies $L^1(W(\varepsilon))<L^1(W)$ for $\varepsilon\in (0,1]$. From this the claim follows.
\end{proof}

\begin{remark}
	With the notation $$
	V=\begin{pmatrix}
	W_1^T\\W_2
	\end{pmatrix}\in \RR^{2d\times r}  \text{ and } C=XX^T
	\in \RR^{d\times d} $$
and assuming that $C$ has full rank, the flow (\ref{flow2layers}) can be written as the following Riccati-type-like ODE.
\begin{equation}\label{riccati}
\dot{V}=\left(I_{2d}+\begin{pmatrix}
-C&0\\
0&0
\end{pmatrix}VV^T
\begin{pmatrix}
0&0\\
C^{-1}&0
\end{pmatrix}
+
\begin{pmatrix}
0&0\\
0&-I_d
\end{pmatrix}VV^T
\begin{pmatrix}
0&I_d\\
0&0
\end{pmatrix}
\right)
\begin{pmatrix}
0&C\\
C&0
\end{pmatrix}V.
\end{equation}
\end{remark}



\end{document}